\documentclass[a4paper,11pt]{amsart}

\usepackage{amssymb,amsmath,amsthm,mathrsfs,enumerate,graphicx, color}
\usepackage[pdfpagelabels,colorlinks,linkcolor=blue,citecolor=black,urlcolor=blue]{hyperref}
\usepackage{esint}
\usepackage{tikz}
\usetikzlibrary{arrows}
\usepackage{citeref}
\usepackage{mathtools}
\mathtoolsset{showonlyrefs}

\newtheorem{thm}{Theorem}[section]

\newtheorem{cor}[thm]{Corollary}
\newtheorem{lem}[thm]{Lemma}
\newtheorem{prop}[thm]{Proposition}
\newtheorem{defn}[thm]{Definition}
\newtheorem{rem}[thm]{Remark}

\newcommand{\lesi}{\lesssim}

\newcommand{\dx}{d\mu(x)}
\newcommand{\dy}{d\mu(y)}
\newcommand{\dz}{d\mu(z)}

\newcommand{\supp}{\operatorname{supp}}
\newcommand{\f}{\frac}

\newcommand{\vc}{\infty}

\title[Bilinear and Fractional Leibniz Rules Beyond Euclidean Spaces]{Bilinear and Fractional Leibniz Rules Beyond Euclidean Spaces: Weighted Besov and Triebel--Lizorkin Estimates}         % Enter your title between curly braces

\author[T. A. Bui]{The Anh Bui}
\address{School of Mathematical and Physical Sciences, Macquarie University, NSW 2109,
	Australia}
\email{the.bui@mq.edu.au}


\begin{document}
	
	\begin{abstract}
	We establish fractional Leibniz rules in weighted settings for nonnegative self-adjoint operators on spaces of homogeneous type. Using a unified method that avoids Fourier transforms, we prove bilinear estimates for spectral multipliers on weighted Hardy, Besov, and Triebel–Lizorkin spaces. Our approach is flexible and applies beyond the Euclidean setting—covering, for instance, nilpotent Lie groups, Grushin operators, and Hermite expansions—thus extending classical Kato–Ponce inequalities. The framework also yields new weighted bilinear estimates including fractional Leibniz rules for Hermite, Laguerre, and Bessel operators, with applications to scattering formulas and related PDE models.
		
	\end{abstract}
	\date{}

	\maketitle
	
	\tableofcontents
	
	\section{Introduction}\label{sec: intro}
	
	\subsection{State of the art}
	
	The fractional Leibniz rules, also known as Kato--Ponce inequalities, describe how fractional differentiation interacts with products of functions. More precisely, for $f,g \in \mathscr{S}(\mathbb{R}^n)$ one has inequalities of the form
	\[
	\|D^s(fg)\|_{L^p} \lesssim \|D^s f\|_{L^{p_1}} \|g\|_{L^{p_2}} + \|f\|_{L^{p_1}} \|D^s g\|_{L^{p_2}},
	\]
	and analogues with the Bessel potential operator $J^s$. Here $D^s f(\xi) = |\xi|^s \widehat{f}(\xi)$ and $J^s f(\xi) = (1+|\xi|^2)^{s/2} \widehat{f}(\xi)$, with restrictions on the exponents depending on $s$, the dimension, and the integrability indices. Such inequalities play a central role in the analysis of nonlinear PDEs, including the Euler and Navier--Stokes equations \cite{KatoPonce}, dispersive models such as the Korteweg--de Vries equation \cite{ChristWeinstein,KenigPonceVega}, and the smoothing effects of Schr\"odinger semigroups \cite{GulisashviliKon}. More recent advances address endpoint cases and nonclassical ranges of exponents; see for example \cite{GrafakosOh,MuscaluSchlag,BourgainLi}.
	
	Beyond simple products, one can replace $fg$ by bilinear pseudodifferential operators
	\[
	T_\sigma(f,g)(x) = \int_{\mathbb{R}^{2n}} e^{2\pi i x \cdot (\xi+\eta)} \, \sigma(x,\xi,\eta) \, \widehat{f}(\xi) \, \widehat{g}(\eta) \, d\xi \, d\eta,
	\]
	with symbol $\sigma$. Such operators include bilinear multipliers, and their mapping properties have been investigated in a wide variety of function spaces. Examples include results for homogeneous H\"ormander classes via molecular decompositions \cite{BrummerNaibo}, estimates in Lebesgue and mixed-norm spaces under minimal smoothness conditions \cite{HartTorresWu}, and boundedness results in Besov, Sobolev, Triebel--Lizorkin, and related scales \cite{NaiboThomson,BenyiTorres}.
	
	%Weighted versions of fractional Leibniz rules have also been established. For instance, \cite{CruzUribeNaibo} obtained weighted inequalities in Lebesgue and variable-exponent spaces, while \cite{BrummerNaiboWeighted} treated bilinear Coifman--Meyer multipliers in weighted settings, including biparameter situations. More recently, \cite{NaiboThomson} developed Leibniz-type rules in weighted Triebel–Lizorkin and Besov spaces built upon weighted Lebesgue, Lorentz, Morrey, and variable-exponent spaces, extending and complementing the earlier works \cite{CruzUribeNaibo,BrummerNaiboWeighted}.
	
	Classical approaches to the fractional Leibniz rule rely heavily on the Fourier transform, restricting their applicability to Euclidean frameworks. However, many modern problems involve non-Euclidean  such as sub-Laplacians on Lie groups or Schr\"odinger operators with potentials, where  methods based on the Fourier transform are no longer effective. Extending Leibniz rules beyond the Euclidean setting has therefore become a central and challenging theme in harmonic analysis and PDEs. For instance, fractional Leibniz rules have been investigated on Lie groups in the scale of $L^p$ and Sobolev spaces \cite{CRT,BBR}. Estimates in Besov and Triebel–Lizorkin spaces for nilpotent Lie groups and the Grushin operator setting were obtained in \cite{Bruno}, while fractional Leibniz rules on stratified Lie groups were studied in \cite{Fang et al}. In \cite{NaiboLy}, the fractional rules associated with Hermite operators were considered in both $L^\infty$ and Besov/Triebel–Lizorkin spaces. On spaces of homogeneous type, fractional Leibniz rules in the Hardy scale have been derived for restricted ranges of fractional orders \cite{LZ}. Nevertheless, existing results often impose limitations on the smoothness parameter $s$ and the integrability and summability parameters $p$ and $q$ in the Besov and Triebel-Lizorkin spaces.
	
	In this work, we establish a unified framework for  fractional Leibniz rules on spaces of homogeneous type, associated with nonnegative self-adjoint operators, in the general scale of new weighted Besov and Triebel-Lizorkin spaces. Our results apply to not only obtain new results, but also significantly improve existing results in settings including the Laplacian on $\mathbb R^n$, the sub-Laplacian on Lie groups, Schr\"odinger operators on stratified groups, the Grushin operators on $\mathbb R^{n+1}$ and the Hermite operator on $\mathbb R^n$. This yields a robust and flexible method for analyzing bilinear forms in settings where the classical Fourier analysis is unavailable.
	
	The novelty of our approach lies in completely bypassing Fourier methods, instead combining bilinear spectral multiplier theory with heat semigroup techniques.  To our knowledge, this constitutes the first comprehensive and unified treatment of fractional Leibniz rules in  non-Euclidean framework.

\subsection{Main results}

Let $X$ be a space of homogeneous type with a distance $d$ and a nonnegative Borel measure $\mu$ satisfying the doubling property
\begin{equation}\label{doubling1}
	V(x,2r) \le C\, V(x,r), \qquad x\in X, \ r>0,
\end{equation}
where $V(x,r)=\mu(B(x,r))$ and $B(x,r)=\{y\in X: d(x,y)<r\}$.  
This implies the following estimates
\begin{equation}\label{doubling2}
	V(x,\lambda r)\leq c_\star\, \lambda^n V(x,r), \qquad \lambda \ge 1,
\end{equation}
and 
\begin{equation}\label{doubling3}
	V(x,r)\le C \Big(1+\frac{d(x,y)}{r}\Big)^n V(y,r), \qquad x,y\in X, \ r>0.
\end{equation}
%These geometric properties of the measure form the basis for the analysis of operators and function spaces on $X$.

We consider a nonnegative self-adjoint operator $L$ on $L^2(X)$ generating a semigroup $\{e^{-tL}\}_{t>0}$ with kernel $p_t(x,y)$, subject to the following conditions:

\begin{enumerate}
	\item[(A1)] There exist constants $C,c>0$ such that
	\[
	|p_t(x,y)| \le \frac{C}{V(x,\sqrt t)} \exp\Big(-\frac{d(x,y)^2}{ct}\Big),
	\qquad x,y \in X, \ t>0.
	\]
	
	\item[(A2)] There exists $\delta_0 \in (0,1]$ such that
	\[
	|p_t(x,y)-p_t(\overline x,y)| \le \Big(\frac{d(x,\overline x)}{\sqrt t}\Big)^{\delta_0}
	\frac{C}{V(x,\sqrt t)} \exp\Big(-\frac{d(x,y)^2}{ct}\Big)
	\]
	for all $t>0$ and $x,\overline x,y\in X$ with $d(x,\overline x)<\sqrt t$.
\end{enumerate}

Denote by \(E_L(\lambda)\) the spectral (projection-valued) resolution of the nonnegative self-adjoint operator \(L\). By the spectral theorem, for any bounded Borel function \(F:[0,\infty)\to\mathbb C\) the operator
\[
F(L)=\int_0^\infty F(\lambda)\,dE_L(\lambda)
\]
is bounded on \(L^2(X)\). In what follows, for a smooth function \(F\) we denote by \(F(L)(x,y)\) the integral kernel of \(F(L)\) (see Proposition \ref{thm-kernel estimate for functional calculus} for the existence of this kernel). We make the following assumption.

\begin{enumerate}
	\item[(A3)] For all even functions \(\varphi,\psi\in\mathscr S(\mathbb R)\), all integers \(k\ge0\), and all \(N>0\), there exists a constant \(C_{N,k,\varphi,\psi}>0\) such that, for every \(t>0\) and every \(x,y,z\in X\),
	\[
	\Big|L^k_x\big[\varphi(t\sqrt L)(x,y)\,\psi(t\sqrt L)(x,z)\big]\Big|
	\le C_{N,k,\varphi,\psi}\; t^{-2k}\,D_{t,N}(x,y)\,D_{t,N}(x,z),
	\]
	where \(L^k_x\) denotes the operator \(L^k\) acting on the \(x\)-variable, \(\mathscr S(\mathbb R)\) is the Schwartz space on \(\mathbb R\), and \(D_{t,N}(\cdot,\cdot)\) is defined in \eqref{eq- Dlambda N} below.
\end{enumerate}

\medskip
The assumptions (A1) and (A2) are known as the Gaussian upper bound and H\"older continuity. The condition (A3) may appear to be technical; however, in many situations it is simply a consequence of the Gaussian upper bounds for higher-order derivatives of the heat kernel (see Section 5 for more details). The condition (A3) is essential to obtain the full range for the smoothness index $s$ in the fractional Leibniz rules. See Theorem \ref{main thm} and Corollary \ref{cor-main} below. Without (A3), we only expect a limited range for $s$. See Theorem \ref{main thm - without A3} below. \textit{In this paper, we do not assume that the operator $L$ satisfies all conditions (A1)–(A3). Throughout, we always assume (A1) and (A2), and explicitly indicate whenever (A3) is required.}

Assumptions (A1) and (A2) allow us to define bilinear form in this general setting.  
Let $L_1 = L\otimes 1$ and $L_2= 1\otimes L$ be the tensor product operators acting on $L^2(X\times X)$ with product measure $\mu \times \mu$. By the spectral theorem, there exists a unique spectral decomposition $E$ such that for all Borel sets $A \subset \mathbb R^2$,
\[
E(A) \text{ is a projection on } L^2(X\times X), 
\quad
E(A_1 \times A_2) = E_{L_1}(A_1)\otimes E_{L_2}(A_2).
\]
Hence for any function $F : \mathbb{R}^2 \to \mathbb{C}$ one can define the operator   by the formula
\begin{equation}
	F(L_1,L_2) = \int_{\mathbb{R}^2} F(\lambda_1,\lambda_2)\, dE(\lambda_1,\lambda_2)
\end{equation}
with domain
\[
\mathcal D = \{g\in L^2(X\times X): \int_{\mathbb{R}^2} |F(\lambda_1,\lambda_2)|^2\, d\langle E(\lambda_1,\lambda_2)g,g \rangle<\vc \}.
\]
A straightforward variation of classical spectral theory arguments shows that for any bounded 
Borel function $F : \mathbb{R}^2 \to \mathbb{C}$ the operator $F(L_1,L_2)$ is continuous on 
$L^2(X_1 \times X_2)$ and its norm is bounded by $\|F\|_{\infty}$. See for example \cite{Sikora}.

Let $\boldsymbol{\mathrm{m}}:\mathbb R^2\to \mathbb C$ satisfy
\begin{equation}\label{eq-condition on m}
	\big|\partial_\xi^\alpha \partial_\eta^\beta \boldsymbol{\mathrm{m}}(\xi,\eta)\big|
	\le C_{\alpha,\beta}\,(|\xi|+|\eta|)^{\gamma-(|\alpha|+|\beta|)},
\end{equation}
for all $\alpha,\beta \in \mathbb N$ and some $\gamma \in \mathbb R$.  
We associate to $\boldsymbol{\mathrm{m}}$ the bilinear form
\begin{equation}\label{eq-bilinear form}
	B_{\boldsymbol{\mathrm{m}},L}(f,g)(x) = \boldsymbol{\mathrm{m}}(L_1,L_2)(f\otimes g)(x,x),
\end{equation}
for $f,g\in \mathcal S_\infty$ and $x\in X$, where $\mathcal S_\infty$ denotes the class of test functions associated with $L$. When $L=-\Delta$ is the Laplacian on $\mathbb R^n$, the class $\mathcal{S}_\vc$ is identical to the subspace of functions in $\mathscr S(\mathbb R^n)$ that have vanishing moments of all orders. See Subsection~\ref{sec-distribution}.  
In general, since the diagonal $\{(x,x):x\in X\}$ has zero measure in $(X\times X,\mu\times\mu)$, the formula \eqref{eq-bilinear form} need not be well-defined. However, under conditions (A1)--(A2) and within the framework of $\mathcal S_\infty$, identity \eqref{eq-bilinear form} is well-defined; see Subsection~\ref{Sec-Appendix} for a detailed proof.

\medskip

Finally, for $w \in A_\infty$ (see Subsection~2.1 for the definition of Muckenhoupt weights), set
\[
\tau_w = \inf\{\tau \in (1,\infty): w \in A_\tau\}.
\]
Given $0<p,q\le\infty$, define
\[
\tau_{p,q}(w) := n\Big(\frac{1}{\min(p/\tau_w,q,1)} - 1\Big),
\qquad
\tau_p(w) := n\Big(\frac{1}{\min(p/\tau_w,1)} - 1\Big).
\]
If $w=1$ (so that $\tau_w=1$), we simply write $\tau_{p,q}$ and $\tau_p$.  
Note that  $\tau_{p,q}(w)\ge \tau_{p,q}$, $\tau_p(w)\ge \tau_p$ for any $w \in A_\infty$.

We denote by $\dot{B}^{s,L}_{p,q}(w)$, $\dot{F}^{s,L}_{p,q}(w)$, and $H^p_L(w)$ the weighted Besov, Triebel--Lizorkin, and Hardy spaces associated with $L$, respectively, with the relevant parameters. For precise definitions, see Section~3.

\medskip

The main theorem of this paper is stated as follows.
	
\begin{thm}\label{main thm}
	Let $L$ satisfy conditions (A1), (A2), and (A3), and let $\boldsymbol{\mathrm{m}}$ satisfy \eqref{eq-condition on m} for some $\gamma \in \mathbb{R}$. Consider $0<q\le \vc$, $0 < p,p_1,p_2,p_3,p_4 \leq \infty$ with 
	\begin{equation}\label{eq- pi condition}
	\frac{1}{p} = \frac{1}{p_1} + \frac{1}{p_2} = \frac{1}{p_3} + \frac{1}{p_4},
	\end{equation}
	and weights $w_1,w_2,w_3,w_4 \in A_\infty$ such that 
	\begin{equation}\label{eq- wi condition}
	w_1^{p/p_1} w_2^{p/p_2} = w_3^{p/p_3} w_4^{p/p_4} =: w.
	\end{equation}
	If $0<p_1,p_4<\vc$, $0<p_2,p_3\le \vc$, and $s > \tau_{p,q}(w)$, then
	\begin{equation}\label{eq-main thm TL spaces}
		\|B_{\boldsymbol{\mathrm{m}},L}(f,g)\|_{\dot{F}^{s,L}_{p,q}(w)} 
		\lesssim \|f\|_{\dot{F}^{s+\gamma,L}_{p_1,q}(w_1)} \|g\|_{H^{p_2}_L(w_2)} 
		+ \|f\|_{H^{p_3}_L(w_3)} \|g\|_{\dot{F}^{s+\gamma,L}_{p_4,q}(w_4)} 
	\end{equation}
	for all $f,g \in \mathcal{S}_\infty$.
	
	\medskip
	\noindent If $0 < p,p_1,p_2,p_3,p_4 \leq \infty$ and $s > \tau_p(w)$, then
	\begin{equation}\label{eq-main thm B spaces}
		\|B_{\boldsymbol{\mathrm{m}},L}(f,g)\|_{\dot{B}^{s,L}_{p,q}(w)} 
		\lesssim \|f\|_{\dot{B}^{s+\gamma,L}_{p_1,q}(w_1)} \|g\|_{H^{p_2}_L(w_2)} 
		+ \|f\|_{H^{p_3}_L(w_3)} \|g\|_{\dot{B}^{s+\gamma,L}_{p_4,q}(w_4)} 
	\end{equation}
	for all $f,g \in \mathcal{S}_\infty$. 
	In both inequalities, $H^{p_i}_L(w_i)$ is replaced by $L^\infty$ if $p_i=\infty$ and $w_i=1$ for $i=2,3$.
\end{thm}
By choosing $\boldsymbol{\mathrm{m}}\equiv 1$ so that $T_{\boldsymbol{\mathrm{m}}}(f,g) = fg$, Theorem \ref{main thm} implies the following corollary:

\begin{cor}\label{cor-main}
	Let $L$ satisfy (A1)--(A3). Let $0<q\le \vc$, and let $0<p,p_1,p_2,p_3,p_4 \le \infty$ satisfy \eqref{eq- pi condition}, and $w, w_1, w_2, w_3, w_4 \in A_\infty$ satisfy \eqref{eq- wi condition}. Then we have:
	\begin{equation}
		\|L^{s/2}(fg)\|_{\dot{F}^{0,L}_{p,q}(w)}\simeq \|fg\|_{\dot{F}^{s,L}_{p,q}(w)} 
		\lesssim \|f\|_{\dot{F}^{s,L}_{p_1,q}(w_1)} \|g\|_{H^{p_2}_L(w_2)} 
		+ \|f\|_{H^{p_3}_L(w_3)} \|g\|_{\dot{F}^{s,L}_{p_4,q}(w_4)},
	\end{equation}
	whenever $0<p_1,p_4<\vc$, $0<p_2,p_3\le \vc$, and $s > \tau_{p,q}(w)$.  
	
	\medskip
	A similar estimate holds on $\dot{B}^{s,L}_{p,q}(w)$ provided $0<p_1,p_2,p_3,p_4\le \vc$ and $s > \tau_p(w)$.  
	Again, $H^{p_i}_L(w_i)$ is replaced by $L^\infty$ if $p_i=\infty$ and $w_i=1$ for $i=2,3$.
\end{cor}

Theorem \ref{main thm} and Corollary \ref{cor-main} extend the fractional Leibniz rules to a broad non-Euclidean framework through the bilinear multiplier estimates on weighted Besov, and Triebel–Lizorkin spaces associated with operators. They not only recover the classical Euclidean case but also significantly generalize previous results obtained in more restrictive settings such as nilpotent Lie groups or Hermite operators. Importantly, the bilinear multiplier formulation goes beyond the pointwise product, encompassing a large class of bilinear pseudodifferential operators. This unified framework thus provides new tools for operator-adapted harmonic analysis and for nonlinear PDEs on spaces of homogeneous type, where bilinear estimates are crucial. Finally, it emphasizes that our method is flexible and can be applied to obtain fractional Leibniz-type rules in a variety of function spaces, including Triebel–Lizorkin and Besov spaces based on Lorentz and Morrey spaces, as well as variable-exponent Lesbesgue spaces.

We record below a few important consequences of these results.  
\begin{enumerate}[(i)]
	\item Since $\dot F^{0,L}_{p,2}(w) =L^p(w)$ for $1<p<\vc$ and $w\in A_p$ (see \cite{BBD}), the result in Corollary \ref{cor-main} gives new estimates in the weighted $L^p$-spaces in the setting of spaces of homogeneous type.
	\medskip
	
	\item For $L=-\Delta$ on $\mathbb{R}^n$, it is well-known that the weighted Besov spaces $\dot{B}^{s,L}_{p,q}(w)$ and Triebel–Lizorkin spaces $\dot{F}^{s,L}{p,q}(w)$ coincide with the classical weighted Besov and Triebel–Lizorkin spaces (see, for example, \cite{BPT}). Hence, the estimates in Corollary \ref{cor-main} read as follows: for $0<q\le \vc$, $0<p,p_1,p_2,p_3,p_4 \le \infty$ satisfying \eqref{eq- pi condition}, and $w, w_1, w_2, w_3, w_4 \in A_\infty$ satisfying \eqref{eq- wi condition}, we have
		\begin{equation}
		\|D^s(fg)\|_{\dot{F}^{0}_{p,q}(w)}\simeq \|fg\|_{\dot{F}^{s}_{p,q}(w)} 
		\lesssim \|f\|_{\dot{F}^{s}_{p_1,q}(w_1)} \|g\|_{H^{p_2}(w_2)} 
		+ \|f\|_{H^{p_3}(w_3)} \|g\|_{\dot{F}^{s}_{p_4,q}(w_4)},
	\end{equation}
	whenever $0<p_1,p_4<\vc$, $0<p_2,p_3\le \vc$, and $s > \tau_{p,q}(w)$. A similar estimate holds on $\dot{B}^{s}_{p,q}(w)$ provided $0<p_1,p_2,p_3,p_4\le \vc$ and $s > \tau_p(w)$.  
	Again, $H^{p_i}(w_i)$ is replaced by $L^\infty$ if $p_i=\infty$ and $w_i=1$ for $i=2,3$. Here, $\dot{B}^{s}_{p,q}(w), \dot{F}^{s}_{p,q}(w)$ and $H^p(w)$ denote the (classical) weighted Besov, Triebel-Lizorkin and Hardy spaces. See for example \cite{BPT}. 
	
	The above estimates are new. Moreover, under the identifications $p_1=p_3$, $p_2=p_4$, $w_1=w_3$, and $w_2=w_4$, Corollary \ref{cor-main} recovers the main results of \cite{NaiboThomson}. Hence, Corollary \ref{cor-main} provides a genuine extension of the results in \cite{NaiboThomson}.
	\medskip
	\item For nilpotent Lie groups and the Grushin operator, our results extend the algebra properties of Besov and Triebel–Lizorkin spaces obtained in \cite{Bruno, Fang et al, BBR}, enlarging the admissible range of exponents and incorporating weights, with new consequences even in the unweighted case. See Sections 5.1 and 5.2.
	\medskip
	\item For Schr\"odinger operators with nonnegative polynomial potentials on stratified Lie groups, we obtain new fractional Leibniz rules, including new contributions in the Euclidean setting. See Section 5.1.
	\medskip
	\item For Hermite operators, we substantially extend the results of \cite{NaiboLy}, where only the special case $p_1=p_4$, $p_2=p_3=\infty$ was considered, again with notable improvements even without weights. See Section 5.3.
\end{enumerate}

When condition (A3) is dropped, the following still holds:

\begin{thm}\label{main thm - without A3}
	Let $L$ satisfy (A1) and (A2), and let $\boldsymbol{\mathrm{m}}$ satisfy \eqref{eq-condition on m} for some $\gamma\in \mathbb{R}$.  Let $0<q\le \vc$, and let $0<p,p_1,p_2,p_3,p_4 \le \infty$ satisfy \eqref{eq- pi condition}, and $w, w_1, w_2, w_3, w_4 \in A_\infty$ satisfy \eqref{eq- wi condition}. Then we have:
	
	If $0<p_1,p_4<\vc$, $0<p_2,p_3\le \vc$, and $\tau_{p,q}(w)<s<\delta_0$, then
	\[
	\|B_{\boldsymbol{\mathrm{m}},L}(f,g)\|_{\dot{F}^{s,L}_{p,q}(w)} 
	\lesssim \|f\|_{\dot{F}^{s+\gamma,L}_{p_1,q}(w_1)} \|g\|_{H^{p_2}_L(w_2)} 
	+ \|f\|_{H^{p_3}_L(w_3)} \|g\|_{\dot{F}^{s+\gamma,L}_{p_4,q}(w_4)},
	\]
	and an analogous inequality holds on $\dot{B}^{s,L}_{p,q}(w)$ for $0<p_1,p_2,p_3,p_4\le \vc$ and $\tau_p(w)<s<\delta_0$.  
	As before, $H^{p_i}_L(w_i)$ is replaced by $L^\infty$ if $p_i=\infty$ and $w_i=1$  for $i=2,3$.
\end{thm}

Compared with Theorem \ref{main thm}, the admissible range of smoothness is restricted: $\tau_{p,q}(w)<s<\delta_0$ in the Triebel–Lizorkin scale and $\tau_p(w)<s<\delta_0$ in the Besov scale. This version still yields fractional Leibniz rules analogous to Corollary \ref{cor-main}, but with this limited range. Examples of operators satisfying (A1)–(A2) include the Dirichlet Laplacian on Lipschitz or convex domains \cite{AR}, uniformly elliptic divergence-form operators \cite{A, Nash, CR}, Laplace–Beltrami and Schr\"odinger operators with certain potentials on manifolds with a doubling measure  and the Poincar\'e inequality \cite{Sa, BDK}, as well as multidimensional Bessel and Laguerre operators \cite{DPW, BDK}.

\medskip

\noindent\textbf{Proof strategy.}  
Since our framework extends beyond the Euclidean setting and lacks the Fourier transform, the standard techniques of \cite{NaiboThomson} are inapplicable. Even the definition of the bilinear form \eqref{eq-bilinear form} requires careful examinations. To address this, we employ the space of test functions introduced in \cite{PK} (see also \cite{BBD}) and define bilinear forms via multivariate spectral multipliers, combining this with singular integral theory and the recent development of function spaces beyond the classical Calderón–Zygmund setting \cite{DY, BBD, HLMMY, JY, YZ}.  

In comparison, the approach of \cite{LZ} to fractional Leibniz rules on spaces of homogeneous type assumes, besides (A1)–(A2), additional constraints such as the Markov property $e^{-tL}1=1$, reverse-doubling, and non-collapsing conditions on $\mu$. Moreover, their results apply only to Hardy spaces and in a narrow range of smoothness $s$. Crucially, their methods do not extend to Besov and Triebel–Lizorkin scales. By contrast, our approach is both flexible and robust, yielding a unified framework for bilinear estimates and fractional Leibniz rules beyond the Euclidean settings.

\subsection{Applications to scattering properties in PDEs} We now discuss an application to scattering properties of solutions to certain systems of partial differential equations.  We consider the system
\begin{equation}\label{eq:general-system}
	\begin{cases}
		u_t = v w,\\
		v_t + a(L)v = 0,\\
		w_t + b(L)w = 0,\\
		u(0,x) = 0, \quad v(0,x) = f(x), \quad w(0,x) = g(x),
	\end{cases}
\end{equation}
where $a(L)$ and $b(L)$ are general operators associated with a positive self-adjoint operator $L$ on $L^2(X)$.  The operators $a(L)$ and $b(L)$ cover a wide range of phenomena, from classical and fractional diffusion to more abstract evolutions on manifolds or Lie groups, capturing both local and nonlocal effects. Note that the particular case $L=-\Delta$ the Laplacian on $\mathbb R^n$ was investigated in \cite{NaiboThomson}.

%The quadratic term $u_t = v w$ introduces nonlinearity in an otherwise linear framework, modeling \emph{how two independent evolutions interact to generate new dynamics}. Understanding such interactions is crucial for analyzing regularity, growth, and decay properties in nonlinear PDEs, and it provides a natural setting for \emph{fractional Leibniz-type estimates}, which control derivatives of products of functions.

To make the problem tractable while retaining essential features, we focus on the symmetric case
\begin{equation}\label{eq:symmetric-case}
	a(L) = b(L) = L^\gamma, \quad \gamma > 0,
\end{equation}
reducing \eqref{eq:general-system} to
\begin{equation}\label{eq:symmetric-system}
	\begin{cases}
		u_t = v w,\\
		v_t + L^\gamma v = 0,\\
		w_t + L^\gamma w = 0,\\
		u(0,x) = 0, \quad v(0,x) = f(x), \quad w(0,x) = g(x).
	\end{cases}
\end{equation}

In this setting, the solutions to the linear equations can be explicitly written as
\[
v(t)=e^{-tL^\gamma}f,\qquad w(t)=e^{-tL^\gamma}g,
\]
so that
\[
u(t,x)=\int_0^t \big(e^{-sL^\gamma}f\big)(x)\,\big(e^{-sL^\gamma}g\big)(x)\,ds.
\]
Ignoring convergence issues, assume that $\lim_{t\to\infty}u(t,x)$ exists and denote
$u_\infty(x)=\lim_{t\to\infty}u(t,x)$. Then we may write
\[
\begin{aligned}
	u_\infty(x)
	&=\int_0^\infty \big(e^{-sL^\gamma}f\big)(x)\,\big(e^{-sL^\gamma}g\big)(x)\,ds\\[4pt]
	&=\int_0^\infty e^{-s\big(L^\gamma\otimes I+I\otimes L^\gamma\big)}\big[f\otimes g\big](x,x)\,ds\\[4pt]
	&=\big(L^\gamma\otimes I+I\otimes L^\gamma\big)^{-1}\big[f\otimes g\big](x,x)
	= B_{\mathbf m,L}\big[f\otimes g\big](x,x),
\end{aligned}
\]
where
\[
\mathbf m(\xi,\eta)=\frac{1}{\xi^\gamma+\eta^\gamma}.
\]
It is easy to check that the symbol $\mathbf m$ above satisfies \eqref{eq-condition on m} with $\gamma$ replaced by $-\gamma$. As a consequence of Theorem \ref{main thm}, we obtain scattering estimates for $u_\infty$ in the scale of weighted Besov and Triebel–Lizorkin spaces associated with $L$.

The paper is organized as follows. Section 2 begins with the definitions of Muckenhoupt weights and the Fefferman–Stein maximal inequality. We then establish several kernel estimates and recall certain functional calculus results. The section concludes with a review of distributions and Peetre maximal functions associated with the operator $L$. Section 3 introduces weighted Besov and Triebel–Lizorkin spaces related to $L$, together with their basic properties, including a new characterization of the weighted Hardy spaces associated with $L$. The main results are proved in Section 4. Applications to Theorem \ref{main thm} and Corollary \ref{cor-main} are presented in Section 5. Finally, Section 6 is devoted to the proofs of the bilinear formula \eqref{eq-bilinear form} and the auxiliary results stated in Sections 2 and 3.
\section{Preliminaries}	

\subsection{Muckenhoupt weights and Fefferman-Stein inequality}
To simplify notation, we will often use $B$ for $B(x_B, r_B)$.  
Given $\lambda > 0$, we denote by $\lambda B$ the $\lambda$-dilated ball, i.e. the ball centered at $x_B$ with radius $r_{\lambda B} = \lambda r_B$.

\medskip

A weight $w$ is a nonnegative, measurable, and locally integrable function on $X$.  
For any measurable set $E \subset X$, we define
\[
w(E) := \int_E w(x)\dx, \qquad V(E)=\mu(E).
\]
We also use
\[
\fint_E h(x)\dx=\f{1}{V(E)}\int_E h(x)\dx.
\]
For $1 \leq p \leq \infty$, let $p'$ denote the conjugate exponent of $p$, i.e. $1/p + 1/p' = 1$.

\medskip

We say that a weight $w \in A_p$, $1 < p < \infty$, if
\begin{equation}
	\label{defn Ap}
	[w]_{A_p}:=\sup_{B: {\rm balls}}\Big(\fint_B w(x)\dx\Big)^{1/p}\Big(\fint_B w(x)^{-1/(p-1)}\dx\Big)^{(p-1)/p}< \vc.
\end{equation}
For $p = 1$, we say that $w \in A_1$ if there exists $C>0$ such that for every ball $B \subset X$,
\[
\fint_B w(y)\dy \leq Cw(x) \quad \text{for a.e. } x\in B.
\]
We set $A_\vc=\cup_{p\geq 1}A_p$.

\medskip
%
%The reverse H\"older classes are defined as follows: a weight $w\in RH_q$, $1 < q < \infty$, if there exists $C>0$ such that for every ball $B \subset X$,
%\[
%\Big(\fint_B w(y)^q \dy\Big)^{1/q} \leq C \fint_B w\dx.
%\]
%In the endpoint case $q = \infty$, we say $w \in RH_\infty$ if there exists $C>0$ such that for any ball $B \subset X$,
%\[
%w(x)\leq C \fint_B w(y)\dy  \quad \text{for a.e. } x\in B.
%\]

%\medskip

For $w \in A_\vc$ and $0<p<\infty$, the weighted Lebesgue space $L^p(w)$ is defined by
\[
\Big\{f :\int_{X} |f(x)|^p w(x)\dx < \infty\Big\}
\]
with the norm
\[
\|f\|_{L^p(w)}=\Big(\int_{X} |f(x)|^p w(x)\dx\Big)^{1/p}.
\]

\medskip

%We recall some standard properties of weight classes (see \cite{ST}).

%%	The following hold:
%	\begin{enumerate}[{\rm (i)}]
%		\item $A_1\subset A_p\subset A_q$ for $1< p\leq q< \infty$.
%		\item If $w \in A_p$, $1 < p < \vc$, then there exists $1<r < p$ such that $w \in A_r$.
%		\item If $w\in A_p$, $1<p<\vc$, then $w^{1-p'}\in A_{p'}$.
%		\item If $w\in A_p$, $1\le p<\vc$, then there exists $C>0$ such that for any ball $B$ and measurable $E\subset B$,
%		\begin{equation}
%			\label{eq- doubling w}
%			\f{w(B)}{w(E)}\le C\Big(\f{V(B)}{V(E)}\Big)^p.
%		\end{equation}
%		\item If $w \in RH_p$, $1 < p < \vc$, then there exists $q>p$ such that $w \in RH_q$.
%		\item $\cup_{1<p<\vc} A_p \subset \cup_{1<q<\vc}RH_q$.
%	\end{enumerate}
%\end{lem}

For $w\in A_\vc$ we define
\[
\tau_w=\inf\{q: w\in A_q\}.%, \qquad r_w=\sup\{r: w\in RH_r\}.
\]

\bigskip

Let $w\in A_\vc$ and $0<r<\vc$. The weighted Hardy--Littlewood maximal function $\mathcal{M}_{r,w}$ is defined by
\[
\mathcal{M}_{r,w} f(x)=\sup_{x\in B}\Big(\f{1}{w(B)}\int_B|f(y)|^r w(y)\dy\Big)^{1/r},
\]
where the supremum is taken over all balls $B$ containing $x$.  
We drop the subscripts $r$ or $w$ when $r=1$ or $w\equiv 1$.

It is well known that if $w\in A_\vc$ and $0<r<\vc$, then
\begin{equation}
	\label{boundedness maximal function}
	\|\mathcal{M}_{r,w} f\|_{L^p(w)}\lesi \|f\|_{L^p(w)}
\end{equation}
for all $p>r$.

Moreover, for $0<r<\vc$ and $p>r$,
\begin{equation}
	\label{boundedness maximal function 2}
	\|\mathcal{M}_{r} f\|_{L^p(w)}\lesi \|f\|_{L^p(w)}
\end{equation}
holds for all $w\in A_{p/r}$. In particular,
\begin{equation}
	\label{eq-sharp weighted estimates for M}
	\|\mathcal{M} f\|_{L^p(w)}\le c_{n,p}^{\star} [w]_{A_p}^{\f{1}{p-1}} \|f\|_{L^p(w)}, 
	\quad w\in A_p, \ 1<p<\vc,
\end{equation}
where $c_{n,p}^{\star}$ depends only on $n$ and $p$.

\medskip

We recall the Fefferman-Stein vector-valued maximal inequality and a variant from \cite{GLY}. For $0<p<\vc$, $0<q\leq \vc$, $0<r<\min \{p,q\}$ and $w\in A_{p/r}$, one has for any sequence of measurable functions $\{f_\nu\}$,
\begin{equation}\label{FSIn}
	\Big\|\Big(\sum_{\nu}|\mathcal{M}_r f_\nu|^q\Big)^{1/q}\Big\|_{L^p(w)}
	\lesi 
	\Big\|\Big(\sum_{\nu}|f_\nu|^q\Big)^{1/q}\Big\|_{L^p(w)}.
\end{equation}

Finally, combining Young’s inequality with \eqref{FSIn} yields: if $\{a_\nu\} \in \ell^{q}\cap \ell^{1}$, then
\begin{equation}\label{YFSIn}
	\Big\|\sum_{j}\Big(\sum_\nu|a_{j-\nu}\mathcal{M}_r f_\nu|^q\Big)^{1/q}\Big\|_{L^p(w)}
	\lesi 
	\Big\|\Big(\sum_{\nu}|f_\nu|^q\Big)^{1/q}\Big\|_{L^p(w)}.
\end{equation}

\subsection{Dyadic cubes and some kernel estimates}

We will now recall  an important covering lemma in \cite{C}.
\begin{lem}\label{Christ'slemma} There
	exists a collection of open sets $\{Q_\tau^k\subset X: k\in
	\mathbb{Z}, \tau\in I_k\}$, where $I_k$ denotes certain (possibly
	finite) index set depending on $k$, and constants $\rho\in (0,1),
	a_{\diamond}\in (0,1]$ and $a_{\sharp}\in (0,\vc)$ such that
	\begin{enumerate}[(i)]
		\item $\mu(X\backslash \cup_\tau Q_\tau^k)=0$ for all $k\in
		\mathbb{Z}$;
		\item if $i\geq k$, then either $Q_\tau^i \subset Q_\beta^k$ or $Q_\tau^i \cap
		Q_\beta^k=\emptyset$;
		\item for every $(k,\tau)$ and each $i<k$, there exists a unique $\tau'$
		such $Q_\tau^k\subset Q_{\tau'}^i$;
		\item the diameter ${\rm diam}\,(Q_\tau^k)\leq a_{\sharp} \rho^k$;
		\item each $Q_\tau^k$ contains certain ball $B(x_{Q_\tau^k}, a_\diamond\rho^k)$.
	\end{enumerate}
\end{lem}
\begin{rem}\label{rem1-Christ}
	Since the constants $\rho$ and $a_\diamond$ are not essential in the paper, without loss of generality, we may assume that $\rho=a_\diamond=1/2$.  We then fix a collection of open sets in Lemma \ref{Christ'slemma} and denote this collection by $\mathcal{D}$. We call open sets in $\mathcal{D}$ the dyadic cubes in $X$ and $x_{Q_\tau^k}$ the center of the cube $Q_\tau^k \in \mathcal{D}$. Let $k_0$ is a fixed natural number such that 
	\begin{equation}
		\label{eq-constant k0}
		\max\{a_\sharp 2^{-k_0+1}, a_nb_\sharp c^2_\sharp c_\diamond c_{n,2}^\star  2^{-k_0+3}\}< 1,
	\end{equation}
	where $a_\sharp, a_n,b_\sharp$, $c_\sharp, c_\diamond$ and $ c_{n,2}^\star$ are constants in Lemmas \ref{Christ'slemma}, \ref{lem-elementary}, Remark \ref{rem1}, Lemmas \ref{lem-product exp}, \ref{lem-exp kernel} below, and \eqref{eq-sharp weighted estimates for M} for $p=2$, respectively. 
	
	Then we denote 
	$$\mathcal{D}_\nu:=\{Q_\tau^{\nu+k_0} \in \mathcal{D} : \tau\in I_{\nu+k_0}\}
	$$ 
	for each $\nu\in \mathbb{Z}$. Then for $Q\in \mathcal{D}_\nu$, we have $$B(x_Q,  2^{-(\nu+k_0+1)})\subset Q\subset B(x_Q, a_\sharp 2^{-{(\nu+k_0)}})\subset B(x_Q, 2^{-{(\nu+1)}}).
	$$
\end{rem}

The following is taken from \cite{BBD}.

\begin{lem}\label{lem1- thm2 atom Besov}
	Let $w\in A_{\vc}$, $N>n$, $\kappa\in [0,1]$, and $\eta, \nu \in \mathbb{Z}$, $\nu\geq \eta$. Assume that  $\{f_Q\}_{Q\in \mathcal{D}_\nu}$ is a  sequence of functions satisfying
	$$
	|f_Q(x)|\lesi \Big(\f{V(Q)}{V(x_Q,2^{-\eta})}\Big)^\kappa\Big(1+\f{d(x,x_Q)}{2^{-\eta}}\Big)^{-N}.
	$$
	Then for $\f{n\tau_w}{N}<r\leq 1$ and a sequence of numbers $\{s_Q\}_{Q\in \mathcal{D}_\nu}$, we have
	$$
	\sum_{Q\in \mathcal{D}_\nu}|s_Q|\,|f_Q(x)|\lesi 2^{n(\nu-\eta)(\tau_w/r-\kappa)}\mathcal{M}_{w,r}\Big(\sum_{Q\in \mathcal{D}_\nu}|s_Q|1_Q\Big)(x).
	$$
\end{lem}

The following elementary estimates will be used frequently. See for example \cite{BDK}.
\begin{lem}\label{lem-elementary}
	For any $f\in L^1_{\rm loc}(X)$ we have
	$$
	\int_X\f{1}{V(x\wedge y ,s)}\Big(1+\f{d(x,y)}{s}\Big)^{-(n+1)}|f(y)|\dy\le  a_n\mathcal{M}f(x),
	$$
	for all $x\in X$ and $s>0$, where $V(x\wedge y,r)=\min\{V(x,r),V(y,r)\}$ and $a_n$ depends on $n$ only.
	
	Consequently,
	$$
	\int_X\f{1}{V(x\wedge y ,s) }\Big(1+\f{d(x,y)}{s}\Big)^{-(n+1)} \dy\le  a_n,
	$$
	uniformly in  $x\in X$ and $s>0$.
\end{lem}

For $\lambda, N >0$ and $x,y\in X$, we set 
%$$
%\widetilde D_{\lambda, N}(x,y) = \f{1}{\sqrt{V(x,\lambda)V(x,\lambda)}}\Big(1+\f{d(x,y)}{\lambda}\Big)^{-N} 
%$$
%and
\begin{equation}\label{eq- Dlambda N}
D_{\lambda, N}(x,y) = \f{1}{ V(x\vee y,\lambda) }\Big(1+\f{d(x,y)}{\lambda}\Big)^{-N}, 
\end{equation}
where $V(x\vee y,\lambda)=\max\{V(x,\lambda), V(y,\lambda)\}\simeq V(x,\lambda)+ V(y,\lambda)$.

The following two lemmas  will discuss on the composition of functions $D_{\lambda, N}(\cdot,\cdot)$ which will be useful in the sequel. The proofs of these two lemmas will be given in section Appendix B.

\begin{lem}\label{lem-Ds Dt}
	Let $s,t >0$. Then for $N>n$ 
	\[
	\int_{X}D_{s, N}(x,z)D_{t, N}(z,y)d\mu(z)\lesi D_{s\vee t, N-n}(x,y)
	\]
	for all $x,y\in X$, where $s\vee t=\max\{s,t\}$.
	
	Moreover, we have
	\[
	\int_{X}D_{s, N}(x,z)D_{s, N}(z,y)d\mu(z)\lesi D_{s, N}(x,y)
	\]
	for all $x,y\in X$ and $s>0$. 
\end{lem}

\begin{lem}\label{lem-composite kernel} Let $N>2n$ and $\ell,k\in \mathbb Z$.
	\begin{enumerate}[(a)]
		\item If $\ell\ge k$, then we have
		\[
		\int_X D_{2^{-\ell},N}(x,y)D_{2^{-k},N}(y,z)D_{2^{-k},N}(y,u)d\mu(y)\lesi D_{2^{-k},\f{N-2n}{2}}(x,z)D_{2^{-k},\f{N-2n}{2}}(x,u)
		\]
		for all $x, z,u\in X$.
		
		\item If $k\ge \ell$, then we have
		\[
		\int_X D_{2^{-\ell},N}(x,y)D_{2^{-k},N}(y,z)D_{2^{-k},N}(y,u)d\mu(y)\lesi D_{2^{-\ell},\f{N-2n}{2}}(x,z)D_{2^{-k},\f{N-2n}{2}}(z,u)
		\]
		for all $x, z,u\in X$.
	\end{enumerate} 
\end{lem}

We now recall  the results in \cite[Theorem 3.1]{PK} and \cite[Corollary 3.5]{PK}.

\begin{prop}\label{thm-kernel estimate for functional calculus}
	Assume that $L$ satisfies conditions (A1) and (A2). Let $f \in C^{k}([0,\vc))$ with $k \ge 3n/2+1$, $\mathrm{supp}\, f \subset [0,R]$ for some $R \ge 1$, and 
	$f^{(2\nu+1)}(0) = 0$ for all $\nu \ge 0$ such that $2\nu+1 \le k$.  
	Then, for $\lambda>0$, the operator $f(\lambda\sqrt{L})$ is an integral operator with kernel $f(\lambda\sqrt{L})(x,y)$ satisfying
	\begin{equation}
		\label{3.1}
		\left| f(\lambda\sqrt{L})(x,y) \right| \le c_{k}   D_{\lambda,k-n/2}(x,y), 
	\end{equation}
	and
	\begin{equation}
		\label{3.2}
		\left| f(\lambda\sqrt{L})(x,y) - f(\lambda\sqrt{L})(x,y') \right| 
		\le c_{k}' \left( \frac{d(y,y')}{\lambda} \right)^{\delta_0} 
		D_{\lambda,k-n/2}(x,y) \quad \text{whenever} \quad d(y,y') \le \lambda.
	\end{equation}
	Here
	\begin{equation}
		\label{3.3}
		c_{k} = c_{k}(f) := R^{n} \left[ (c_{1} k)^{k} \| f \|_{\infty} 
		+ (c_{2} R)^{k} \| f^{(k)} \|_{\infty} \right],
	\end{equation}
	where $c_{1}, c_{2} > 0$ are independent of $f, R, \lambda$, and $k$,   
	$c_{k}' = c_{3} c_{k} R^{\delta_0}$ with $c_{3} > 0$ independent of $f, R, \lambda$, and $\delta_0 > 0$ is the constant from (A2).  
\end{prop}

\begin{prop}\label{prop- kernel of f(L)}
	Assume that $L$ satisfies conditions (A1) and (A2). Let $f$ be an even function in $\mathscr{S}(\mathbb R)$. Then, for $\lambda>0$, $f(\lambda\sqrt{L})$ is an integral operator with kernel $f(\lambda\sqrt{L})(x,y)$ satisfying
	\[
	|f(\lambda \sqrt L)(x,y)|\lesi_{N,f} D_{\lambda,N}(x,y),
	\]
	for all $\lambda>0$ and $x,y\in X$.
\end{prop}

\begin{rem}\label{rem1} In what follows, define $\mathscr{S}([0,\infty))$ to be the set of all functions $f$ defined on $[0,\vc)$ such that there exists an even function $g\in \mathscr{S}(\mathbb{R})$ with $f(x)=g(x)$ for all $x\in [0,\vc)$.
	
	Fix a function $\Sigma\in \mathscr{S}([0,\infty))$ such that $0\le \Sigma\le 1$, $\Sigma=1$ on $[0,4]$, and $\supp \Sigma\subset[0,5]$. Then, by Proposition \ref{thm-kernel estimate for functional calculus}, there exists a universal constant $b_\sharp$ such that for $\lambda>0$,
	\begin{equation}\label{eq- kernel Gamma function}
		\left| \Sigma(\lambda\sqrt{L})(x,y) - \Sigma(\lambda\sqrt{L})(x,y') \right| 
		\le b_\sharp \left( \frac{d(y,y')}{\lambda} \right)^{\delta_0} 
		D_{\lambda,n+1}(x,y) \quad \text{whenever} \quad d(y,y') \le \lambda.
	\end{equation}
\end{rem}

For each $\kappa ,\lambda>0$, define
\[
E_{\lambda, \kappa}(x,y)=\frac{1}{\sqrt{V(x,\lambda)V(y,\lambda)}}\exp\Big\{-\kappa \Big(\frac{d(x,y)}{\lambda}\Big)^{1/2}\Big\}.
\]
The following lemma, taken from \cite[Theorem 3.6]{PK}, establishes the existence of a functional calculus with kernels satisfying sub-exponential estimates.

\begin{lem}\label{lem-exp kernel}
	Assume that $L$ satisfies conditions (A1) and (A2). There exist constants $\kappa_0, c_\diamond>0$ and a cut-off function $\Gamma\in \mathscr{S}([0,\infty))$ such that $0\le \Gamma \le 1$, $\Gamma = 1$ on $[0,3]$, $\supp \Gamma \subset [0,4]$, and for any $\lambda>0$,
	\begin{equation}
		\label{3.15}
		|\Gamma(\lambda \sqrt{L})(x,y)| \leq c_{\diamond}E_{\kappa_0,\lambda}(x,y), 
		\qquad x,y \in X,
	\end{equation}
	and
	\begin{equation}
		\label{3.16}
		|\Gamma(\lambda \sqrt{L})(x,y)-\Gamma(\lambda \sqrt{L})(x,y')| 
		\leq c_{\diamond}\Big(\frac{d(y,y')}{\lambda}\Big)^{\delta_0} E_{\kappa_0,\lambda}(x,y) \quad \text{whenever} \quad d(y,y') \le \lambda.
	\end{equation}
	
	Furthermore, for $\lambda>0$ and any $m \in \mathbb{N}$,
	\begin{equation}
		\label{3.17}
		|L^{m}\Gamma(\lambda \sqrt{L})(x,y)| 
		\leq  c_{m}\lambda^{-2m} E_{\kappa_0,\lambda}(x,y), 
		\qquad x,y \in X.
	\end{equation}
\end{lem}

The following lemma addresses the composition of sub-exponential kernels. Its proof is similar to that of \cite[Lemma 3.10]{PK} and hence  the details are omitted. 
\begin{lem}\label{lem-product exp}
	Let $\kappa_0$ as in Lemma \ref{lem-exp kernel}. There exists $c_\sharp>0$ such that for $j\in \mathbb Z$ and $x,y\in X$,
	\begin{equation}\label{eq1-product}
		\int_X E_{2^{-j}, \kappa_0}(x,z)E_{2^{-j}, \kappa_0}(z,y)d\mu(z)\le c_{\sharp}E_{2^{-j}, \kappa_0}(x,y)
	\end{equation}
	and
	\begin{equation}\label{eq2-product}
		\sum_{Q\in \mathcal D_j} V(Q) E_{2^{-j}, \kappa_0}(x,x_Q)E_{2^{-j}, \kappa_0}(x_Q,y) \le c_{\sharp}E_{2^{-j}, \kappa_0}(x,y).
	\end{equation}
\end{lem}

\subsection{Distributions and Peetre maximal functions }\label{sec-distribution} 

Fix a reference point $x_0\in X$. The class of test functions $\mathcal{S}$ associated with $L$ is defined as the set of all functions $\phi \in \cap_{m\geq 1}D(L^m)$ such that
\begin{equation}
	\label{Pml norm}
	\mathcal{P}_{m,\ell}(\phi)=\sup_{x\in X}(1+d(x,x_0))^m|L^\ell \phi(x)|<\vc, \ \ \forall m>0, \ell \in \mathbb{N}.
\end{equation}
It was proved in \cite{PK} that $\mathcal{S}$ is a complete locally convex space with topology generated by the family of seminorms $\{\mathcal{P}_{m,\ell}: \, m>0, \ell \in \mathbb{N}\}$. It was proved in \cite{PK} that when $L=-\Delta$ the Laplacian on $\mathbb R^n$, the class $\mathcal S$ coincide with the Schwartz class $\mathscr S(\mathbb R^n)$ on $\mathbb R^n$. As usual, we define the space of distributions $\mathcal{S}'$ as the set of all continuous linear functionals on $\mathcal{S}$, with the inner product defined by
\[
\langle f,\phi\rangle=f(\overline{\phi}),
\]
for all $f\in \mathcal{S}'$ and $\phi\in \mathcal{S}$.

The space $\mathcal{S}'$ can be used to define the inhomogeneous Besov and Triebel–Lizorkin spaces. However, to study their homogeneous counterparts, some modifications are required. 

Following \cite{G.etal}, we define the space $\mathcal{S}_\vc$ as the set of all functions $\phi \in \mathcal{S}$ such that for each $k\in \mathbb{N}$ there exists $g_k\in \mathcal{S}$ with $\phi=L^kg_k$. Note that such $g_k$, if it exists, is unique (see \cite{G.etal}).

The topology of $\mathcal{S}_\vc$ is generated by the family of seminorms 
\[
\mathcal{P}^*_{m,\ell,k}(\phi)=\mathcal{P}_{m,\ell}(g_k), \ \forall m>0; \ell, k\in \mathbb{N},
\]
where $\phi=L^k g_k$.

When $L=-\Delta$, it can be verified that the class $\mathcal{S}_\vc$ is identical to the subspace of functions in $\mathscr S(\mathbb R^n)$ that have vanishing moments of all orders. We then denote by $\mathcal{S}_\vc'$ the space of all continuous linear functionals on $\mathcal{S}_\vc$.

For $\lambda>0, j\in \mathbb{Z}$, and $\varphi\in \mathscr{S}(\mathbb{R})$, the Peetre-type maximal function is defined, for $f\in \mathcal S_\vc$, by
\begin{equation}
	\label{eq1-PetreeFunction}
	\varphi_{j,\lambda}^*(\sqrt{L})f(x)=\sup_{y\in X}\f{|\varphi_j(\sqrt{L})f(y)|}{(1+2^jd(x,y))^\lambda}\, , \quad x\in X,
\end{equation}
where $\varphi_j(\lambda)=\varphi(2^{-j}\lambda)$.

Clearly,
\[
\varphi_{j,\lambda}^*(\sqrt{L})f(x)\geq |\varphi_j(\sqrt{L})f(x)|, \ \ \ x\in X.
\]
Similarly, for $s, \lambda>0$ we set
\begin{equation}
	\label{eq2-PetreeFunction}
	\varphi_{\lambda}^*(s\sqrt{L})f(x)=\sup_{y\in X}\f{|\varphi(s\sqrt{L})f(y)|}{(1+d(x,y)/s)^\lambda}, \ \ \ f\in \mathcal S_\vc.
\end{equation}

\begin{prop}
	\label{prop1-maximal function}
	Let $\psi\in \mathscr{S}(\mathbb{R})$ with ${\rm supp}\,\psi\subset [1/3,3]$, and let $\varphi\in \mathscr{S}(\mathbb{R})$ be a partition of unity. Then, for any $\lambda>0$ and $j\in \mathbb{Z}$,
	\begin{equation}
		\label{eq-psistar vaphistar}
		\psi^*_{j,\lambda}(\sqrt{L})f(x) \lesi C_\lambda\big[\|\psi\|_\vc + \|\psi^{(\ell)}\|_{L^\vc}\big]\sum_{k=j-2}^{j+3} \varphi^*_{k,\lambda}(\sqrt{L})f(x),
	\end{equation}
	for all $x\in X$, where $\ell = \lfloor 2n+\lambda \rfloor +2$.
\end{prop}
The proof of this proposition is given in Appendix B.

\section{Weighted Besov and Triebel-Lizorkin spaces associated to $L$}\label{sec3}

This section recalls the definitions of the weighted Besov and Triebel–Lizorkin spaces associated with $L$ and some basic properties from \cite{BBD}. 
\textit{Throughout this section, we assume that the operator $L$ satisfies conditions (A1) and (A2).}

In what follows, by a “partition of unity’’ we mean a function $\psi \in \mathcal{S}(\mathbb{R})$ such that 
$\supp \psi \subset [1/2,2]$, $\int \psi(\xi)\,\f{d\xi}{\xi}\neq 0$, and
\[
\sum_{j\in \mathbb{Z}}\psi_j(\lambda)=1 \quad \text{for } \lambda>0,
\]
where $\psi_j(\lambda):=\psi(2^{-j}\lambda)$ for each $j\in \mathbb{Z}$.

\begin{defn}[\cite{BBD}]
	Let $\psi$ be a partition of unity. For $0< p, q\leq \vc$, $s\in \mathbb{R}$ and $w\in A_\vc$, we define the weighted homogeneous Besov space $\dot{B}^{s, \psi, L}_{p,q}(w)$ by
	\[
	\dot{B}^{s, \psi, L}_{p,q}(w) = 
	\Big\{f\in \mathcal{S}'_\vc:  \|f\|_{\dot{B}^{s, \psi, L}_{p,q}(w)}<\vc\Big\},
	\]
	where
	\[
	\|f\|_{\dot{B}^{s, \psi, L}_{p,q}(w)}= \Big\{\sum_{j\in \mathbb{Z}}\big(2^{js}\|\psi_j(\sqrt{L})f\|_{L^p(w)}\big)^q\Big\}^{1/q}.
	\]
	
	Similarly, for $0< p<\vc$, $0<q\le \vc$, $s\in \mathbb{R}$ and $w\in A_\vc$, the weighted homogeneous Triebel–Lizorkin space $\dot{F}^{s, \psi, L}_{p,q}(w)$ is defined by 
	\[
	\dot{F}^{s, \psi, L}_{p,q}(w) = 
	\Big\{f\in \mathcal{S}'_\vc:  \|f\|_{\dot{F}^{s, \psi, L}_{p,q}(w)}<\vc\Big\},
	\]
	where
	\[
	\|f\|_{\dot{F}^{s, \psi, L}_{p,q}(w)}= \Big\|\Big[\sum_{j\in \mathbb{Z}}\big(2^{js}|\psi_j(\sqrt{L})f|\big)^q\Big]^{1/q}\Big\|_{L^p(w)}.
	\]
\end{defn}

It was shown in \cite{BBD} that these spaces are independent of the choice of partition of unity. Hence, we drop the notation $\psi$ in the above definition and simply write $\dot{B}^{s,L}_{p,q}(w)$ and $\dot{F}^{s,L}_{p,q}(w)$ for the weighted Besov and Triebel–Lizorkin spaces associated with $L$.

The following results are taken from \cite[Proposition 3.3 \& Theorem 7.1]{BBD}.
\begin{prop}
	\label{prop2-thm1}
	Let $\psi$ be a partition of unity. Then:
	\begin{enumerate}[{\rm (a)}]
		\item For $0< p, q\le \vc$, $s\in \mathbb{R}$ and $\lambda>nq_w/p$,
		\[
		\Big\{\sum_{j\in \mathbb{Z}}\big(2^{js}\|\psi^*_{j,\lambda}(\sqrt{L})f\|_{L^p(w)}\big)^q\Big\}^{1/q}\simeq \|f\|_{\dot{B}^{s, L}_{p,q}(w)}.
		\]
		
		\item For $0< p<\vc$, $0<q\le \vc$, $s\in \mathbb{R}$ and $\lambda>\max\{n/q, nq_w/p\}$,
		\[
		\Big\|\Big[\sum_{j\in \mathbb{Z}}(2^{js}|\psi^*_{j,\lambda}(\sqrt{L})f|)^q\Big]^{1/q}\Big\|_{L^p(w)}\simeq \|f\|_{\dot{F}^{s, L}_{p,q}(w)}.
		\]
	\end{enumerate}
	Here the maximal function $\psi^*_{j,\lambda}(\sqrt{L})$ is defined in \eqref{eq1-PetreeFunction}.
\end{prop}

\begin{prop}\label{prop- Ls on TL B spaces}
	Let $s\in \mathbb R$. Then for $\alpha\in \mathbb R$ and $w\in A_\vc$, we have
	\[
	\|L^{s/2}f\|_{\dot{B}^{\alpha,L}_{p,q}(w)}\simeq \|f\|_{\dot{B}^{s+\alpha,L}_{p,q}(w)}, \quad 0<p,q\le \vc,
	\]
	and
	\[
	\|L^{s/2}f\|_{\dot{F}^{\alpha,L}_{p,q}(w)}\simeq \|f\|_{\dot{F}^{s+\alpha,L}_{p,q}(w)}, \quad 0<p<\vc,\; 0<q\le \vc.
	\]
\end{prop}

Let $f\in L^2(X)$. For $0<p\le \vc$ and $w\in A_\vc$, define
\[
\|f\|_{H^p_{L}(w)} =\Big\| \sup_{t>0}|e^{-tL}f|\Big\|_{L^p(w)}.
\]
The theory of Hardy spaces, Besov spaces, and Triebel--Lizorkin spaces associated with operators is an important topic in harmonic analysis and has attracted considerable attention from mathematicians. Similar to the classical cases, these spaces enjoy many of the same properties as their classical counterparts, such as interpolation properties and square function characterizations. For further details, we refer to \cite{DY, DY1, PK, HLMMY, JY, BBD} and the references therein.

\medskip
We recall some identifications of the Besov and Triebel--Lizorkin spaces with certain well-known function spaces from \cite[Section~5]{BBD}:
\begin{rem}
	\label{rem2}
\begin{itemize} 
	\item It follows from Theorem~1.8 in \cite{YY} and Theorem~5.2 in \cite{BBD} that 
	\[
	H^p_{L}(w) \equiv \dot{F}^{0,L}_{p,2}(w), \qquad 0<p\le 1.
	\]
	In fact, this identification remains valid for all $0<p<\vc$, although we omit the proof here. 
	
	\item $\dot{F}^{0,L}_{p,2}(w)\equiv L^p(w)\equiv H^p_L(w)$ for $1<p<\vc$ and $w\in A_p$; 
	%\item $\dot{F}^{0,L}_{p,2}(w)\equiv H^p_L(w)$ for $0<p\le 1$ and $w\in A_\vc$ as above;
	
	\item $\dot{F}^{0,L}_{\vc,2}\equiv BMO_L $ where $BMO_L$ is the BMO space associated to $L$ as in \cite{DY1};
	
	\item Apart from (A1) and (A2), if $L$ additionally satisfies the following \begin{equation}\label{eq-conservation}
		\displaystyle \int_X {{p_t}\left( {x,y} \right)d\mu \left( x \right)}  = 1
\end{equation} 
	for all $y \in X$ and $t>0$, then $\dot{F}^{s,L}_{p,q} =\dot{F}^{s}_{p,q} $  for $\frac{n}{n+\delta_0}<p,q<\infty$ and and $\dot{B}^{s,L}_{p,q}=\dot{B}^{s}_{p,q}$ for $\frac{n}{p\wedge q}-n-\delta_0<s<\delta_0$, where $\dot{B}^{s}_{p,q}$ and $\dot{F}^{s}_{p,q} $ are the Besov and Triebel--Lizorkin spaces defined as in \cite{HMY, HS}.
\end{itemize} 
For further properties and other identifications, we refer to \cite{BBD}.
\end{rem}

We have the following result.
\begin{prop}\label{thm1-maximal function}
	Let $\varphi\in \mathscr{S}([0,\infty))$  supported in $[0,4]$. For $0<p< \vc$, $w\in A_\infty$ and $\lambda> \max\{ n\tau_w/p,n/\tau_w \}$ we have
	\[
	\sup_{s>0}\varphi^*_\lambda(s\sqrt L)f\lesi \|\varphi\|_{W^\ell_\vc} \sup_{s>0}(e^{-sL})^*_\lambda f,
	\]
	and
	\[
	\Big\|\sup_{s>0}(e^{-sL})^*_\lambda f\Big\|_{L^p(w)}\lesi \|f\|_{H^p_L(w)},
	\]
	where $\varphi^*_\lambda(s\sqrt L)f$ and $(e^{-tL})^*_\lambda f$ are defined as in \eqref{eq2-PetreeFunction},  $\ell = \lfloor 2n+\lambda \rfloor+2$ and $\|\varphi\|_{W^\ell_\vc}=\sum_{|\alpha|\le \ell}\|\partial^\alpha f\|_{L^\vc}$.
\end{prop}
\begin{proof}
	Set 
	\[
	C_{\varphi,\ell}= \|\varphi\|_\vc + \|\varphi^{(\ell)}\|_\vc, 
	\quad \ell = \lfloor 2n+\lambda \rfloor +2.
	\] 
	Fix $m>\lambda/2$. By integration by parts, 
	\[
	\int_s^\infty (t^2u)^m e^{-t^2u}\,\frac{dt}{t} = P(us),
	\]
	where $P$ is a polynomial with $\deg P = m-1$.  
	
	Hence, 
	\begin{equation}\label{eq-the identity}
		I = P(s^2L)e^{-s^2L} + \int_0^{s} (t^2L)^m e^{-tL}\,\frac{dt}{t}.
	\end{equation}
	
	It follows that 
	\[
	\begin{aligned}
		\varphi(s\sqrt L)f
		&= P(s^2L)e^{-s^2L}\varphi(s\sqrt L)f
		+ \int_0^{s} (t^2L)^m e^{-t^2L}\varphi(s\sqrt L)f\,\frac{dt}{t} \\
		&= P(s^2L)e^{-s^2L}\varphi(s\sqrt L)f
		+ \int_0^{s}\Big(\frac{t}{s}\Big)^{2m} e^{-t^2L} (s\sqrt L)^{2m}\varphi(s\sqrt L)f\,\frac{dt}{t}.
	\end{aligned}
	\]
	
	By Theorem~\ref{thm-kernel estimate for functional calculus} and Lemma \ref{lem-elementary}, 
	\[
	\begin{aligned}
		\frac{|\varphi(s\sqrt L)f(y)|}{(1+d(x,y)/s)^\lambda}
		&\lesi C_{\varphi,\ell}\Big(1+\frac{d(x,y)}{s}\Big)^{-\lambda}
		\int_X D_{s,\lambda+n+1}(y,z)
		|e^{-s^2L}f(z)|\,\dz \\
		&\quad + C_{\varphi,\ell}\Big(1+\frac{d(x,y)}{s}\Big)^{-\lambda}
		\int_0^{s^2}\Big(\frac{t}{s}\Big)^{2m}
		\int_X D_{s,\lambda+n+1}(y,z)
		|e^{-t^2L}f(z)|\,\dz \,\frac{dt}{t} \\
		&\lesi C_{\varphi,\ell} \sup_{s>0}\sup_{z}\frac{|e^{-s^2L}f(z)|}{(1+d(x,z)/s)^\lambda}\int_X D_{s,n+1}(y,z)
		|e^{-s^2L}f(z)|\,\dz \\
		&\quad + C_{\varphi,\ell}\int_0^{s^2}\Big(\frac{t}{s}\Big)^{2m-\lambda}
		\int_X D_{s,n+1}(y,z)
		\Big(1+\frac{d(z,x)}{t}\Big)^{-\lambda}
		|e^{-t^2L}f(z)|\,\dz \,\frac{dt}{t} \\
		&\lesi C_{\varphi,\ell} \sup_{s>0}\sup_{z}\frac{|e^{-s^2L}f(z)|}{(1+d(x,z)/s)^\lambda}.
	\end{aligned}
	\]
	
	This implies 
	\[
	\sup_{s>0}\varphi^*_\lambda(s\sqrt L)f 
	\lesi C_{\varphi,\ell}\sup_{s>0}(e^{-s^2L})^*_\lambda f,
	\]
	which completes the first part.  
	
\bigskip
	
	For the second part, choose $\theta \in (0,1)$ with $\lambda\theta>n$ and $p/\theta>\tau_w$. By \eqref{boundedness maximal function 2}, it suffices to prove that  
	\begin{equation}\label{eq-ML}
		\sup_{s>0}(e^{-s^2L})^*_\lambda f
		\lesi \mathcal M_\theta\big(\sup_{t>0}|e^{-t^2L}f|\big).
	\end{equation}
	
	Indeed, from \eqref{eq-the identity}, 
	\[
	e^{-s^2L}f 
	= P(s^2L)e^{-s^2L}e^{-s^2L}f
	+ \int_0^{s}\Big(\frac{t}{s}\Big)^{2m}e^{-t^2L}(s^2L)^me^{-s^2L}f\,\frac{dt}{t}.
	\]
	
	By Theorem~\ref{thm-kernel estimate for functional calculus}, 
	\[
	\begin{aligned}
		\frac{|e^{-s^2L}f(y)|}{(1+d(x,y)/s)^\lambda}
		&\lesi \Big(1+\frac{d(x,y)}{s}\Big)^{-\lambda}
		\int_X D_{s,\lambda}(y,z)
		|e^{-s^2L}f(z)|\,\dz \\
		&\quad + \Big(1+\frac{d(x,y)}{s}\Big)^{-\lambda}
		\int_0^{s}\Big(\frac{t}{s}\Big)^{2m}
		\int_X D_{s,\lambda}(y,z)
		|e^{-t^2L}f(z)|\,\dz \,\frac{dt}{t} \\
		&\lesi \int_X \f{1}{V(z,s)} \Big(1+\frac{d(x,z)}{s}\Big)^{-\lambda}|e^{-s^2L}f(z)|\,\dz \\
		&\quad + \int_0^{s^2}\Big(\frac{t}{s}\Big)^{2m-\lambda}
		\int_X \f{1}{V(z,t)} \Big(1+\frac{d(x,z)}{t}\Big)^{-\lambda}|e^{-t^2L}f(z)|\,\dz \,\frac{dt}{t}.
	\end{aligned}
	\]
	
	It follows that
	\[
	\begin{aligned}
		(e^{-s^2L})^*_\lambda f(x)
		&\lesi \big[(e^{-s^2L})^*_\lambda f(x)\big]^{1-\theta}
		\int_X \f{1}{V(z,s)} \Big(1+\frac{d(x,z)}{s}\Big)^{-\lambda\theta}
		|e^{-s^2L}f(z)|^\theta\,\dz \\
		&\quad + \big[(e^{-s^2L})^*_\lambda f(x)\big]^{1-\theta}
		\int_0^{s^2}\Big(\frac{t}{s}\Big)^{2m-\lambda}
		\int_X \f{1}{V(z,t)} \Big(1+\frac{d(x,z)}{t}\Big)^{-\lambda\theta}
		|e^{-t^2L}f(z)|^\theta\,d\mu(z) \,\frac{dt}{t}.
	\end{aligned}
	\]
	
	By Lemma~\ref{lem-elementary}, we conclude
	\[
	(e^{-s^2L})^*_\lambda f(x)
	\lesi \mathcal M_\theta\big(\sup_{t>0}|e^{-t^2L}f|\big)(x),
	\]
	for all $s>0$, which proves the claim.
	
\end{proof}

\section{Fractional Leibniz Rules}

In order to prove the main results, we need the following lemma whose proof will be given in Appendix C.
\begin{lem}\label{lem-kernel-representation}
	Let $L$ satisfy the conditions (A1) and (A2). There exists a sequence of measurable functions $\{\widetilde\Gamma_j\}_{j\in \mathbb Z}$ defined on $X\times X$ satisfying the following properties:
	\begin{enumerate}[\rm (i)]
		\item For every  $m\in \mathbb N$, there exists $C_{m}>0$ such that for all $x,y\in X$ and $j\in \mathbb Z$,
		\[
		|L^m\widetilde\Gamma_j(x,y)|\le C_{m}\,2^{2jm} E_{2^{-j},\kappa_0/2}(x,y).
		\]
		
		\item For   $j\in \mathbb Z$,
		\[
		|\widetilde\Gamma_j(x,y)-\widetilde\Gamma_j(\overline x,y)|
		\le  C  (2^jd(x,\overline x) )^{\delta_0}  E_{2^{-j},\kappa_0/2}(x,y),
		\quad \text{whenever } d(x,\overline x)<2^{-j}.
		\] 
		\item For $\psi \in \mathscr{S}([0,\infty))$ with $\supp \psi \subset [0,3]$ and any $j\in \mathbb Z$, 
		\[
		\psi_j(\sqrt L)f(x) 
		= \sum_{Q\in \mathcal D_j}\omega_Q \,\psi_j(\sqrt L)f(x_Q)\,\widetilde\Gamma_j(x,x_Q) 
		\]
		for $f\in \mathcal S_{\vc}, $where $\omega_Q\simeq V(Q)$ for each $Q\in \mathcal D_j$ and each $j\in \mathbb N$.
		\item Moreover, if $L$ satisfies (A3) additionally, then for every $m\in \mathbb N$ and $N>0$, there exists $C_{m,N}>0$ such that for all $x,y,z\in X$ and $j\in \mathbb Z$,
		\[
		\big|L^m\big[\widetilde\Gamma_j(x,y)\widetilde\Gamma_j(x,z)\big]\big|
		\lesssim C_{m,N}\,2^{2jm}D_{2^{-j},N}(x,y)D_{2^{-j},N}(x,z).
		\]

	\end{enumerate}
\end{lem}

We are ready to give the proof for Theorem \ref{main thm}.
	
\begin{proof}[Proof of Theorem \ref{main thm}:] We only give the proof of \eqref{eq-main thm TL spaces} since the proof of \eqref{eq-main thm B spaces} can be done similarly.
	
	\medskip
	
We now fix  $0 < p,p_1 p_4 < \infty$ and $0 <  p_2,p_3  \le \infty$ such that $1/p = 1/p_1 + 1/p_2=1/p_3 + 1/p_4$, $0 < q \le \infty$ and $w_1,w_2,w_3,w_4 \in A_\infty$ such that $w_1^{p/p_1} w_2^{p/p_2}=w_3^{p/p_3} w_4^{p/p_4}=:w$.

Since $w_1,w_2,w_3,w_4 \in A_\infty$, $w\in A_\vc$. Hence, $\tau_w$ is well-defined.  Fix \[
	s>\tau_{p,q}(w) := n \left( \frac{1}{\min(p/\tau_w,q,1)} - 1 \right).\]
%Recall that $\tau_p(w) := n \left( \frac{1}{\min(p/\tau_w,1)} - 1 \right).$

We also fix $0<r_1<p_1$, $0<r_2<p_2$ and $r<\min\{p/\tau_w,q,1\}$ such that $s>\tau_{r,q}(w)$.

Let $\psi\in C^\infty(\mathbb{R})$ be  a partition of unity  such that $0\le\psi\le1$ and
\[
\supp\psi(\xi)\subset [1/2,2],\qquad \sum_{j\in\mathbb Z}\psi_j(\xi) = 1 \ \text{for} \ \xi>0,
\]
where $\psi_j(\cdot)=\psi(2^{-j}\cdot)$.

Set
\[
\phi(\xi) =\begin{cases}
	\sum\limits_{j\le 0}\psi_j(\xi), \ \ & \xi>0\\
	0, \ \ & \xi=0.
\end{cases}
\]
It is easy to see that $\phi\in \mathscr{S}([0,\infty))$ and $\supp \phi\subset [0,2]$; moreover, for any $k\in \mathbb Z$, 
\[
\phi_k(\xi):=\phi(2^{-k}\xi) = \sum_{j\le k}\psi_j(\xi). 
\]

We write
\begin{equation}\label{eq- - m decomposition}
\begin{aligned}
	\boldsymbol{\mathrm{m}}(\xi,\eta)&=\sum_{k\in \mathbb Z}\sum_{j\in \mathbb{Z}} \boldsymbol{\mathrm{m}}(\xi,\eta)\psi_k(\xi)\psi_j(\eta)\\
	&=\sum_{k\in \mathbb Z}\sum_{j\in \mathbb Z\atop  j\le  k} \boldsymbol{\mathrm{m}}(\xi,\eta)\psi_k(\xi)\psi_j(\eta) +\sum_{j\in \mathbb Z}\sum_{k\in \mathbb Z \atop k\le j-1} \boldsymbol{\mathrm{m}}(\xi,\eta)\psi_k(\xi)\psi_j(\eta)\\
	&= \sum_{k\in \mathbb Z}  \boldsymbol{\mathrm{m}}(\xi,\eta)\psi_k(\xi)\phi_k(\eta) +\sum_{j\in \mathbb Z}  \boldsymbol{\mathrm{m}}(\xi,\eta)\phi_{j-1}(\xi)\psi_j(\eta).
\end{aligned}
\end{equation}
Consequently,
\[
B_{\boldsymbol{\mathrm{m}},L}(f,g)= \sum_{k\in \mathbb Z}  B_{\boldsymbol{\mathrm{m}}_{1,k},L}(f,g) +\sum_{j\in \mathbb Z}  B_{\boldsymbol{\mathrm{m}}_{2,j},L}(f,g),
\]
where $\boldsymbol{\mathrm{m}}_{1,k}(\xi,\eta)=\boldsymbol{\mathrm{m}}(\xi,\eta)\psi_k(\xi)\phi_k(\eta)$ and $\boldsymbol{\mathrm{m}}_{2,j}(\xi,\eta)=\boldsymbol{\mathrm{m}}(\xi,\eta)\phi_{j-1}(\xi)\psi_j(\eta)$.

Since these two terms above are similar, we need only to estimate the first term 
$$
\displaystyle \sum_{k}  B_{\boldsymbol{\mathrm{m}}_{1,k},L} (f,g).
$$
Note that
\begin{equation}\label{eq-supp m1k}
	\operatorname{supp} \boldsymbol{\mathrm{m}}_{1,k} \subset \{(\xi,\eta): 2^{k-2}\le|\xi|\le2^{k+1},\;0\le|\eta|\le2^{k+1}\}.
\end{equation}
Following \cite{CM}, we  define
the Fourier coefficients as
\begin{equation}\label{eq-cn}
c_k(n_1,n_2)
=2^{-2k} \int_{2^{k-1}<|\xi|\le 2^{k+1}}\int_{|\eta|\le 2^{k+1}}\,
\boldsymbol{\mathrm{m}}_{1,k}(\xi,\eta)e^{-2\pi i\frac{\xi n_1+\eta n_2}{2^k}}\,d\xi\,d\eta.
\end{equation}
Then, by integration by parts, \eqref{eq-supp m1k} and \eqref{eq-condition on m},
\begin{equation}\label{eq-cnn}
|c_k(n_1,n_2)| \lesssim_N 2^{k\gamma}(1+|n_1|+|n_2|)^{-N}
\end{equation}
for every $N$, with the implicit constant depending on  $N$.

In addition,  by the Fourier expansion, 
\[
\begin{aligned}
	\boldsymbol{\mathrm{m}}_{1,k}(\xi,\eta)&=\sum_{(n_1,n_2)\in\mathbb{Z}^2} c_k(n_1,n_2)\,
	e^{2\pi i\big(\frac{\xi n_1+\eta n_2}{2^k}\big)}.
\end{aligned}
\]
Let $\Psi, \Phi\in C^\infty(\mathbb{R}^d)$ be two functions such that $0\le\Psi, \Phi\le1$, 
\[
\supp \Psi(\xi) \subset[1/3,3],\qquad
\Psi(\xi)=1\quad\text{if }|\xi|\in[1/2,2]
\]
and
\[
\supp\Phi(\xi) \subset[0,3],\qquad
\Phi(\xi)=1\quad\text{if }|\xi|\in[0,2].
\]
From \eqref{eq-supp m1k}, we can write
\[
\begin{aligned}
	\boldsymbol{\mathrm{m}}_{1,k}(\xi,\eta)&=\sum_{(n_1,n_2)\in\mathbb{Z}^2} c_k(n_1,n_2)\,
	e^{2\pi i\big(\frac{\xi n_1+\eta n_2}{2^k}\big)}\Psi_k(\xi)\Phi_k(\eta)\\
	&=:	\sum_{(n_1,n_2)\in\mathbb{Z}^2} c_k(n_1,n_2)\,
	\Psi_{k,n_1}(\xi)\Phi_{k,n_2}(\eta),
\end{aligned}
\]
where $\Psi_{k,n_1}(\xi)=e^{\frac{2\pi i\xi n_1}{2^k}}\Psi_k(\xi)$ and $\Phi_{k,n_2}(\eta)=e^{\frac{2\pi i\eta n_2}{2^k}}\Phi_k(\eta)$.

Consequently,
\[
B_{\boldsymbol{\mathrm{m}}_{1,k},L}(f,g)(x)= \sum_{k\in \mathbb Z}\sum_{(n_1,n_2)\in\mathbb{Z}^2} c_k(n_1,n_2)\,
\Psi_{k,n_1}(\sqrt L)f(x) \Phi_{k,n_2}(\sqrt L)g(x).
\]

Applying Lemma \ref{lem-kernel-representation},
\[
\begin{aligned}
	&B_{\boldsymbol{\mathrm{m}}_{1,k},L}(f,g)(x)\\
	 &= \sum_{k\in \mathbb Z} \sum_{(n_1,n_2)\in\mathbb{Z}^2} c_k(n_1,n_2) \sum_{Q\in \mathcal D_k} \sum_{R\in \mathcal D_k}\omega_Q\Psi_{k,n_1}(\sqrt{L})f(x_Q) \widetilde{\Gamma}_k(x,x_Q) \omega_R\Phi_{k,n_2}(\sqrt{L})g(x_R)  \widetilde{\Gamma}_k(x,x_R).	
\end{aligned}
\]
Setting
\begin{equation}\label{eq-theta}
\theta_\ell(\xi)= (2^{-\ell}\xi)^s\psi_\ell(\xi),
\end{equation}
then we can further write
\begin{equation}\label{eq- 1st eq BmL}
\begin{aligned}
	\psi_\ell&(\sqrt L) L^{s/2}\big[B_{\boldsymbol{\mathrm{m}}_{1,k},L}(f,g)\big](x)\\
	=&  \sum_{k\in \mathbb Z} \sum_{(n_1,n_2)\in\mathbb{Z}^2} c_k(n_1,n_2)\sum_{Q\in \mathcal D_k} \sum_{R\in \mathcal D_k}2^{\ell s} \omega_Q \Psi_{k,n_1}(\sqrt{L})f(x_Q)  \omega_{R}\Phi_{k,n_2}(\sqrt{L})g(x_R) \\
	& \ \ \times  \int_X \theta_\ell(\sqrt L)(x,y)\widetilde{\Gamma}_k(y,x_Q) \widetilde{\Gamma}_k(y,x_R) d\mu(y)\\
	=&  \sum_{k\in \mathbb Z} \sum_{(n_1,n_2)\in\mathbb{Z}^2} c_k(n_1,n_2)\sum_{Q\in \mathcal D_k} \sum_{R \in \mathcal D_k}2^{s( \ell-k)} \omega_Q \overline{\Psi}_{k,n_1}(\sqrt{L})(L^{s/2}f)(x_Q)  \omega_R \Phi_{k,n_2}(\sqrt{L})g(x_R)\\
	& \ \ \times   \int_X \theta_\ell(\sqrt L)(x,y)\widetilde{\Gamma}_k(y,x_Q) \widetilde{\Gamma}_k(y,x_R) d\mu(y)\\
	=&: \sum_{k \in \mathbb Z \atop k\le \ell} E_{\ell,k}+\sum_{k \in \mathbb Z \atop k>\ell} E_{\ell,k},
\end{aligned}
\end{equation}
where $\overline{\Psi}_{k,n_1}(\sqrt L) = [2^{-k}\sqrt L]^{-s}\Psi_{k,n_1}(\sqrt L)$.

\bigskip

We will estimate the sums corresponding to $k\le \ell$ and $k>\ell$ separately. 

\textbf{Case 1: $\ell \ge k$}

\bigskip

From Proposition \ref{thm-kernel estimate for functional calculus} and Lemmas \ref{lem-kernel-representation} and \ref{lem-composite kernel},  we have, for $\mathbb N\ni m>(s+1)/2$ and a fixed $M>\max\{\f{n\tau_{w_1}}{r_1}, \f{n\tau_{w_2}}{r_2}\}$,
\begin{equation}\label{eq- the use of A3}
\begin{aligned}
	\Big|\int_X \theta_\ell(\sqrt L)(x,y)& \widetilde{\Gamma}_k(y,x_Q)\widetilde{\Gamma}_k(y,x_R) d\mu(y)\Big|\\
	&=\Big|\int_X \theta_\ell(\sqrt L)(y,x) \widetilde{\Gamma}_k(y,x_Q)\widetilde{\Gamma}_k(y,x_R) d\mu(y)\Big|\\
	&=\Big|\int_X 2^{-2\ell m} (2^{-\ell}\sqrt L)^{2m}\theta_\ell(\sqrt L)(y,x) L^m\big[\widetilde{\Gamma}_k(y,x_Q)\widetilde{\Gamma}_k(y,x_R)\big] d\mu(y)\Big|\\
	&\lesi \int_X 2^{-2m(\ell-k)}D_{2^{-\ell},2M+2n}(x,y)D_{2^{-k},2M+2n}(y,x_Q)D_{2^{-k},2M+2n}(y,x_R)  d\mu(y)\\
	&\lesi  2^{-(s+1)(\ell-k)}D_{2^{-k},M}(x,x_Q)D_{2^{-k},M}(x,x_R).
\end{aligned}
\end{equation}

Hence, using Lemma \ref{lem1- thm2 atom Besov} and the fact $\omega_Q\simeq V(Q)$ and $\omega_R\simeq V(R)$,
\[
\begin{aligned}
	|E_{\ell,k}(x)|&\lesi  \sum_{(n_1,n_2)\in\mathbb{Z}^2} |c_k(n_1,n_2)|\sum_{Q\in \mathcal D_k} \sum_{R\in \mathcal D_k}2^{-(\ell-k)} V(Q)  |\overline{\Psi}_{k,n_1}(\sqrt{L})(L^{s/2}f)(x_Q)|\\
	& \ \ \ \ \ \times  V(R) |\Phi_{k,n_2}(\sqrt{L})g(x_R)|D_{2^{-k},M}(x,x_Q)D_{2^{-k},M}(x,x_R)\\
	& \lesi \sum_{(n_1,n_2)\in\mathbb{Z}^2} |c_k(n_1,n_2)|2^{(k-\ell)} \\
	& \ \ \ \ \ \times \mathcal M_{r_1,w_1}\Big(\sum_{Q\in \mathcal D_k}|\overline{\Psi}_{k,n_1}(\sqrt{L})(L^{s/2}f)(x_Q)|1_{Q} \Big)(x) \mathcal M_{r_2,w_2}\Big(\sum_{R\in \mathcal D_k}|{\Phi}_{k,n_2}(\sqrt{L})g(x_R)|1_{R} \Big)(x)  
\end{aligned}
\]
It can be verified that for $\lambda> \max\{n/q, n\tau_w/p,, n/\tau_w\}$, 
\begin{equation}\label{eq - overlinePhi estimate}
\sum_{Q\in \mathcal D_k}|\overline{\Psi}_{k,n_1}(\sqrt{L})(L^{s/2}f)(x_Q)|1_Q \lesi \overline{\Psi}_{k,n_1,\lambda}^*(\sqrt{L})(L^{s/2}f),
\end{equation}
and
\begin{equation}\label{eq - Phi estimate}
\sum_{R\in \mathcal D_k}|{\Phi}_{k,n_2}(\sqrt{L})g(x_R)|1_{R}\lesi {\Phi}^*_{k,n_2,\lambda}(\sqrt{L})(g),
\end{equation}
where
\begin{equation*}
	\overline{\Psi}_{k,n_1,\lambda}^*(\sqrt{L})f(x)=\sup_{y\in X}\f{|\overline{\Psi}_{k,n_1}(\sqrt{L})f(y)|}{(1+2^kd(x,y))^\lambda} \ \ \ \text{and} \ \ \ {\Phi}^*_{k,n_2,\lambda}(\sqrt L)f(x)=\sup_{y\in X}\f{|{\Phi}^*_{k,n_2}(\sqrt{L})f(y)|}{(1+2^kd(x,y))^\lambda}. 
\end{equation*}
Hence, we further imply
\begin{equation*}
\begin{aligned}
	|E_{\ell,k}(x)|&\lesi   2^{(k-\ell) }\sum_{(n_1,n_2)\in\mathbb{Z}^2} |c_k(n_1,n_2)| \\
	& \ \ \times\mathcal M_{r_1,w_1}\Big( |\overline{\Psi}^*_{k,n_1,\lambda}(\sqrt{L})(L^{s/2}f)\Big)(x) \mathcal M_{r_2,w_2}\Big(  {\Phi}^*_{k,n_2,\lambda}(\sqrt{L})g  \Big)(x)\\
	&\lesi  2^{(k-\ell)} \sum_{(n_1,n_2)\in\mathbb{Z}^2} |c_k(n_1,n_2)|\\
	& \ \ \times\mathcal M_{r_1,w_1}\Big( |\overline{\Psi}^*_{k,n_1,\lambda}(\sqrt{L})(L^{s/2}f)\Big)(x) \mathcal M_{r_2,w_2}\Big(  \sup_{j\in \mathbb Z}{\Phi}^*_{j,n_2,\lambda}(\sqrt{L})g  \Big)(x)
\end{aligned}
\end{equation*}
This, together with Proposition \ref{prop1-maximal function}, Theorem \ref{thm1-maximal function} and \eqref{eq-cnn}, further implies
\begin{equation}\label{eq- k lt l}
	\begin{aligned}
		|E_{\ell,k}(x)|  
		&\lesi  2^{(k-\ell)} \sum_{(n_1,n_2)\in\mathbb{Z}^2} |c_k(n_1,n_2)|(1+|n_1|+|n_2|)^{2\lfloor n+\lambda\rfloor +4}\\
		& \ \ \ \times  \mathcal M_{r_1,w_1}\Big( |\overline{\Psi}^*_{k,\lambda}(\sqrt{L})(L^{s/2}f)\Big)(x) \mathcal M_{r_2,w_2}\Big( \sup_{j\in \mathbb Z}{\Phi}^*_{j,n_2,\lambda}(\sqrt{L})g   \Big)(x)\\
		&\lesi  2^{(k-\ell)} 2^{k\gamma} \mathcal M_{r_1,w_1}\Big( |\overline{\Psi}^*_{k,\lambda}(\sqrt{L})(L^{s/2}f)\Big)(x) \mathcal M_{r_2,w_2}\Big(  \sup_{j\in \mathbb Z}{\Phi}^*_{j,n_2,\lambda}(\sqrt{L})g  \Big)(x),
	\end{aligned}
\end{equation}
where $\overline{\Psi}^*_{k,\lambda}(\sqrt{L})$ is the Peetre maximal function associated with $\overline{\Psi}$ defined as in \eqref{eq1-PetreeFunction}.
\bigskip

\textbf{Case 2: $k \ge \ell$}

\medskip

From Proposition \ref{thm-kernel estimate for functional calculus} and Lemmas \ref{lem-kernel-representation} and \ref{lem-composite kernel},  we have, for  a fixed $M>\max\{\f{n\tau_{w_1}}{r_1}, \f{n\tau_{w_2}}{r_2}\}$,
\[
\begin{aligned}
	\Big|\int_X \theta_\ell(\sqrt L)(x,y) &\widetilde{\Gamma}_k(y,x_Q)\widetilde{\Gamma}_k(y,x_R) d\mu(y)\Big|\\
	&\lesi \int_X D_{2^{-\ell},2M+2n}(x,y)D_{2^{-k},M}(y,x_Q)D_{2^{-k},2M+2n}(y,x_R) d\mu(y)\\
	&\lesi D_{2^{-\ell},M}(x,x_Q)D_{2^{-k},M}(x_Q,x_R).
\end{aligned} 
\]
Hence,
\[
\begin{aligned}
	|E_{\ell,k}(x)|&\lesi  \sum_{(n_1,n_2)\in\mathbb{Z}^2} |c_k(n_1,n_2)|\sum_{Q\in \mathcal D_k} \sum_{R\in \mathcal D_k}2^{(\ell-k)s} \omega_Q  |\overline{\Psi}_{k,n_1}(\sqrt{L})(L^{s/2}f)(x_Q)| \\
	& \ \ \ \ \ \times \omega_R |\Phi_{k,n_2}(\sqrt{L})g(x_R)|D_{\ell,M}(x,x_Q)D_{2^{-k},M}(x_Q,x_R)\\
	&\lesi  \sum_{(n_1,n_2)\in\mathbb{Z}^2} |c_k(n_1,n_2)|\sum_{Q\in \mathcal D_k}  2^{(\ell-k)s} \omega_Q  |\overline{\Psi}_{k,n_1}(\sqrt{L})(L^{s/2}f)(x_Q)| \\
	& \ \ \ \ \ \times \mathcal M_{r_2,w_2}\Big( \sum_{R\in \mathcal D_k}|\Phi_{k,n_2}(\sqrt{L})g(x_R)1_{R}| \Big)(x_Q)D_{\ell,N}(x,x_Q),
\end{aligned}
\]
where in the last inequality we used Lemma \ref{lem1- thm2 atom Besov}.

Using  \eqref{eq - Phi estimate} for $\lambda> \max\{n/q, n\tau_w/p, n/\tau_w\} $ and Lemma \ref{lem1- thm2 atom Besov}  again, we further imply
\[
\begin{aligned}	
	|E_{\ell,k}(x)|&\lesi  \sum_{(n_1,n_2)\in\mathbb{Z}^2} |c_k(n_1,n_2)|\sum_{Q\in \mathcal D_k}  2^{(\ell-k)s} \omega_Q  |\overline{\Psi}_{k,n_1}(\sqrt{L})(L^{s/2}f)(x_Q)|\\
	& \ \ \ \ \ \ \times  \mathcal M_{r_2,w_2}\Big( \Phi_{k,n_2,\lambda}^*(\sqrt{L})g \Big)(x_Q)D_{\ell,N}(x,x_Q)\\
	&\lesi  \sum_{(n_1,n_2)\in\mathbb{Z}^2} |c_k(n_1,n_2)|  2^{(\ell-k)s}2^{n(k-\ell)(\tau_w/r-1)} \\
	& \ \ \ \times\mathcal M_{r,w}\Big[ \sum_{Q\in \mathcal D_k}|\overline{\Psi}_{k,n_1}(\sqrt{L})(L^{s/2}f)(x_Q)|  \mathcal M_{r_2,w_2}\Big(  \Phi_{k,n_2,\lambda}^*(\sqrt{L})g \Big)(x_Q) 1_{Q}\Big](x),
\end{aligned}
\]
where in the last inequality we used Lemma \ref{lem1- thm2 atom Besov}.

We will use the following inequality, whose proof is presented in Appendix C,
\begin{equation}\label{eq- equivalence of maximal functions}
\mathcal M_{r_2,w_2}\Big(  \Phi_{k,n_2,\lambda}^*(\sqrt{L})g \Big)(x)\simeq \mathcal M_{r_2,w_2}\Big(  \Phi_{k,n_2,\lambda}^*(\sqrt{L})g \Big)(y) \ \ \text{for $x,y\in Q\in \mathcal D_k$,}
\end{equation}
and \eqref{eq - overlinePhi estimate}, we have
\[
\begin{aligned}
	\sum_{Q\in \mathcal D_k}|\overline{\Psi}_{k,n_1}(\sqrt{L})(L^{s/2}f)(x_Q)| & \mathcal M_{r_2,w_2}\Big(  \Phi_{k,n_2,\lambda}^*(\sqrt{L})g \Big)(x_Q) 1_{Q}(x)\\
	&\lesi \overline{\Psi}^*_{k,n_1,\lambda}(\sqrt{L})(L^{s/2}f)(x) \mathcal M_{r_2,w_2}\Big(  \Phi_{k,n_2,\lambda}^*(\sqrt{L})g \Big)(x), \ \  \ \ x\in X.
\end{aligned}
\]
Consequently,
\begin{equation*} 
\begin{aligned}
	E_{\ell,k}& \lesi   \sum_{(n_1,n_2)\in\mathbb{Z}^2} |c_k(n_1,n_2)|2^{-(k-\ell)[s-n(\tau_w/r-1)]} \\
	& \ \ \times\mathcal M_{r,w}\Big[\overline{\Psi}^*_{k,n_1,\lambda}(\sqrt{L})(L^{s/2}f) \mathcal M_{r_2,w_2}\Big(  \sup_{j\in \mathbb N}\Phi_{j,n_2,\lambda}^*(\sqrt{L})g \Big)\Big](x).
\end{aligned}
\end{equation*}
Using Proposition \ref{prop1-maximal function}, Proposition \ref{thm1-maximal function} and \eqref{eq-cnn} and arguing similarly to \eqref{eq- k lt l}, we also come up with
\begin{equation}\label{eq- k gt l}
	\begin{aligned}
		E_{\ell,k}& \lesi    2^{-(k-\ell)[s-n(\tau_w/r-1)]}2^{k\gamma} \mathcal M_{r,w}\Big[\overline{\Psi}^*_{k,\lambda}(\sqrt{L})(L^{s/2}f) \mathcal M_{r_2,w_2}\Big( \sup_{t>0} (e^{-tL})^*_\lambda g \Big)\Big](x).
	\end{aligned}
\end{equation}

From \eqref{eq- 1st eq BmL}, \eqref{eq- k lt l} and \eqref{eq- k gt l}, 
\[
\begin{aligned}
	|\psi_\ell&(\sqrt L) L^{s/2}\big[B_{\boldsymbol{\mathrm{m}}_{1,k},L}(f,g)\big](x)|\\
	 &\lesi  \sum_{k\in \mathbb Z\atop k<\ell}2^{k-\ell}  \mathcal M_{r_1,w_1}\Big( |\overline{\Psi}^*_{k,\lambda}(\sqrt{L})(L^{s/2}f)\Big)(x) \mathcal M_{r_2,w_2}\Big(  \sup_{t>0} (e^{-tL})^*_\lambda g  \Big)(x)\\
	 &\  \ \ +  \sum_{k\in \mathbb Z\atop k\ge \ell} 2^{-(k-\ell)[s-n(\tau_w/r-1)]} \mathcal M_{r,w}\Big[\overline{\Psi}^*_{k,\lambda}(\sqrt{L})(L^{s/2}f) \mathcal M_{r_2,w_2}\Big( \sup_{t>0} (e^{-tL})^*_\lambda g \Big)\Big](x).
\end{aligned}
\]
Since $s>\tau_{r,q}(w):=n(\tau_w/r-1)$, it follows that 
\[
\begin{aligned}
		\Big(\sum_{\ell\in\mathbb Z}|\psi_\ell&(\sqrt L) L^{s/2}\big[B_{\boldsymbol{\mathrm{m}}_{1,k},L}(f,g)\big](x)|^q\Big)^{1/q}\\
	&\lesi  \Big(\sum_{k\in \mathbb Z}  2^{k\gamma q} \Big| \mathcal M_{r_1,w_1}\Big( |\overline{\Psi}^*_{k,\lambda}(\sqrt{L})(L^{s/2}f)\Big)(x) \mathcal M_{r_2,w_2}\Big( \sup_{t>0} (e^{-tL})^*_\lambda g  \Big)(x)\Big|^q\Big)^{1/q}\\
	&\ \ + \Big(\sum_{k\in \mathbb Z}  2^{k\gamma q}\Big|   \mathcal M_{r,w}\Big[\overline{\Psi}^*_{k,\lambda}(\sqrt{L})(L^{s/2}f) \mathcal M_{r_2,w_2}\Big(  \sup_{t>0} (e^{-tL})^*_\lambda g \Big)\Big](x)\Big|^q\Big)^{1/q}.
\end{aligned}
\]
By the maximal inequality \eqref{FSIn}, the H\"older inequality and Proposition \ref{prop1-maximal function}, we further imply
\[
\begin{aligned}
\Big\|	\Big(\sum_{\ell\in\mathbb Z}|\psi_\ell&(\sqrt L) L^{s/2}\big[B_{\boldsymbol{\mathrm{m}}_{1,k},L}(f,g)\big](x)|^q\Big)^{1/q}\Big\|_{L^p_w}\\
	&\lesi  \Big\|\Big(\sum_{k\in \mathbb Z}  2^{k\gamma q} \Big| \mathcal M_{r_1,w_1}\Big( |\overline{\Psi}^*_{k,\lambda}(\sqrt{L})(L^{s/2}f)\Big)\Big|^q\Big)^{1/q}\Big\|_{L^{p_1}(w_1)}\Big\| \mathcal M_{r_2,w_2}\Big( \sup_{t>0} (e^{-tL})^*_\lambda g  \Big)\Big\|_{L^{p_2}(w_2)}\\
&\ \ + \Big\|\Big(\sum_{k\in \mathbb Z}  2^{k\gamma q}\Big|   \overline{\Psi}^*_{k,\lambda}(\sqrt{L})(L^{s/2}f) \mathcal M_{r_2,w_2}\Big(  \sup_{t>0} (e^{-tL})^*_\lambda g\Big) \Big|^q\Big)^{1/q}\Big\|_{L^p(w)}\\
	&\lesi  \Big\|\Big(\sum_{k\in \mathbb Z}  2^{k\gamma q} \Big|  \overline{\Psi}^*_{k,\lambda}(\sqrt{L})(L^{s/2}f \Big|^q\Big)^{1/q}\Big\|_{L^{p_1}(w_1)}\Big\| \mathcal M_{r_2,w_2}\Big( \sup_{t>0} (e^{-tL})^*_\lambda g  \Big)\Big\|_{L^{p_2}(w_2)}\\
	&\lesi  \Big\|\Big(\sum_{k\in \mathbb Z}  2^{k\gamma q} \Big|  \Psi^*_{k,\lambda}(\sqrt{L})(L^{s/2}f \Big|^q\Big)^{1/q}\Big\|_{L^{p_1}(w_1)}\Big\| \sup_{t>0} (e^{-tL})^*_\lambda g   \Big\|_{L^{p_2}(w_2)}.
\end{aligned}
\]
At this stage, applying Propositions \ref{prop- Ls on TL B spaces}, \ref{prop2-thm1} and \ref{thm1-maximal function} we arrive at
\[
\begin{aligned}
	 \|B_{\boldsymbol{\mathrm{m}}_{1,k},L}(f,g)\|_{\dot{F}^{s,L}_{p,q}(w)}&\simeq \|L^{s/2}\big[B_{\boldsymbol{\mathrm{m}}_{1,k},L}(f,g)\big]\|_{\dot{F}^{0,L}_{p,q}(w)}\\
	&\lesi \|L^{s/2}f\|_{\dot{F}^{\gamma,L}_{p_1,q}(w_1)}\|g\|_{H^{p_2}_L(w_2)}\\
	&\simeq \|f\|_{\dot{F}^{s+\gamma,L}_{p_1,q}(w_1)}\|g\|_{H^{p_2}_L(w_2)}\\
\end{aligned}
\]
Similarly,
\[
\begin{aligned}
	\|B_{\boldsymbol{\mathrm{m}}_{2,k},L}(f,g)\|_{\dot{F}^{s,L}_{p,q}(w)}
	&\lesi \|f\|_{H^{p_3}_L(w_3)}\|g\|_{\dot{F}^{s+\gamma,L}_{p_4,q}(w_4)}.
\end{aligned}
\]
This completes our proof.
\end{proof}

\begin{proof}[Proof of Theorem \ref{main thm - without A3}:] By a careful examination of the proof of Theorem \ref{main thm}, we see that the condition (A3) is used only in \eqref{eq- the use of A3}. Hence, it suffices to prove that for $\ell \ge k$,  
	\[
	\Big|\int_X \theta_\ell(\sqrt L)(x,y) \widetilde{\Gamma}_k(y,x_Q)\widetilde{\Gamma}_k(y,x_R)\, d\mu(y)\Big|\lesi_N 2^{-\delta_0(\ell-k)}D_{2^{-k},N}(x,x_Q)D_{2^{-k},N}(x,x_R)
	\]
	for all $x\in X$, all $Q,R\in \mathcal D_k$, and $N>0$, where $\theta$ is defined by \eqref{eq-theta} with $\psi$ is a partition of unity.
	
	To this end, we write
	\[
	\begin{aligned}
		\Big|\int_{X} \theta_\ell(\sqrt L)(x,y) \widetilde{\Gamma}_k(y,x_Q)\widetilde{\Gamma}_k(y,x_R)\, d\mu(y)\Big|
		&\le \Big|\int_{d(x,y)\ge 2^{-k}}\ldots\Big|+\Big|\int_{d(x,y)< 2^{-k}}\ldots\Big| \\
		&=: E_1+E_2.
	\end{aligned}
	\]
	
	Using Proposition \ref{thm-kernel estimate for functional calculus} and Lemma \ref{lem-kernel-representation},  
	\[
	\begin{aligned}
		E_1&\lesi \int_{d(x,y)\ge 2^{-k}} D_{2^{-\ell},2N+2n}(x,y) D_{2^{-k},2N+2n}(y,x_Q)D_{2^{-k},2N+2n}(y,x_R)\, d\mu(y) \\
		&\lesi 2^{-\delta_0(\ell-k)}\int_X D_{2^{-\ell},2N+2n}(x,y) D_{2^{-k},2N+2n}(y,x_Q)D_{2^{-k},2N+2n}(y,x_R)\, d\mu(y) \\
		&\lesi 2^{-\delta_0(\ell-k)}D_{2^{-k},N}(x,x_Q)D_{2^{-k},N}(x,x_R),
	\end{aligned}
	\]
	where in the last inequality we used Lemma \ref{lem-composite kernel}.
	
	For the second term $E_2$, we further decompose it as
	\[
	\begin{aligned}
		E_2&\le \Big|\int_{d(x,y)< 2^{-k}} \theta_\ell(\sqrt L)(x,y)\, \widetilde{\Gamma}_k(y,x_Q)\,[\widetilde{\Gamma}_k(y,x_R)-\widetilde{\Gamma}_k(x,x_R)]\, d\mu(y)\Big| \\
		&\quad + \Big|\int_{d(x,y)< 2^{-k}} \theta_\ell(\sqrt L)(x,y)\, \widetilde{\Gamma}_k(y,x_Q)\,\widetilde{\Gamma}_k(x,x_R)\, d\mu(y)\Big| \\
		&\le \Big|\int_{d(x,y)< 2^{-k}} \theta_\ell(\sqrt L)(x,y)\, \widetilde{\Gamma}_k(y,x_Q)\,[\widetilde{\Gamma}_k(y,x_R)-\widetilde{\Gamma}_k(x,x_R)]\, d\mu(y)\Big| \\
		&\quad + \Big|\int_{X} \theta_\ell(\sqrt L)(x,y)\, \widetilde{\Gamma}_k(y,x_Q)\,\widetilde{\Gamma}_k(x,x_R)\, d\mu(y)\Big| \\
		&\quad + \Big|\int_{d(x,y)\ge 2^{-k}} \theta_\ell(\sqrt L)(x,y)\, \widetilde{\Gamma}_k(y,x_Q)\,\widetilde{\Gamma}_k(x,x_R)\, d\mu(y)\Big| \\
		&=:E_{21}+E_{22}+E_{23}.
	\end{aligned}
	\]
	
	Using Theorem \ref{thm-kernel estimate for functional calculus} together with Lemmas \ref{lem-kernel-representation} and \ref{lem-composite kernel},  
	\[
	\begin{aligned}
		E_{21}&\lesi (2^kd(x,y))^{\delta_0}\int_X D_{2^{-\ell},2N+2n}(x,y) D_{2^{-k},2N+2n}(y,x_Q)D_{2^{-k},2N+2n}(y,x_R)\, d\mu(y) \\
		&\lesi 2^{-\delta_0(\ell-k)}\int_X D_{2^{-\ell},2N+2n}(x,y) D_{2^{-k},2N+2n}(y,x_Q)D_{2^{-k},2N+2n}(y,x_R)\, d\mu(y) \\
		&\lesi 2^{-\delta_0(\ell-k)}D_{2^{-k},N}(x,x_Q)D_{2^{-k},N}(x,x_R).
	\end{aligned}
	\]
	
	Similarly, by Theorem \ref{thm-kernel estimate for functional calculus} and Lemmas \ref{lem-kernel-representation} and \ref{lem-Ds Dt},  
	\[
	\begin{aligned}
		E_{23}&\lesi \int_{d(x,y)\ge 2^{-k}} D_{2^{-\ell},N+n+\delta_0}(x,y) D_{2^{-k},N+n}(y,x_Q)D_{2^{-k},N}(x,x_R)\, d\mu(y) \\
		&\lesi 2^{-\delta_0(\ell-k)}D_{2^{-k},N}(x,x_R)\int_X D_{2^{-\ell},N+n}(x,y) D_{2^{-k},N+n}(y,x_Q)\, d\mu(y) \\
		&\lesi 2^{-\delta_0(\ell-k)}D_{2^{-k},N}(x,x_Q)D_{2^{-k},N}(x,x_R).
	\end{aligned}
	\]
	
	For the remaining term $E_{22}$, we write
	\[
	\begin{aligned}
		E_{22}&=|\widetilde{\Gamma}_k(x,x_R)|\Big|\int_{X} \theta_\ell(\sqrt L)(y,x)\, \widetilde{\Gamma}_k(y,x_Q)\, d\mu(y)\Big| \\
		&=2^{-2\ell}|\widetilde{\Gamma}_k(x,x_R)|\Big|\int_{X} (2^{-2\ell}L)^{-1}\theta_\ell(\sqrt L)(y,x)\, L\widetilde{\Gamma}_k(y,x_Q)\, d\mu(y)\Big|.
	\end{aligned}
	\]
	
	Applying Theorem \ref{thm-kernel estimate for functional calculus} and Lemmas \ref{lem-kernel-representation} and \ref{lem-Ds Dt}, we obtain
	\[
	\begin{aligned}
		E_{22}&\lesi 2^{-2(\ell-k)}D_{2^{-k},N}(x,x_R)\Big|\int_{X} D_{2^{-\ell},N+n}(x,y) D_{2^{-k},N+n}(y,x_Q) \, d\mu(y)\Big| \\
		&\lesi 2^{-2(\ell-k)}D_{2^{-k},N}(x,x_Q)D_{2^{-k},N}(x,x_R).
	\end{aligned}
	\]
	
	This completes the proof.
	
\end{proof}

\section{Applications}
In this section, we discuss some applications of our main results, focusing on Theorem \ref{main thm} and Corollary \ref{cor-main}, since several examples of operators covered by Theorem \ref{main thm - without A3} were given in the introduction. The list below is not exhaustive but illustrates the flexibility and effectiveness of our method.

\subsection{Fractional Leibniz Rule on nilpotent Lie groups} 

We start by recalling some standard definitions and facts concerning nilpotent Lie groups (see \cite{G,NSW,VSC}).  
Throughout, let $G$ denote a connected, simply connected nilpotent Lie group, and let $\mu$ be its bi-invariant Haar measure.  

Consider a finite family of left-invariant vector fields  
\[
\mathbf{X} = \{X_1,\ldots,X_\ell\}
\]
on $G$ that satisfies H\"ormander’s condition, meaning that $\{X_1,\ldots,X_\ell\}$ together with their commutators up to order $k_0$ generate the tangent space at every point of $G$.  
The differential operator associated with this system,  
\begin{equation}\label{eq1.1}
	L := -\sum_{i=1}^\ell X_i^2,
\end{equation}
is referred to as the \emph{sub-Laplacian}.  

The system $\mathbf{X}$ naturally defines a control metric $\rho$ on $G$, which is left-invariant and induces the topology of $G$.  
For $x \in G$, we set $|x| := \rho(x,e)$, where $e$ denotes the identity.  
By left-invariance,  
\[
\rho(x,y) = |y^{-1}x| = |x^{-1}y|, \qquad x,y \in G.
\]  

For $r>0$, let $B(x,r)$ denote the open $\rho$-ball of radius $r$ centered at $x$, and write $V(x,r) := \mu(B(x,r))$.  
There exists a constant $C_1 > 0$ such that for all $x \in G$,  
\begin{equation}\label{eq1.2}
	C_1^{-1} r^d \leq V(x,r) \leq C_1 r^d, \qquad r \in (0,1],
\end{equation}
and  
\begin{equation}\label{eq1.3}
	C_1^{-1} r^D \leq V(x,r) \leq C_1 r^D, \qquad r \in [1,\infty),
\end{equation}
where the constants $d>0$ and $D>0$ are called the \emph{local dimension} and the \emph{dimension at infinity} of $G$, respectively.  
Since $G$ is simply connected, one always has $0 < d \leq D$ (see \cite[Remark IV.5.9]{VSC}).  

In particular, the measure $\mu$ obeys the doubling property \eqref{doubling2} with $n = D$, so that $(G,\rho,\mu)$ is a homogeneous type space in the sense of Coifman--Weiss.  

For $k \in \mathbb{N}$, define the set of multi-indices  
\[
\mathcal{I}(\ell,k) := \{1,\ldots,\ell\}^k = \{(i_1,\ldots,i_k) : 1 \leq i_j \leq \ell \ \text{for all } j\}.
\]  
By convention, let $\mathcal{I}(\ell,0) := \{(0)\}$, and set  
\[
\mathcal{I}(\ell) := \bigcup_{k \in \mathbb{N}_0} \mathcal{I}(\ell,k).
\]  

Because $\mathbf{X}$ satisfies H\"ormander’s condition, the collection $\{X^I : I \in \mathcal{I}(\ell)\}$ spans the algebra of left-invariant differential operators on $G$.  
For a multi-index $I = (i_1,\ldots,i_k) \in \mathcal{I}(\ell,k)$, define  
\[
X^I :=
\begin{cases}
	X_{i_1} \cdots X_{i_k}, & k \geq 1, \\
	\mathrm{Id}, & k=0,
\end{cases}
\]
and write $|I| := k$ for its length.

\begin{prop}\label{prop-derivative of heat kernel on Lie group} {\rm (\cite[Theorem IV.4.2]{VSC})}  
	Let $h_{t}(x,y)$ be the heat kernel of the sub-Laplacian $L$. Then, for every $I \in \mathcal{I}(\ell)$,
	\[
	| X^I_x h_t (x,y)| \leq C \frac{1}{t^{|I|/2}V(x,\sqrt t)} \exp \left(-\frac{|y^{-1}x|}{c t}\right)
	\]
	for all $x,y \in G$ and $t > 0$.  
\end{prop}

\begin{thm}\label{main thm-nilpotent}
	Let $L$ be a sub-Laplacian on a stratified Lie group $G$, and let $\boldsymbol{\mathrm{m}}$ satisfy \eqref{eq-condition on m} for some $\gamma \in \mathbb{R}$. Let $0<q\le \vc$, and let $0<p,p_1,p_2,p_3,p_4 \le \infty$ satisfy \eqref{eq- pi condition}, and $w, w_1, w_2, w_3, w_4 \in A_\infty$ satisfy \eqref{eq- wi condition}.
	
	If $0<p_1,p_4<\infty$, $0<p_2,p_3 \le \infty$, and $s > \tau_{p,q}(w)$, then
	\begin{equation}
		\|B_{\boldsymbol{\mathrm{m}},L}(f,g)\|_{\dot{F}^{s,L}_{p,q}(w)} 
		\lesssim \|f\|_{\dot{F}^{s+\gamma,L}_{p_1,q}(w_1)} \|g\|_{H^{p_2}_L(w_2)} 
		+ \|f\|_{H^{p_3}_L(w_3)} \|g\|_{\dot{F}^{s+\gamma,L}_{p_4,q}(w_4)}.
	\end{equation}
	
	Similarly, for the Besov scale, if $0<p,p_1,p_2,p_3,p_4 \le \infty$ and $s > \tau_p(w)$, it holds that
	\begin{equation}
		\|B_{\boldsymbol{\mathrm{m}},L}(f,g)\|_{\dot{B}^{s,L}_{p,q}(w)} 
		\lesssim \|f\|_{\dot{B}^{s+\gamma,L}_{p_1,q}(w_1)} \|g\|_{H^{p_2}_L(w_2)} 
		+ \|f\|_{H^{p_3}_L(w_3)} \|g\|_{\dot{B}^{s+\gamma,L}_{p_4,q}(w_4)}.
	\end{equation}
	
	In both cases, the Hardy spaces $H^{p_2}_L(w_2)$ and $H^{p_3}_L(w_3)$ are replaced by $L^\infty$ if $p_2=\infty, w_2=1$ or $p_3=\infty, w_3=1$, respectively.
\end{thm}

\begin{proof}
	Obviously, by Proposition \ref{prop-derivative of heat kernel on Lie group}, the sub-Laplacian $L$ satisfies conditions (A1) and (A2). 
	For (A3), note that by Proposition \ref{thm-kernel estimate for functional calculus}, for any even function $\varphi \in \mathscr S(\mathbb R)$, 
	\[
	|\varphi(t\sqrt{L})(x,y)| \lesssim_{N} D_{t,N}(x,y),
	\]  
	for all $t,N>0$ and $x,y \in G$.  
	
	Combining this with Proposition \ref{prop-derivative of heat kernel on Lie group}, and following the argument in the proof of Proposition~3.2 of \cite{BD}, we obtain that for each $I \in \mathcal I(\ell)$,  
	\[
	|X^I \varphi(t\sqrt{L})(x,y)| \lesssim_{N,I} t^{-|I|} D_{t,N}(x,y),
	\]  
	for all $t,N>0$ and $x,y \in G$.  
	
	This estimate, together with the product rule, shows that condition (A3) is satisfied. Hence, the result follows directly from Theorem \ref{main thm}.  
	
\end{proof}

\begin{cor}\label{cor-nilpotent}
	Under the same conditions as in Theorem \ref{main thm-nilpotent}, we have:
	\begin{enumerate}[\rm (a)]
		\item If $0<p_1,p_4<\infty$, $0<p_2,p_3 \le \infty$, and $s > \tau_{p,q}(w)$, then
	\begin{equation}
		\|L^{s/2}(fg)\|_{\dot{F}^{0,L}_{p,q}(w)} \simeq \|fg\|_{\dot{F}^{s,L}_{p,q}(w)} 
		\lesssim \|f\|_{\dot{F}^{s,L}_{p_1,q}(w_1)} \|g\|_{H^{p_2}_L(w_2)} 
		+ \|f\|_{H^{p_3}_L(w_3)} \|g\|_{\dot{F}^{s,L}_{p_4,q}(w_4)}.
	\end{equation}
	
	\item Similarly, for the Besov scale, if $0<p,p_1,p_2,p_3,p_4 \le \infty$ and $s > \tau_p(w)$, it holds that
	\begin{equation}
		\|L^{s/2}(fg)\|_{\dot{B}^{0,L}_{p,q}(w)} \simeq \|fg\|_{\dot{B}^{s,L}_{p,q}(w)} 
		\lesssim \|f\|_{\dot{B}^{s,L}_{p_1,q}(w_1)} \|g\|_{H^{p_2}_L(w_2)} 
		+ \|f\|_{H^{p_3}_L(w_3)} \|g\|_{\dot{B}^{s,L}_{p_4,q}(w_4)}.
	\end{equation}
\end{enumerate}
	In both inequalities, the Hardy spaces $H^{p_2}_L(w_2)$ and $H^{p_3}_L(w_3)$ are replaced by $L^\infty$ if $p_2=\infty, w_2=1$ or $p_3=\infty, w_3=1$, respectively.
\end{cor}

\noindent \textbf{Fractional Leibniz rules on stratified Lie groups:} We now consider an important subclass of nilpotent Lie groups, which is stratified Lie groups. A connected and simply connected Lie group is said to be \textbf{stratified} if its Lie algebra $\mathfrak{g}$ admits a stratification, i.e., a direct sum decomposition

\[
\mathfrak{g} = V_1 \oplus V_2 \oplus \cdots \oplus V_\kappa
\]

such that

\begin{equation}\label{eq-basic stratified Lie group}
[V_1, V_{i-1}] = V_i \quad \text{for } 2 \leq i \leq \kappa, \qquad [V_1, V_\kappa] = \{0\}. \tag{3.22}
\end{equation}

Obviously, such a Lie group is nilpotent of step $\kappa$. Let $m = \dim(V_1)$, and let $\{X_1, \ldots, X_m\}$ be an arbitrary basis of $V_1$. Identifying $\mathfrak{g}$ with the Lie algebra of all left-invariant vector fields on $G$, it follows from \eqref{eq-basic stratified Lie group} that the system $\{X_1, \ldots, X_m\}$ satisfies the H\"ormander condition. Consider the sub-Laplacian operator
\[
L := -\sum_{i=1}^{m} X_i^2,
\]
which we still denote by $L$ without risk of confusion.

A key feature of these spaces is that the weighted Besov and Triebel-Lizorkin spaces are independent of the choices of the sub-Laplacians. More precisely, we have the following result.

\begin{prop}[\cite{Hu}]
	Let $G$ be a stratified Lie group, and let $L$ and $\widetilde{L}$ be two arbitrary sub-Laplacians on $G$. Then:
	\begin{enumerate}[(i)]
		\item For $p,q \in (0,\infty]$, $s \in \mathbb{R}$, and $w \in A_\infty$,
		\[
		\dot{B}^{s,L}_{p,q}(w) = \dot{B}^{s,\widetilde{L}}_{p,q}(w)
		\]
		with equivalent quasi-norms.
		
		\item For $p\in (0,\vc), q \in (0,\infty]$, $s \in \mathbb{R}$, and $w \in A_\infty$,
		\[
		\dot{F}^{s,L}_{p,q}(w) = \dot{F}^{s,\widetilde{L}}_{p,q}(w)
		\]
		with equivalent quasi-norms.
	\end{enumerate}
\end{prop}
In fact, \cite[Theorem 7]{Hu} proved the coincidence of the unweighted spaces; however, the same argument also works for the weighted case.  Hence, from the above result, we can omit the symbol $L$ in the notation of the Besov and Tribel-Lizorkin spaces associated to $L$ to briefly write $\dot{B}^{s}_{p,q}(w)$ and $\dot{F}^{s}_{p,q}(w)$. It is important to notice that in the particular case $w=1$, $\dot{F}^{0}_{p,2}(w)$ becomes the Hardy spaces $H^p$ introduced in \cite{FS}. Hence, from Theorem \ref{main thm-nilpotent Schrodinger} and Corollary \ref{cor-nilpotent  Schrodinger} we deduce the following results.

\begin{cor}\label{main thm-stratified group}
	Let  $L$ be a sub-Laplacian on the stratified Lie group $G$, and let $\boldsymbol{\mathrm{m}}$ satisfy \eqref{eq-condition on m} for some $\gamma\in \mathbb{R}$. Let $0<q\le \vc$, and let $0<p,p_1,p_2,p_3,p_4 \le \infty$ satisfy \eqref{eq- pi condition}, and $w, w_1, w_2, w_3, w_4 \in A_\infty$ satisfy \eqref{eq- wi condition}. 
	
	If $0<p_1,p_4<\infty$, $0<p_2,p_3 \le \infty$, and $s > \tau_{p,q}(w)$, then
	\begin{equation}
		\|B_{\boldsymbol{\mathrm{m}}}(f,g)\|_{\dot{F}^{s}_{p,q}(w)} 
		\lesssim \|f\|_{\dot{F}^{s+\gamma}_{p_1,q}(w_1)} \|g\|_{H^{p_2}(w_2)} 
		+ \|f\|_{H^{p_3}(w_3)} \|g\|_{\dot{F}^{s+\gamma}_{p_4,q}(w_4)}.
	\end{equation}
	
	Similarly, for the Besov scale, if $0<p,p_1,p_2,p_3,p_4 \le \infty$ and $s > \tau_p(w)$, it holds that
	\begin{equation}
		\|B_{\boldsymbol{\mathrm{m}},L}(f,g)\|_{\dot{B}^{s}_{p,q}(w)} 
		\lesssim \|f\|_{\dot{B}^{s+\gamma}_{p_1,q}(w_1)} \|g\|_{H^{p_2}(w_2)} 
		+ \|f\|_{H^{p_3}(w_3)} \|g\|_{\dot{B}^{s+\gamma}_{p_4,q}(w_4)}.
	\end{equation}
	
	In both cases, the Hardy spaces $H^{p_2}(w_2)$ and $H^{p_3}(w_3)$ are replaced by $L^\infty$ if $p_2=\infty, w_2=1$ or $p_3=\infty, w_3=1$, respectively.
\end{cor}

\begin{cor}\label{cor-stratified group}
	Under the same conditions as in Theorem \ref{main thm-stratified group}, we have:
	\begin{enumerate}[\rm (a)]
		\item If $0<p_1,p_4<\infty$, $0<p_2,p_3 \le \infty$, and $s > \tau_{p,q}(w)$, then
	\begin{equation}
		\|L^{s/2}(fg)\|_{\dot{F}^{0}_{p,q}(w)} \simeq \|fg\|_{\dot{F}^{s}_{p,q}(w)} 
		\lesssim \|f\|_{\dot{F}^{s}_{p_1,q}(w_1)} \|g\|_{H^{p_2}(w_2)} 
		+ \|f\|_{H^{p_3}(w_3)} \|g\|_{\dot{F}^{s}_{p_4,q}(w_4)}.
	\end{equation}
	
	\item Similarly, for the Besov scale, if $0<p,p_1,p_2,p_3,p_4 \le \infty$ and $s > \tau_p(w)$, it holds that
	\begin{equation}
		\|L^{s/2}(fg)\|_{\dot{B}^{0}_{p,q}(w)} \simeq \|fg\|_{\dot{B}^{s}_{p,q}(w)} 
		\lesssim \|f\|_{\dot{B}^{s}_{p_1,q}(w_1)} \|g\|_{H^{p_2}(w_2)} 
		+ \|f\|_{H^{p_3}(w_3)} \|g\|_{\dot{B}^{s}_{p_4,q}(w_4)}.
	\end{equation}
\end{enumerate}
	In both cases, the Hardy spaces $H^{p_2}(w_2)$ and $H^{p_3}(w_3)$ are replaced by $L^\infty$ if $p_2=\infty, w_2=1$ or $p_3=\infty, w_3=1$, respectively.
\end{cor}

One more interesting application is about the Sch\"odinger operator on the stratified group. Let $V$ is a nonnegative polynomial on the stratified Lie group $G$. We consider the Schr\"odinger operator $L_V = L +V$. Define
\begin{equation}\label{eq-rho function}
\rho(x)=\sup\left\{r>0: \f{1}{r^{n-2}}\int_{B(x,r)}V(y)dy\leq 1\right\}
\end{equation}
is a {\it critical function} (see, for example, \cite{F,Shen, BHH}).

\begin{prop}[\cite{BHH}]\label{derivatives of heat kernel Schrodinger}
	Let $p_{t}(x,y)$ be the heat kernel of $L_V$.
	For any $I \in \mathcal I(m)$, and $N >0$, there are constants $C,c>0$ so that
	\begin{equation} \label{derivative}
		|  X_x^I  p_t(x,y)|
		\lesi_{N,I} C t^{-\frac{Q+|I|}{2}} e^{-\frac{|y^{-1}x|^2}{c t}} \left( 1+\frac{\sqrt{t}}{\rho(x)} +\frac{\sqrt{t}}{\rho(y)}\right)^{-N},
	\end{equation}
	where the subscript $x$ in $X_x^I$ means a derivative with respect to the variable $x$. 
\end{prop}

\bigskip

\begin{thm}\label{main thm-nilpotent Schrodinger}
	Let $G$ be a stratified Lie group and let $L_V = L + V$, where $V$ is a nonnegative polynomial on $G$. Let $\boldsymbol{\mathrm{m}}$ satisfy \eqref{eq-condition on m} for some $\gamma \in \mathbb{R}$. Suppose $0<q\le \vc$, and   $0<p,p_1,p_2,p_3,p_4 \le \infty$ satisfies \eqref{eq- pi condition}, and $w, w_1, w_2, w_3, w_4 \in A_\infty$ satisfies \eqref{eq- wi condition}. 
	
	If $0<p_1,p_4<\infty$, $0<p_2,p_3 \le \infty$, and $s > \tau_{p,q}(w)$, then
	\begin{equation}
		\|B_{\boldsymbol{\mathrm{m}},L_V}(f,g)\|_{\dot{F}^{s,L_V}_{p,q}(w)} 
		\lesssim \|f\|_{\dot{F}^{s+\gamma,L_V}_{p_1,q}(w_1)} \|g\|_{H^{p_2}_{L_V}(w_2)} 
		+ \|f\|_{H^{p_3}_{L_V}(w_3)} \|g\|_{\dot{F}^{s+\gamma,L_V}_{p_4,q}(w_4)}.
	\end{equation}
	
	Similarly, for the Besov scale, if $0<p,p_1,p_2,p_3,p_4 \le \infty$ and $s > \tau_p(w)$, it holds that
	\begin{equation}
		\|B_{\boldsymbol{\mathrm{m}},L_V}(f,g)\|_{\dot{B}^{s,L_V}_{p,q}(w)} 
		\lesssim \|f\|_{\dot{B}^{s+\gamma,L_V}_{p_1,q}(w_1)} \|g\|_{H^{p_2}_{L_V}(w_2)} 
		+ \|f\|_{H^{p_3}_{L_V}(w_3)} \|g\|_{\dot{B}^{s+\gamma,L_V}_{p_4,q}(w_4)}.
	\end{equation}
	
	In both cases, the Hardy spaces $H^{p_2}_{L_V}(w_2)$ and $H^{p_3}_{L_V}(w_3)$ are replaced by $L^\infty$ if $p_2=\infty, w_2=1$ or $p_3=\infty, w_3=1$, respectively.
\end{thm}

\begin{proof}
	We first recall the inequality (5.3) in \cite{BHH} that
	\begin{equation}\label{eq-bound for V}
		|X^I V(x)|\lesi \rho(x)^{-|I|-2}
	\end{equation}	
	for all $x\in X$ and $I\in \mathcal I(m)$.
	
	By Proposition \ref{derivatives of heat kernel Schrodinger}, the Schr\"odinger operartor $L_V$ satisfies conditions (A1) and (A2). It remains to verify (A3), we first note that from Proposition \ref{derivatives of heat kernel Schrodinger} and \eqref{eq-bound for V} we have, for each $I, J\in \mathcal I(m)$ and $j\in \mathbb N$,
	\[
	|(X^{I} V^j)X^{J}p_t(x,y)|\lesi t^{-\frac{Q+|I|+|J|+2j}{2}} e^{-\frac{|y^{-1}x|^2}{c t}} \left( 1+\frac{\sqrt{t}}{\rho(x)} +\frac{\sqrt{t}}{\rho(y)}\right)^{-N}.
	\]
	This, together with Proposition~\ref{thm-kernel estimate for functional calculus} 
	and the argument used in the proof of Proposition~3.2 of \cite{BD}, 
	implies that for any even function $\varphi \in \mathscr S(\mathbb R)$, 
	$I, J \in \mathcal I(m)$, and $j \in \mathbb N$, we have
	\[
	|(X^{I} V^j )X^{J} \varphi(t\sqrt{L})(x,y)|
	\lesssim t^{-\frac{Q+|I|+|J|+2j}{2}} \, D_{t,N}(x,y),
	\]
	for all $t,N>0$ and $x,y \in G$.
	
	At this stage, the condition (A3) follows directly from the above estimate, the product rule and the following identity
	\begin{equation*}
		(\mathcal L + V)^m 
			= \sum_{j=0}^m a_{m,j} 
			\sum_{|I|\le 2(m-j)} 
			b_{m,j,I} \, (X^I V^j)\, X^{2(m-j)-|I|},
	\end{equation*}
where $X^{2(m-j)-|I|}$ stands for a homogeneous differential operator of degree $2(m-j)-|I|$ built from the $X_1, \ldots, X_\ell$ and $a_{m,j}, b_{m,j,I}$ are constants.

Hence, the result is a direct consequence of Theorem \ref{main thm}.  This completes our proof.
	
\end{proof}

\begin{cor}\label{cor-nilpotent Schrodinger}
	Under the same conditions as in Theorem \ref{main thm-nilpotent Schrodinger}, we have:
	\begin{enumerate}[\rm (a)]
		\item If $0<p_1,p_4<\infty$, $0<p_2,p_3 \le \infty$, and $s > \tau_{p,q}(w)$, then
	\begin{equation}
		\|fg\|_{\dot{F}^{s,L_V}_{p,q}(w)} 
		\lesssim \|f\|_{\dot{F}^{s,L_V}_{p_1,q}(w_1)} \|g\|_{H^{p_2}_{L_V}(w_2)} 
		+ \|f\|_{H^{p_3}_{L_V}(w_3)} \|g\|_{\dot{F}^{s,L_V}_{p_4,q}(w_4)}.
	\end{equation}
	
	\item Similarly, for the Besov scale, if $0<p,p_1,p_2,p_3,p_4 \le \infty$ and $s > \tau_p(w)$, it holds that
	\begin{equation}
		\|fg\|_{\dot{B}^{s,L_V}_{p,q}(w)} 
		\lesssim \|f\|_{\dot{B}^{s,L_V}_{p_1,q}(w_1)} \|g\|_{H^{p_2}_{L_V}(w_2)} 
		+ \|f\|_{H^{p_3}_{L_V}(w_3)} \|g\|_{\dot{B}^{s,L_V}_{p_4,q}(w_4)}.
	\end{equation}
	\end{enumerate}
	In both cases, the Hardy spaces $H^{p_2}_{L_V}(w_2)$ and $H^{p_3}_{L_V}(w_3)$ are replaced by $L^\infty$ if $p_2=\infty, w_2=1$ or $p_3=\infty, w_3=1$, respectively.
	
\end{cor}

In particular case when $s=0$, $q=2$ and $p_i\in (0,\vc), w_i=1,  i=1,2,3,4$, the estimate in Corollary \ref{cor-nilpotent Schrodinger} (a) reads
\[
\|fg\|_{H^p_{L_V}} 
\lesssim \|f\|_{H^{p_1}_{L_V}} \|g\|_{H^{p_2}_{L_V}} 
+ \|f\|_{H^{p_3}_{L_V}} \|g\|_{H^{p_4}_{L_V}},
\]
where $1/p=1/p_1+1/p_2=1/p_3+1/p_4$.

By Remark \ref{rem2}, $H^p_{L_V}=L^p$ for $1<p<\vc$. When $0<p\le 1$, the Hardy space $H^p_{L_V}$ can be characterized as follows.  See for example \cite{BHH}.
\begin{defn}\label{def: rho atoms}
	Let $\rho$ be an arbitrary critical function on $G$ as in  \eqref{eq-rho function}.	Let $p\in (0,1]$ and $q\in [1,\infty]\cap (p,\infty]$.
	A function $a$ is called a  $(p,q,\rho)$-atom associated to the ball $B=B(x_B,r_B)$ if
	\begin{enumerate}[\rm (i)]
		\item ${\rm supp}\, a\subset B$;
		\item $\|a\|_{L^q(G)}\leq |B|^{1/q-1/p}$;
		\item $\displaystyle \int_G a(x)P(x)d\mu(x) =0$  for all polynomial $P$ with the degree at most $N_p: = \lfloor Q(1/p-1)\rfloor$, if $r_B< \rho(x_B)/4$.
	\end{enumerate}
\end{defn}

\begin{defn} \label{local atomic v2}
	Let $\rho$ be an arbitrary critical function on $G$ as in \eqref{eq-rho function}. Let $p\in (0,1]$ and $q\in [1,\infty]\cap (p,\infty]$.
	We define the atomic Hardy space $H^{p,q}_{{\rm at},\rho}(G)$ as the completion of the set of functions $f$ of the form  $f=\sum_{j=1}^\infty
	\lambda_ja_j$, with $\{\lambda_j\}_{j=1}^\infty\in \ell^p$, each $a_j$ being a $(p,q,\rho)$-atom, and the sum converging in $L^2$, under the (quasi-)norm
	$$
	\|f\|_{H^{p,q}_{{\rm at},\rho}(G)}:=\inf\Big\{\Big(\sum_{j=1}^\infty|\lambda_j|^p\Big)^{1/p}:
	f=\sum_{j=1}^\infty \lambda_ja_j\Big\},
	$$
	where the infimum is taken over all possible $(p,q,\rho)$-atomic decompositions of $f$.
\end{defn}
Then in this case, $H^p_{L_V}\equiv H^{p,q}_{{\rm at},\rho}(G)$. 

\subsection{Fractional Leibniz rule associated with Grushin operators}

For $x = (x', x'') \in \mathbb{R}^{n+1}$ with $x' \in \mathbb{R}^n$ and $x'' \in \mathbb{R}$, we define the Grushin operator
\begin{equation}\label{eq-Grushin operator}
	L_G = -\Delta_{x'} - |x'|^2 \partial_{x''}^2, \qquad
	\Delta_{x'} = \sum_{j=1}^n \partial_{x'_j}^2.
\end{equation}

The operator $L_G$ is hypoelliptic and is homogeneous of degree $2$ with respect to the anisotropic dilations
\[
\delta_r(x) = (r x', r^2 x''), \qquad r > 0,
\]
in the sense that
\[
L_G(f \circ \delta_r) = r^2 (L_G f) \circ \delta_r.
\]

Set
\begin{equation}\label{Grushin-vector-fields}
	X_j = \partial_{x'_j}, \quad X_{n+j} = x'_j \partial_{x''}, \qquad j = 1, \dots, n,
\end{equation}
and let $X=\{X_1,\dots,X_{2n}\}$. Then
\[
L_G = -\sum_{j=1}^{2n} X_j^2,
\]
and a direct computation gives the nontrivial commutator
\[
[X_j, X_{n+j}] = \partial_{x''}, \qquad j=1,\dots,n,
\]
so the Hörmander condition is satisfied.

The operator \(L_G\) is symmetric and nonnegative on \(L^2(\mathbb{R}^{n+1})\); in fact it is essentially self-adjoint. We equip \(\mathbb{R}^{n+1}\) with the control (Carnot–Carathéodory) distance \(d\) associated to the vector fields \(X\) (see, e.g., \cite{RS}). This distance is homogeneous with respect to \(\delta_r\):
\[
d(\delta_r x, \delta_r y) = r\, d(x,y),
\]
and satisfies the two-sided comparability
\[
d(x,y) \simeq |x' - y'| + \min\!\Big\{ \frac{|x''-y''|}{|x'|+|y'|},\; |x''-y''|^{1/2}\Big\}.
\]

If \(|\cdot|\) denotes Lebesgue measure and \(B(x,r)\) the ball for \(d\), then
\[
|B(x,r)| \simeq r^{n+1}\,(r + |x'|),
\]
and consequently for \(R\ge r>0\),
\[
\Big(\frac{R}{r}\Big)^{n+1} \lesssim \frac{|B(x,R)|}{|B(x,r)|} \lesssim \Big(\frac{R}{r}\Big)^{n+2}.
\]
Thus \((\mathbb{R}^{n+1},d,|\cdot|)\) is doubling and satisfies the geometric conditions used below.

For multi-index notation, let for \(h\in\mathbb{N}\)
\[
\mathcal{I}^h=\{I=(i_1,\dots,i_h): i_j\in\{1,\dots,2n\}\}.
\]
For \(I=(i_1,\dots,i_h)\in\mathcal{I}^h\) define
\[
X^I := X_{i_1}X_{i_2}\cdots X_{i_h}.
\]
%with the convention that \(X^I f = X_{i_1}\big(X_{i_2}(\cdots(X_{i_h} f)\cdots)\big)\) (the right-most index is applied first). If \(h=0\) we write \(\mathcal{I}^0=\{\varnothing\}\) and \(X^\varnothing=\mathrm{Id}\).

\begin{prop}\label{prop-Grushin}
	Let \(H_t(x,y)\) denote the heat kernel of \(L_G\). Then the following properties hold.
	\begin{enumerate}[\rm (i)]
		\item There exist constants \(b_0\) such that for every \(t>0\) and \(x,y\in\mathbb{R}^{n+1}\)
		\[
		 0< H_t(x,y) \;\lesssim\; |B(x,\sqrt{t})|^{-1} \exp\Big(-b_0\frac{d(x,y)^2}{t}\Big).
		\]
		\item There exists $\delta \in (0,1)$ such that 
		\[
		|H_t(x,y)-H_t(\overline x,y)|\lesi \Big(\f{d(x,\overline x)}{\sqrt t}\Big)^{\delta} \f{1}{|B(x,\sqrt{t})|}\exp\Big(-c\frac{d(x,y)^2}{t}\Big)
		\]
		for all $t>0$ and $x,\overline x, y\in \mathbb R^{n+1}$ with $d(x,\overline x)<\sqrt t$.
		
		\item For every \(h\in\mathbb{N}\) there exists a constant \(a=a_h>0\) (depending on \(h\) and \(n\)) such that for all \(t>0\), \(x,y\in\mathbb{R}^{n+1}\) and every \(I\in\mathcal{I}^h\),
		\[
		\big|X^I_x H_t(x,y)\big| \;\lesssim\; t^{-h/2}\,|B(x,\sqrt{t})|^{-1}\,\exp\Big(-c\frac{d(x,y)^2}{t}\Big),
		\]
		where \(X^I_x\) means the differential operator \(X^I\) acts on the \(x\)-variable of \(H_t(x,y)\).
	\end{enumerate}
\end{prop}
\begin{proof}
	The item (i) can be found in \cite{RS2}. The H\"older continuity (ii) just follows from \cite[Lemma 2.4 \& Corollary 2.5]{DJ}. The estimate (iii) was proved in \cite[Proposition 8.6]{Bruno}.
\end{proof}
Since the product rule holds true, i.e., $X_j(fg)=X_jf g + fX_jg$ for  $j=1,\ldots, 2n$, from Proposition \ref{prop-Grushin} by using the same argument as in the proof of Theorem \ref{main thm-nilpotent} we have:
\begin{thm}\label{main thm-Grushin}
	Let $L_G$ be the Grushin operator as in \eqref{eq-Grushin operator}, and let $\boldsymbol{\mathrm{m}}$ satisfy \eqref{eq-condition on m} for some $\gamma \in \mathbb{R}$. Suppose $0<q\le \vc$, and   $0<p,p_1,p_2,p_3,p_4 \le \infty$ satisfies \eqref{eq- pi condition}, and $w, w_1, w_2, w_3, w_4 \in A_\infty$ satisfies \eqref{eq- wi condition}. 
	
	If $0<p_1,p_4<\infty$, $0<p_2,p_3 \le \infty$, and $s > \tau_{p,q}(w)$, then
	\begin{equation}
		\|B_{\boldsymbol{\mathrm{m}},L_G}(f,g)\|_{\dot{F}^{s,L_G}_{p,q}(w)} 
		\lesssim \|f\|_{\dot{F}^{s+\gamma,L_G}_{p_1,q}(w_1)} \|g\|_{H^{p_2}_{L_G}(w_2)} 
		+ \|f\|_{H^{p_3}_{L_G}(w_3)} \|g\|_{\dot{F}^{s+\gamma,L_G}_{p_4,q}(w_4)}.
	\end{equation}
	
	Similarly, for the Besov scale, if $0 < p,p_1,p_2,p_3,p_4 \le \infty$ and $s > \tau_p(w)$, it holds that
	\begin{equation}
		\|B_{\boldsymbol{\mathrm{m}},L_G}(f,g)\|_{\dot{B}^{s,L_G}_{p,q}(w)} 
		\lesssim \|f\|_{\dot{B}^{s+\gamma,L_G}_{p_1,q}(w_1)} \|g\|_{H^{p_2}_{L_G}(w_2)} 
		+ \|f\|_{H^{p_3}_{L_G}(w_3)} \|g\|_{\dot{B}^{s+\gamma,L_G}_{p_4,q}(w_4)}.
	\end{equation}
	
	In both inequalities, the Hardy spaces $H^{p_2}_{L_G}(w_2)$ and $H^{p_3}_{L_G}(w_3)$ are replaced by $L^\infty$ if $p_2=\infty, w_2=1$ or $p_3=\infty, w_3=1$, respectively.
\end{thm}

\begin{cor}\label{cor-Grushin}
	Under the same conditions as in Theorem \ref{main thm-Grushin}, we have:
	\begin{enumerate}[\rm (a)]
		\item  If $0<p_1,p_4<\infty$, $0<p_2,p_3 \le \infty$, and $s > \tau_{p,q}(w)$, then
	\begin{equation}
		\|L_G^{s/2}(fg)\|_{\dot{F}^{0,L_G}_{p,q}(w)} \simeq \|fg\|_{\dot{F}^{s,L_G}_{p,q}(w)} 
		\lesssim \|f\|_{\dot{F}^{s,L_G}_{p_1,q}(w_1)} \|g\|_{H^{p_2}_{L_G}(w_2)} 
		+ \|f\|_{H^{p_3}_{L_G}(w_3)} \|g\|_{\dot{F}^{s,L_G}_{p_4,q}(w_4)}.
	\end{equation}
	
	\item Similarly, for the Besov scale, if $0 < p,p_1,p_2,p_3,p_4 \le \infty$ and $s > \tau_p(w)$, it holds that
	\begin{equation}
		\|L_G^{s/2}(fg)\|_{\dot{B}^{0,L_G}_{p,q}(w)} \simeq \|fg\|_{\dot{B}^{s,L_G}_{p,q}(w)} 
		\lesssim \|f\|_{\dot{B}^{s,L_G}_{p_1,q}(w_1)} \|g\|_{H^{p_2}_{L_G}(w_2)} 
		+ \|f\|_{H^{p_3}_{L_G}(w_3)} \|g\|_{\dot{B}^{s,L_G}_{p_4,q}(w_4)}.
	\end{equation}
\end{enumerate}
	In both inequalities, the Hardy spaces $H^{p_2}_{L_G}(w_2)$ and $H^{p_3}_{L_G}(w_3)$ are replaced by $L^\infty$ if $p_2=\infty, w_2=1$ or $p_3=\infty, w_3=1$, respectively.
\end{cor}

Note that the Grushin operator $L_G$ also satisfies the conservation property \eqref{eq-conservation}. Hence, by Remark \ref{rem2} the (unweighted) Besov and Triebel-Lizorkin $\dot{B}^{s,L_G}_{p_4,q}$ and $\dot{F}^{s,L_G}_{p_4,q}$ coincide with Besov and Triebel-Lizorkin defined in \cite{HMY, HS} for certain indices detailed in Remark \ref{rem2}. In particular, $H^p_{L_G}\equiv H^p_{\rm CW}$ for $\f{n+2}{n+3}<p\le 1$, where $H^p_{\rm CW}$ is the Hardy spaces of Coifmann and Weiss in \cite{CW}. See \cite{DJ,BDK}.
\subsection{Fractional Leibniz Rules associated with Hermite operators}

For each nonnegative integer $k \in \mathbb{N}_0$, the Hermite polynomial of order $k$ is defined by  
\[
H_k(t) = (-1)^k e^{t^2} \frac{d^k}{dt^k}\big(e^{-t^2}\big).
\]

In the one-dimensional case, the Hermite function of degree $k$ is given by  
\[
h_k(t) = \big(2^k k! \sqrt{\pi}\big)^{-1/2} H_k(t)\, e^{-t^2/2}, \qquad t \in \mathbb{R}.
\]

For higher dimensions, if $\xi = (\xi_1,\dots,\xi_n) \in \mathbb{N}_0^n$ is a multi-index, the associated Hermite function is constructed as  
\[
h_\xi(x) = \prod_{j=1}^n h_{\xi_j}(x_j), \qquad x = (x_1,\dots,x_n) \in \mathbb{R}^n.
\]

These functions serve as eigenfunctions of the Hermite operator $\mathcal{L}$, satisfying  
\[
\mathcal{L}(h_\xi) = (2|\xi| + n) h_\xi,
\]  
and they constitute a complete orthonormal system in $L^2(\mathbb{R}^n)$ (see, for example, \cite{Th}).  

For each $k \geq 0$, denote  
\[
W_k := \mathrm{span}\{h_\xi : |\xi| = k\}, \qquad V_N := \bigoplus_{k=0}^N W_k.
\]  
The orthogonal projection of a function $f \in L^2(\mathbb{R}^n)$ onto $W_k$ is given by  
\[
\mathbb{P}_k f = \sum_{|\xi|=k} \langle f, h_\xi \rangle \, h_\xi,
\]  
with the corresponding kernel representation  
\[
\mathbb{P}_k(x,y) = \sum_{|\xi|=k} h_\xi(x) h_\xi(y).
\]

\medskip

Given a ‘symbol’ $\textbf{m} : \mathbb{R}^n \times \mathbb{R}_+ \times \mathbb{R}_+ \to \mathbb{C}$ 
the \emph{bilinear Hermite pseudo-multiplier} $T_{\boldsymbol{\mathrm{m}}}$ is defined through the formula
\begin{equation}\label{eq1.5}
	T_{\boldsymbol{\mathrm{m}},\mathcal L}(f,g)(x) = \sum_{j,k \in \mathbb{N}_0} \boldsymbol{\mathrm{m}}(x,\lambda_j,\lambda_k) 
	\, \mathbb{P}_j f(x) \, \mathbb{P}_k g(x),
\end{equation}
where, for $\ell \in \mathbb{N}_0 = \mathbb{N}\cup\{0\}$, $\lambda_\ell = 2\ell+n$ are the eigenvalues of 
$\mathcal{L}$ and $\mathbb{P}_\ell$ are orthogonal projectors onto spaces spanned by Hermite functions. 

In this section, we assume that  there exists  $\gamma \in \mathbb{R}$   such that for all $\alpha, \beta \in \mathbb N$ and $\sigma\in \mathbb N^n$, 
	\begin{equation}\label{eq1.6}
		\big| \partial_x^\sigma \partial_\xi^\alpha\partial^\beta_\eta \boldsymbol{\mathrm{m}}(x,\xi,\eta) \big|
		\lesssim_{\kappa,\alpha,\beta} 
		\big(\xi + \eta\big)^{\gamma -  (\alpha+\beta) + |\sigma|},
		\qquad \forall x \in \mathbb{R}^n,\; \xi, \eta \in \mathbb{R}_+.
	\end{equation}
	
	The following result address the boundedness of the bilinear Hermite pseudo-multiplier. 
	
	\begin{thm}\label{main thm - Hermite}
		Let $\mathcal L$ be the Hermite operator, and $\boldsymbol{\mathrm{m}}$ satisfy \eqref{eq-condition on m} for some $\gamma\in \mathbb R$.  Suppose $0<q\le \vc$, and   $0<p,p_1,p_2,p_3,p_4 \le \infty$ satisfies \eqref{eq- pi condition}, and $w, w_1, w_2, w_3, w_4 \in A_\infty$ satisfies \eqref{eq- wi condition}. 
		
If $0<p_1,p_4<\infty$, $0<p_2,p_3 \le \infty$, and $s > \tau_{p,q}(w)$, then
			\begin{equation}
				\|T_{\boldsymbol{\mathrm{m}},\mathcal L}(f,g)\|_{\dot{F}^{s,\mathcal L}_{p,q}(w)}
				\lesssim \|f\|_{\dot{F}^{s+\gamma,\mathcal L}_{p_1,q}(w_1)} \|g\|_{H^{p_2}_{\mathcal L}(w_2)}
				+ \|f\|_{H^{p_3}_{\mathcal L}(w_3)} \|g\|_{\dot{F}^{s+\gamma,\mathcal L}_{p_4,q}(w_4)}.
			\end{equation}
			
 If $0<p,p_1,p_2,p_3,p_4 \le \infty$ and $s > \tau_p(w)$, then
			\begin{equation}
				\|T_{\boldsymbol{\mathrm{m}},\mathcal L}(f,g)\|_{\dot{B}^{s,\mathcal L}_{p,q}(w)}
				\lesssim \|f\|_{\dot{B}^{s+\gamma,\mathcal L}_{p_1,q}(w_1)} \|g\|_{H^{p_2}_{\mathcal L}(w_2)}
				+ \|f\|_{H^{p_3}_{\mathcal L}(w_3)} \|g\|_{\dot{B}^{s+\gamma,\mathcal L}_{p_4,q}(w_4)}.
			\end{equation}
		
		Here, the Hardy spaces $H_{\mathcal L}^{p_2}(w_2)$ and $H_{\mathcal L}^{p_3}(w_3)$ are replaced by $L^\infty$ if $p_2=\infty, w_2=1$ or $p_3=\infty, w_3=1$, respectively.
	\end{thm}

The proof of Theorem \ref{main thm - Hermite} will be postponed at the end of this section. As a consequence of Theorem \ref{main thm - Hermite}, we have the following corollary.
\begin{cor}\label{cor:Hermite}
	Under the same conditions as in Theorem \ref{main thm - Hermite}, we have:
	\begin{enumerate}[\rm (a)]
		\item If $0<p_1,p_4<\infty$, $0<p_2,p_3 \le \infty$, and $s > \tau_{p,q}(w)$, then
		\begin{equation}
			\|fg\|_{\dot{F}^{s,\mathcal L}_{p,q}(w)}
			\lesssim \|f\|_{\dot{F}^{s,\mathcal L}_{p_1,q}(w_1)} \|g\|_{H^{p_2}_{\mathcal L}(w_2)}
			+ \|f\|_{H^{p_3}_{\mathcal L}(w_3)} \|g\|_{\dot{F}^{s,\mathcal L}_{p_4,q}(w_4)}.
		\end{equation}
 
 \item If $0<p,p_1,p_2,p_3,p_4 \le \infty$ and $s > \tau_p(w)$, then
		\begin{equation}
			\|fg\|_{\dot{B}^{s,\mathcal L}_{p,q}(w)}
			\lesssim \|f\|_{\dot{B}^{s,\mathcal L}_{p_1,q}(w_1)} \|g\|_{H^{p_2}_{\mathcal L}(w_2)}
			+ \|f\|_{H^{p_3}_{\mathcal L}(w_3)} \|g\|_{\dot{B}^{s,\mathcal L}_{p_4,q}(w_4)}.
		\end{equation}
			\end{enumerate}
	Here again, the Hardy spaces $H_{\mathcal L}^{p_2}(w_2)$ and $H_{\mathcal L}^{p_3}(w_3)$ are replaced by $L^\infty$ if $p_2=\infty, w_2=1$ or $p_3=\infty, w_3=1$, respectively.
\end{cor}
As a special case of the Schr\"odinger operator in Subsection 3.1,  the Hardy spaces $H^p_{\mathcal L}$ can be characterized by atomic decomposition as in Definition \ref{def: rho atoms} with $\rho(x)\simeq \f{1}{1+|x|}$.

In order to prove Theorem \ref{main thm - Hermite}, we need the following lemma.
\begin{lem}
	There exists a family of functions $\{\widetilde \Gamma_j(\cdot,\cdot)\}_{j\ge 0}$ satisfying (i)-(iv) as in Lemma \ref{lem-kernel-representation} in which $\mathcal L$ will take place of $L$. Moreover, for each $\alpha, \beta \in \mathbb N^n$ and $k\in \mathbb N$,  we have
	\begin{equation}\label{eq- product estimate for hermite}
	\Big|x^\beta \partial^\sigma_x\big[\widetilde \Gamma_k(x,y)\widetilde \Gamma_k(x,z)\big]\Big|\lesi 2^{k(|\alpha|+|\beta|)}D_{2^{-k},N}(x,y)D_{2^{-k},N}(x,z)
	\end{equation}
	for all $x,y,z\in \mathbb R^n$.
\end{lem}
\begin{proof}
Since $|x|^2$ is a nonnegative polynomial in $\mathbb R^n$, it follows from the proof of Theorem \ref{main thm-nilpotent Schrodinger} that the Hermite operator $\mathcal L$ satisfies (A1)–(A3). Hence, by Lemma \ref{lem-kernel-representation}, there exists a family of functions $\{\widetilde \Gamma_j(\cdot,\cdot)\}_{j\ge 0}$ satisfying (i)–(iv) as in Lemma \ref{lem-kernel-representation}. 

Together with Proposition \ref{derivatives of heat kernel Schrodinger}, this implies that for each $\alpha, \beta \in \mathbb N^n$,
\[
\Big| x^\beta \partial^\alpha_x p_t(x,y) \Big| \lesssim \frac{1}{t^{(n+|\alpha|+|\beta|)/2}} \exp\Big(-\frac{|x-y|^2}{ct}\Big)
\]
for all $x,y \in \mathbb R^n$ and $t>0$, since $\rho(x) \simeq \frac{1}{1+|x|}$ implies $|x|\rho(x) \lesssim 1$.

Similarly, by the standard argument as in the proof of Proposition~3.2 of \cite{BD}, for each $\alpha, \beta \in \mathbb N^n$ and an even function $\varphi \in \mathscr{S}(\mathbb R)$, we have
\[
\Big| x^\beta \partial^\alpha_x \varphi(t\mathcal L)(x,y) \Big| \lesssim \frac{1}{t^{(n+|\alpha|+|\beta|)/2}} D_{t,N}(x,y)
\]
for all $x,y \in \mathbb R^n$ and $t, N>0$.

At this stage, repeating the argument used in the proof of Lemma \ref{lem-kernel-representation} (iii) yields \eqref{eq- product estimate for hermite}.

This completes the proof.
  
\end{proof}

We are ready to give the proof of Theorem \ref{main thm - Hermite}.
\begin{proof}[Proof of Theorem \ref{main thm - Hermite}:]

The proof is in the same spirit as that of Theorem \ref{main thm} and we will use the same notation as in the proof of Theorem \ref{main thm}. Similarly to \eqref{eq- - m decomposition}, we can write
\begin{equation*} 
	\begin{aligned}
		\boldsymbol{\mathrm{m}}(x;\xi,\eta) 
		&= \sum_{k\in \mathbb N_0}  \boldsymbol{\mathrm{m}}(x;\xi,\eta)\psi_k(\xi)\phi_k(\eta) +\sum_{j\in \mathbb N_0}  \boldsymbol{\mathrm{m}}(x;\xi,\eta)\phi_{j-1}(\xi)\psi_j(\eta).
	\end{aligned}
\end{equation*}
Consequently,
\[
T_{\boldsymbol{\mathrm{m}},\mathcal L}(f,g)= \sum_{k\in \mathbb N_0}  T_{\boldsymbol{\mathrm{m}}_{1,k},\mathcal L}(f,g) +\sum_{j\in \mathbb N_0}  T_{\boldsymbol{\mathrm{m}}_{2,j},\mathcal L}(f,g),
\]
where $\boldsymbol{\mathrm{m}}_{1,k}(x;\xi,\eta)=\boldsymbol{\mathrm{m}}(x;\xi,\eta)\psi_k(\xi)\phi_k(\eta)$ and $m_{2,j}(x;\xi,\eta)=\boldsymbol{\mathrm{m}}(x;\xi,\eta)\phi_{j-1}(\xi)\psi_j(\eta)$.

We need only to estimate the first term 
$$
\displaystyle \sum_{k}  T_{\boldsymbol{\mathrm{m}}_{1,k},\mathcal L} (f,g),
$$
since these two terms above are similar.

Similarly to the composition of $\boldsymbol{\mathrm{m}}_{1,k}(\xi,\eta)$ in the proof of Theorem \ref{main thm}, we also have 
\[
\begin{aligned}
	\boldsymbol{\mathrm{m}}_{1,k}(x; \xi,\eta)&=\sum_{(n_1,n_2)\in\mathbb{Z}^2} c_k(x; n_1,n_2)\,
	e^{2\pi i\big(\frac{\xi n_1+\eta n_2}{2^k}\big)}\Psi_k(\xi)\Phi_k(\eta)\\
	&=:	\sum_{(n_1,n_2)\in\mathbb{Z}^2} c_k(x; n_1,n_2)\,
	\Psi_{k,n_1}(\xi)\Phi_{k,n_2}(\eta),
\end{aligned}
\]
where $\Psi_{k,n_1}(\xi)=e^{\frac{2\pi i\xi n_1}{2^k}}\Psi_k(\xi)$ and $\Phi_{k,n_2}(\eta)=e^{\frac{2\pi i\eta n_2}{2^k}}\Phi_k(\eta)$ and
\begin{equation}
	c_k(x;n_1,n_2)
	=2^{-2k} \int_{2^{k-1}<|\xi|\le 2^{k+1}}\int_{|\eta|\le 2^{k+1}}\,
	\boldsymbol{\mathrm{m}}_{1,k}(x;\xi,\eta)e^{-2\pi i\frac{\xi n_1+\eta n_2}{2^k}}\,d\xi\,d\eta.
\end{equation}
Also, by integration by parts and \eqref{eq1.6},
\begin{equation}\label{eq-cnn  x}
	|\partial_{x}^\sigma c_k(x; n_1,n_2)| \lesssim_{N,\sigma} 2^{k(\gamma+|\sigma|)}(1+|n_1|+|n_2|)^{-N}
\end{equation}
for every $N$ and $\sigma\in \mathbb N^n$.

Consequently,
\[
T_{\boldsymbol{\mathrm{m}}_{1,k},\mathcal L}(f,g)= \sum_{k\in \mathbb N_0}\sum_{(n_1,n_2)\in\mathbb{Z}^2} c_k(x;n_1,n_2)\,
\Psi_{k,n_1}(\sqrt {\mathcal L})f \Phi_{k,n_2}(\sqrt {\mathcal L})g.
\]

Applying Lemma \ref{lem-kernel-representation},
\[
\begin{aligned}
	B_{\boldsymbol{\mathrm{m}}_{1,k},L}(f,g)(x)	=& \sum_{k\in \mathbb N_0} \sum_{(n_1,n_2)\in\mathbb{Z}^2} c_k(x;n_1,n_2) \\
	& \ \ \times\sum_{Q\in \mathcal D_k} \sum_{R\in \mathcal D_k}\omega_Q\Psi_{k,n_1}(\sqrt{{\mathcal L}})f(x_Q) \widetilde{\Gamma}_k(x,x_Q) \omega_R\Phi_{k,n_2}(\sqrt{{\mathcal L}})g(x_R)  \widetilde{\Gamma}_k(x,x_R).	
\end{aligned}
\]
Setting
\[
\psi_\ell(\sqrt {\mathcal L}) {\mathcal L}^{s/2} = 2^{\ell s} \theta_\ell(\sqrt {\mathcal L}),
\]
then we can further write
\begin{equation}\label{eq- 1st eq BmL}
	\begin{aligned}
		\psi_\ell&(\sqrt {\mathcal L}) L^{s/2}\big[B_{\boldsymbol{\mathrm{m}}_{1,k},{\mathcal L}}(f,g)\big](x)\\
		=&  \sum_{k\in \mathbb N_0} \sum_{(n_1,n_2)\in\mathbb{Z}^2} \sum_{Q\in \mathcal D_k} \sum_{R\in \mathcal D_k}2^{\ell s} \omega_Q \Psi_{k,n_1}(\sqrt{{\mathcal L}})f(x_Q)  \omega_{R}\Phi_{k,n_2}(\sqrt{L})g(x_R) \\
		& \ \ \times  \int_{\mathbb R^n} \theta_\ell(\sqrt {\mathcal L})(x,y)c_k(y;n_1,n_2)\widetilde{\Gamma}_k(y,x_Q) \widetilde{\Gamma}_k(y,x_R) d\mu(y)\\
		=&  \sum_{k\in \mathbb N_0} \sum_{(n_1,n_2)\in\mathbb{Z}^2} \sum_{Q\in \mathcal D_k} \sum_{R \in \mathcal D_k}2^{s( \ell-k)} \omega_Q \overline{\Psi}_{k,n_1}(\sqrt{{\mathcal L}})({\mathcal L}^{s/2}f)(x_Q)  \omega_R \Phi_{k,n_2}(\sqrt{{\mathcal L}})g(x_R)\\
		& \ \ \times   \int_{\mathbb R^n} \theta_\ell(\sqrt {\mathcal L})(x,y)c_k(y;n_1,n_2)\widetilde{\Gamma}_k(y,x_Q) \widetilde{\Gamma}_k(y,x_R) d\mu(y)\\
		=&: \sum_{k \in \mathbb N_0 \atop k\le \ell} E_{\ell,k}+\sum_{k \in \mathbb Z \atop k>\ell} E_{\ell,k},
	\end{aligned}
\end{equation}
where $\overline{\Psi}_{k,n_1}(\sqrt {\mathcal L}) = [2^j\sqrt {\mathcal L}]^{-s}\Psi_{k,n_1}(\sqrt {\mathcal L})$.

The term $\displaystyle \sum_{k \in \mathbb Z \atop k>\ell} E_{\ell,k}$ can be done similarly to the corresponding term in the proof of Theorem \ref{main thm}.

We now take care of the term $\displaystyle \sum_{k \in \mathbb Z \atop k\le \ell} E_{\ell,k}$. We have,
\[
\begin{aligned}
	\Big|\int_{\mathbb R^n}& \theta_\ell(\sqrt {\mathcal L})(x,y)c_k(y;n_1,n_2) \widetilde{\Gamma}_k(y,x_Q)\widetilde{\Gamma}_k(y,x_R) d\mu(y)\Big|\\
	&=\Big|\int_{\mathbb R^n} \theta_\ell(\sqrt {\mathcal L})(y,x) c_k(y;n_1,n_2)\widetilde{\Gamma}_k(y,x_Q)\widetilde{\Gamma}_k(y,x_R) d\mu(y)\Big|\\
	&=\Big|\int_{\mathbb R^n} 2^{-2\ell m} (2^{-\ell}\sqrt L)^{2m}\theta_\ell(\sqrt {\mathcal L})(y,x) L^m\big[c_k(y;n_1,n_2)\widetilde{\Gamma}_k(y,x_Q)\widetilde{\Gamma}_k(y,x_R)\big] d\mu(y)\Big|.
\end{aligned}
\]

Note that
\[
\mathcal L^m(fg) = 
\sum_{\substack{\alpha,\beta,\sigma \in \mathbb{N}^n \\ |\alpha|+|\beta|+|\sigma| \leq 2m}}
c_{m,\alpha,\beta,\sigma}\,
\bigl(\partial^\alpha  f\bigr)\, x^\beta \, \bigl(\partial^\sigma g\bigr).
\]
Consequently, by \eqref{eq-cnn  x}, \eqref{eq- product estimate for hermite} and $k\in \mathbb N_0$, for any $N, M>0$, we have
\[
\begin{aligned}
	|\mathcal L^m &\big[c_k(x,\textbf{n})\widetilde \Gamma_k(x,y)\widetilde \Gamma_k(x,z)\big]|\\
	&\lesi \sum_{\substack{\alpha,\beta,\sigma \in \mathbb{N}^n \\ |\alpha|+|\beta|+|\sigma| \leq 2m}} |\partial^\alpha_x c_k(x;n_1,n_2)| \Big|x^\beta \partial^\sigma_x\big[\widetilde \Gamma_k(x,y)\widetilde \Gamma_k(x,z)\big]\Big|\\
	&\lesi 2^{k(\gamma+|\alpha|)}(1+n_1+n_2)^{-N}\sum_{\substack{\beta,\sigma \in \mathbb{N}^n \\ |\beta|+|\sigma| \leq 2m-|\alpha|}}   \Big|x^\beta \partial^\sigma_x\big[\widetilde \Gamma_k(x,y)\widetilde \Gamma_k(x,z)\big]\Big|\\
	&\lesi 2^{k(\gamma+|\alpha|)}(1+n_1+n_2)^{-N}\sum_{\substack{\beta,\sigma_1,\sigma_2 \in \mathbb{N}^n \\ |\beta|+|\sigma_1|+|\sigma_2| \leq 2m-|\alpha|}}   \Big|x^\beta \partial^{\sigma_1}_x\ \widetilde \Gamma_k(x,y)\partial^{\sigma_2}_x\widetilde \Gamma_k(x,z) \Big|\\
	&\lesi 2^{k(\gamma+|\alpha|)}(1+n_1+n_2)^{-N}\\
	& \ \ \times \sum_{\substack{\beta,\sigma_1,\sigma_2 \in \mathbb{N}^n \\ |\beta|+|\sigma_1|+|\sigma_2| \leq 2m-|\alpha|}}2^{k(|\beta|+|\sigma_1|+|\sigma_2|)} D_{2^{-k},M}(x,y) D_{2^{-k},M}(x,y)D_{2^{-k},M}(x,y)\\
	&\lesi 2^{k(\gamma+|\alpha|)}(1+n_1+n_2)^{-N} 2^{k(m-|\alpha|)} D_{2^{-k},M}(x,y) D_{2^{-k},M}(x,y)D_{2^{-k},M}(x,y)\\
	&\lesi 2^{k(\gamma+m)}(1+n_1+n_2)^{-N} D_{2^{-k},M}(x,y) D_{2^{-k},M}(x,y)D_{2^{-k},M}(x,y).
\end{aligned}
\]
This, together with  Lemmas \ref{lem-kernel-representation} and \ref{lem-composite kernel},  yields that  for $\mathbb N\ni m>(s+1)/2$ and a fixed $M>\max\{\f{n\tau_{w_1}}{r_1}, \f{n\tau_{w_2}}{r_2}\}$,
\[
\begin{aligned}
	\Big|\int_{\mathbb R^n} \theta_\ell(\sqrt {\mathcal L})(x,y)&c_k(y;n_1,n_2) \widetilde{\Gamma}_k(y,x_Q)\widetilde{\Gamma}_k(y,x_R) d\mu(y)\Big|\\
	&\lesi \int_{\mathbb R^n} 2^{-2m(\ell-k)}D_{2^{-\ell},2M+2n}(x,y)D_{2^{-k},2M+2n}(y,x_Q)D_{2^{-k},2M+2n}(y,x_R)  d\mu(y)\\
	&\lesi  2^{-(s+1)(\ell-k)}D_{2^{-k},M}(x,x_Q)D_{2^{-k},M}(x,x_R).
\end{aligned}
\]

Hence, using Lemma \ref{lem1- thm2 atom Besov},
\[
\begin{aligned}
	E_{\ell,k}(x)&\lesi  \sum_{(n_1,n_2)\in\mathbb{Z}^2} (1+n_1+n_2)^{-N}\sum_{Q\in \mathcal D_k} \sum_{R\in \mathcal D_k}2^{-(\ell-k)} \omega_Q  |\overline{\Psi}_{k,n_1}(\sqrt{L})(L^{s/2}f)(x_Q)|\\
	& \ \ \ \ \ \times  \omega_R |\Phi_{k,n_2}(\sqrt{L})g(x_R)|D_{2^{-k},M}(x,\xi)D_{2^{-k},M}(x,x_R)\\
	& \lesi \sum_{(n_1,n_2)\in\mathbb{Z}^2} (1+n_1+n_2)^{-N}2^{(k-\ell)} \\
	& \ \ \ \ \ \times \mathcal M_{r_1,w_1}\Big(\sum_{Q\in \mathcal D_k}|\overline{\Psi}_{k,n_1}(\sqrt{L})(L^{s/2}f)(x_Q)|1_{Q} \Big)(x) \mathcal M_{r_2,w_2}\Big(\sum_{R\in \mathcal D_k}|{\Phi}_{k,n_2}g(x_R)|1_{R} \Big)(x).  
\end{aligned}
\]
At this stage, the remaining steps follow by an argument similar to that used in the proof of Theorem \ref{main thm}, and we therefore omit the details.\\

This completes our proof.	
\end{proof}

\section{Appendices}
\subsection{Appendix A}\label{Sec-Appendix} 
This section is devoted to justifying the validity of the identity \eqref{eq-bilinear form}. 
In fact, it follows from Lemmas \ref{lem2-Appendix} and \ref{lem-kernel m(L,L)} stated below. 
Therefore, the main task here is to prove these two lemmas. 

\textit{Throughout this section we assume that $L$ satisfies (A1) and (A2) only. 
	The condition (A3) is not needed.}

\begin{lem}\label{lem1-Appendix}
	For any $s\in \mathbb R$ we have 
	\[
	(L_1+L_2)^s(f\otimes g)(x,x) \in L^p(X), 
	\qquad 1\le p\le \vc,
	\]
	for all $f,g\in \mathcal S_\vc$.
\end{lem}

\begin{proof}
	Fix $p\in [1,\vc]$. We consider two cases.  
	
	\medskip
	\textbf{Case 1: $s\ge 0$.}  
	
	If $s\in \mathbb N$, then it is straightforward that 
	\((L_1+L_2)^s(f\otimes g)(x,x) \in L^p(X)\), 
	since $L^mf, L^mg \in L^q$ for all $m\in \mathbb Z$ and $1\le q\le \vc$.  
	Hence it suffices to prove the case $s\in (0,1)$.  
	
	In this situation, we write
	\[
	\begin{aligned}
		\Big|(L_1+L_2)^s(f\otimes g)(x,x)\Big| 
		&= \f{1}{\Gamma(1-s)}\Big|\int_0^\vc t^{-s}(L_1+L_2)e^{-t(L_1+L_2)}(f\otimes g)(x,x)\,dt\Big| \\
		&=\Big|\int_0^1 \ldots \,dt\Big| + \Big|\int_1^\vc \ldots \,dt\Big| 
		=:E_1(x)+E_2(x).
	\end{aligned}
	\]
	
	Using H\"older  inequality and the $L^{2p}$-boundedness of $e^{-tL}$ (uniformly in $t$), we obtain
	\[
	\begin{aligned}
		\|E_1\|_p 
		&\lesi \int_0^1 t^{-s}\|e^{-tL}(Lf)(x)e^{-tL}g(x)\|_{L^p}\,dt 
		+ \int_0^1 t^{-s}\|e^{-tL}f(x)e^{-tL}(Lg)(x)\|_{L^p}\,dt \\
		&\lesi \|Lf\|_{{L^{2p}}}\|g\|_{{L^{2p}}} + \|f\|_{{L^{2p}}}\|Lg\|_{L^{2p}}.
	\end{aligned}
	\]
	
	Similarly, using the H\"older inequality and the uniform $L^{2p}$-boundedness of $e^{-tL}$ and $tLe^{-tL}$, we have
	\[
	\begin{aligned}
		\|E_2\|_p 
		&\lesi \int_1^\vc t^{-s-1}\|tLe^{-tL}f(x)e^{-tL}g(x)\|_{{L^p}}\,dt 
		+ \int_1^\vc t^{-s-1}\|e^{-tL}f(x)\,tLe^{-tL}g(x)\|_{L^p}\,dt \\
		&\lesi \|f\|_{{L^{2p}}}\|g\|_{{L^{2p}}}.
	\end{aligned}
	\]
	This completes the proof for $s\ge 0$.  
	
	\medskip
	\textbf{Case 2: $s<0$.}  
	
	We use the formula
	\[
	\begin{aligned}
		\Big|(L_1+L_2)^{-s}(f\otimes g)(x,x)\Big|
		&= \f{1}{\Gamma(-s)}\Big|\int_0^\vc t^{-s-1}(L_1+L_2)e^{-t(L_1+L_2)}(f\otimes g)(x,x)\,dt\Big| \\
		&=\Big|\int_0^1 \ldots \,dt\Big| + \Big|\int_1^\vc \ldots \,dt\Big| 
		=:F_1(x)+F_2(x).
	\end{aligned}
	\]
	
	As in the case of $E_1$, we get
	\[
	\|F_1\|_{L^p} \lesi \|Lf\|_{L^{2p}}\|g\|_{L^{2p}} + \|f\|_{L^{2p}}\|Lg\|_{L^{2p}}.
	\]
	
	For $F_2$, fix $N\in \mathbb N$ with $N>-s$. Then
	\[
	\begin{aligned}
		\|F_2\|_{L^{p}} 
		&\lesi \int_1^\vc t^{-s-1}\|L^{N+1}e^{-tL}(L^{-N}f)(x)e^{-tL}g(x)\|_{L^{p}}\,dt \\
		&\quad + \int_1^\vc t^{-s-1}\|e^{-tL}f(x)\,L^{N+1}e^{-tL}(L^{-N}g)(x)\|_{L^{p}}\,dt \\
		&\lesi \|f\|_{L^{2p}}\|g\|_{L^{2p}}.
	\end{aligned}
	\]
	
	By the H\"older inequality and the uniform $L^{2p}$-boundedness of $e^{-tL}$ and $(tL)^{N+1}e^{-tL}$, we further obtain
	\[
	\|F_2\|_{L^{p}} \lesi \|L^{-N}f\|_{L^{2p}}\|g\|_{L^{2p}} + \|f\|_{L^{2p}}\|L^{-N}g\|_{L^{2p}}.
	\]
	
	This completes the proof.
\end{proof}

\begin{lem}\label{lem2-Appendix}
	Let $\boldsymbol{\mathrm{m}}$ satisfy \eqref{eq-condition on m}. Then 
	$\boldsymbol{\mathrm{m}}(L_1,L_2)(f\otimes g) \in L^2(X\times X)$ for all $f,g\in \mathcal S_\vc$.
\end{lem}

\begin{proof}
	We have
	\[
	\begin{aligned}
		\|\boldsymbol{\mathrm{m}}(L_1,L_2)(f\otimes g)\|_2^2
		&= \int_{(0,\vc)^2} |F(\lambda_1,\lambda_2)|^2\, d\langle E(\lambda_1,\lambda_2)(f\otimes g),f\otimes g \rangle \\
		&\lesi \int_{(0,\vc)^2} (\lambda_1+\lambda_2)^{2\gamma}\, d\langle E(\lambda_1,\lambda_2)(f\otimes g),f\otimes g \rangle \\
		&\simeq \|(L_1+L_2)^\gamma(f\otimes g)\|_2^2.
	\end{aligned}	
	\]
	This, together with Lemma \ref{lem1-Appendix}, implies that 
	$\boldsymbol{\mathrm{m}}(L_1,L_2)(f\otimes g) \in L^2(X\times X)$.
	
	This completes the proof.
\end{proof}  

In what follows, for $\boldsymbol{\mathrm{x}}=(x_1,x_2), \boldsymbol{\mathrm{y}}=(y_1,y_2) \in X\times X$, we define 
\[
\boldsymbol{\mathrm{d}}(\boldsymbol{\mathrm{x}}, \boldsymbol{\mathrm{y}}) = \max\{d(x_1,y_1), d(x_2,y_2)\}, \quad 
B(\boldsymbol{\mathrm{x}},r) = B(x_1,r)\times B(x_2,r), \quad r>0,
\]
and denote $\mu^2(B(\boldsymbol{\mathrm{x}},r)) = \mu(B(x_1,r))\times \mu(B(x_2,r))$. 

\begin{lem}
	Let $R, s>0$. Then there exists a constant $C = C(s)$ such that
	\begin{equation}\label{eq1-kernel mLL}
		\big|F(L_1,L_2)(\boldsymbol{\mathrm{x}},\boldsymbol{\mathrm{y}})\big| 
		\leq C \, \|F(R^2\cdot, R^2\cdot)\|_{\mathcal W^{s+1}_{\vc}(\mathbb R^2)} 
		\frac{\big(1+R \boldsymbol{\mathrm{d}}(\boldsymbol{\mathrm{x}},\boldsymbol{\mathrm{y}})\big)^{-s}}{\mu^2\big(B(\boldsymbol{\mathrm{y}},R^{-1})\big)},
	\end{equation}
	and
	\begin{equation}\label{eq2-kernel mLL}
		\begin{aligned}
			\big|F(L_1,L_2)(\boldsymbol{\mathrm{x}},\boldsymbol{\mathrm{y}}) - F(L_1,L_2)(\overline{\boldsymbol{\mathrm{x}}},\boldsymbol{\mathrm{y}})\big| 
			&\leq C \, \|F(R^2\cdot, R^2\cdot)\|_{\mathcal W^{s+1}_{\vc}(\mathbb R^2)}
			(R \boldsymbol{\mathrm{d}}(\boldsymbol{\mathrm{x}},\overline{\boldsymbol{\mathrm{x}}}))^{\delta_0} \\
			&\quad \times \frac{\big(1+R \boldsymbol{\mathrm{d}}(\boldsymbol{\mathrm{x}},\boldsymbol{\mathrm{y}})\big)^{-s}}{\mu^2\big(B(\boldsymbol{\mathrm{y}},R^{-1})\big)}
		\end{aligned}
	\end{equation}
	for all Borel functions $F$ with $\operatorname{supp} F \subseteq [0,R^2]^2$ and $\boldsymbol{\mathrm{x}}, \overline{\boldsymbol{\mathrm{x}}},\boldsymbol{\mathrm{y}}\in X\times X$ satisfying $R \boldsymbol{\mathrm{d}}(\boldsymbol{\mathrm{x}},\overline{\boldsymbol{\mathrm{x}}}) <1$, where 
	$F(L_1,L_2)(\boldsymbol{\mathrm{x}},\boldsymbol{\mathrm{y}})$ is the integral kernel of $F(L_1,L_2)$ and 
	$\|f\|_{\mathcal W^s_\vc(\mathbb R^2)} := \|(I+\Delta_{\mathbb R^2})^{s/2}f\|_{L^\vc(\mathbb R^2)}$.
\end{lem}

\begin{proof} 
	Since the proof is quite standard, we just sketch the main ideas.
	
	The proof of \eqref{eq1-kernel mLL} is similar to that of Lemma 3.5 in \cite{Sikora} (see also \cite[Lemma 4.3]{BD2}).
	
	For \eqref{eq2-kernel mLL}, we write
	\[
	F(\lambda_1,\lambda_2) = e^{-(\lambda_1 +\lambda_2)/R^2} \widetilde F(\lambda_1,\lambda_2),
	\]
	where $\widetilde F(\lambda_1,\lambda_2) = e^{(\lambda_1 +\lambda_2)/R^2} F(\lambda_1,\lambda_2)$.
	
	It follows that 
	\[
	F(L_1,L_2) = [e^{-L/R^2}\otimes e^{-L/R^2}] \circ \widetilde F(L_1,L_2).
	\]
	At this stage, using \eqref{eq1-kernel mLL}, the condition (A2), and the fact 
	$$|F(R^2\cdot, R^2\cdot)\|_{\mathcal W^{s+1}_{\vc}(\mathbb R^2)} \simeq 
	\|\widetilde F(R^2\cdot, R^2\cdot)\|_{\mathcal W^{s+1}_{\vc}(\mathbb R^2)},
	$$ 
	we obtain \eqref{eq2-kernel mLL}.
	
	This completes the proof.
\end{proof}
\begin{lem}\label{lem-kernel m(L,L)}
	Let $\boldsymbol{\mathrm{m}}$ satisfy \eqref{eq-condition on m}. Then for all $f,g\in \mathcal S_\vc$, the function 
	\[
	\boldsymbol{\mathrm{x}} \mapsto \boldsymbol{\mathrm{m}}(L_1,L_2)(f\otimes g)(\boldsymbol{\mathrm{x}})
	\] 
	is continuous on $(X\times X, \boldsymbol{\mathrm{d}})$.
\end{lem}

\begin{proof}
	Fix $\boldsymbol{\mathrm{x}}, \overline{\boldsymbol{\mathrm{x}}}\in X\times X$ with $\boldsymbol{\mathrm{d}}(\boldsymbol{\mathrm{x}},\overline {\boldsymbol{\mathrm{x}}})<1$, and let $N>|\gamma|+\delta_0$. 
	Let $\psi\in C^\infty(\mathbb{R}^d)$ be a function satisfying $0\le\psi\le1$ and
	\[
	\supp\psi(\xi)\subset [1/2,2],\qquad \sum_{j\in\mathbb Z}\psi_j(\xi) = 1,
	\]
	where $\psi_j(\lambda)=\psi(2^{-j}\lambda)$.
	
	Then we can write
	\[
	\begin{aligned}
		\boldsymbol{\mathrm{m}}(\lambda_1,\lambda_2)&=\sum_{j\in \mathbb Z}\boldsymbol{\mathrm{m}}(\lambda_1,\lambda_2)\psi_{j}(\lambda_1+\lambda_2)\\
		&=\sum_{j\ge 0} 2^{-2jN} \boldsymbol{\mathrm{m}}(\lambda_1,\lambda_2)\psi^1_{j,N}(\lambda_1+\lambda_2)(\lambda_1+\lambda_2)^{2N}\\
		& \ \ +\sum_{j< 0} 2^{2jN} \boldsymbol{\mathrm{m}}(\lambda_1,\lambda_2)\psi^2_{j,N}(\lambda_1+\lambda_2)(\lambda_1+\lambda_2)^{-2N},
	\end{aligned}
	\]
	where $\psi^1_{j,N}(\lambda)=(2^{-j}\lambda)^{-2N}\psi_j(\lambda)$ and $\psi^2_{j,N}(\lambda)=(2^{-j}\lambda)^{2N}\psi_j(\lambda)$.
	
	Consequently, we have
	\[
	\begin{aligned}
		\boldsymbol{\mathrm{m}}(L_1,L_2)(f\otimes g)(\boldsymbol{\mathrm{x}})&=\sum_{j\ge 0\atop 2^{j/2} \boldsymbol{\mathrm{d}}(\boldsymbol{\mathrm{x}},\overline {\boldsymbol{\mathrm{x}}})<1} 2^{-2jN} \boldsymbol{\mathrm{m}}(L_1,L_2)\psi^1_{j,N}(L_1+L_2)[(L_1+L_2)^{2N}(f\otimes g)](\boldsymbol{\mathrm{x}})\\ 	
		& \ \ +\sum_{j\ge 0\atop 2^{j/2} \boldsymbol{\mathrm{d}}(\boldsymbol{\mathrm{x}},\overline {\boldsymbol{\mathrm{x}}})\ge 1} 2^{-2jN} \boldsymbol{\mathrm{m}}(L_1,L_2)\psi^1_{j,N}(L_1+L_2)[(L_1+L_2)^{2N}(f\otimes g)](\boldsymbol{\mathrm{x}})\\
		& \ \ +\sum_{j< 0} 2^{2jN} \boldsymbol{\mathrm{m}}(L_1,L_2)\psi^2_{j,N}(L_1+L_2)[(L_1+L_2)^{-2N}(f\otimes g)](\boldsymbol{\mathrm{x}})\\
		&=: \sum_{j\ge 0\atop 2^{j/2} \boldsymbol{\mathrm{d}}(\boldsymbol{\mathrm{x}},\overline {\boldsymbol{\mathrm{x}}})<1} 2^{-2jN}\boldsymbol{\mathrm{m}}_{j,1}(L_1,L_2)[(L_1+L_2)^{2N}(f\otimes g)](\boldsymbol{\mathrm{x}})\\ 	
		& \ \ +\sum_{j\ge 0\atop 2^{j/2} \boldsymbol{\mathrm{d}}(\boldsymbol{\mathrm{x}},\overline {\boldsymbol{\mathrm{x}}})\ge 1} 2^{-2jN} \boldsymbol{\mathrm{m}}_{j,1}(L_1,L_2)[(L_1+L_2)^{2N}(f\otimes g)](\boldsymbol{\mathrm{x}})\\
		& \ \ +\sum_{j< 0} 2^{2jN} \boldsymbol{\mathrm{m}}_{j,2}(L_1,L_2)[(L_1+L_2)^{-2N}(f\otimes g)](\boldsymbol{\mathrm{x}}),
	\end{aligned}
	\]
	which leads to
	\[
	\begin{aligned}
		&\Big|\boldsymbol{\mathrm{m}}(L_1,L_2)(f\otimes g)(\boldsymbol{\mathrm{x}})-\boldsymbol{\mathrm{m}}(L_1,L_2)(f\otimes g)(\overline{\boldsymbol{\mathrm{x}}})\Big|\\
		&\le \sum_{j\ge 0\atop 2^{j/2} \boldsymbol{\mathrm{d}}(\boldsymbol{\mathrm{x}},\overline {\boldsymbol{\mathrm{x}}})<1} 2^{-2jN}\Big|\boldsymbol{\mathrm{m}}_{j,1}(L_1,L_2)[(L_1+L_2)^{2N}(f\otimes g)](\boldsymbol{\mathrm{x}})-\boldsymbol{\mathrm{m}}_{j,1}(L_1,L_2)[(L_1+L_2)^{2N}(f\otimes g)](\overline{\boldsymbol{\mathrm{x}}})\Big|\\ 	
		& \ \ +\sum_{j\ge 0\atop 2^{j/2} \boldsymbol{\mathrm{d}}(\boldsymbol{\mathrm{x}},\overline {\boldsymbol{\mathrm{x}}})\ge 1} 2^{-2jN}  \big|\boldsymbol{\mathrm{m}}_{j,2}(L_1,L_2)[(L_1+L_2)^{2N}(f\otimes g)](\boldsymbol{\mathrm{x}})\big|\\
		& \ \ +\sum_{j\ge 0\atop 2^{j/2} \boldsymbol{\mathrm{d}}(\boldsymbol{\mathrm{x}},\overline {\boldsymbol{\mathrm{x}}})\ge 1} 2^{-2jN}  \big|\boldsymbol{\mathrm{m}}_{j,2}(L_1,L_2)[(L_1+L_2)^{2N}(f\otimes g)](\overline{\boldsymbol{\mathrm{x}}})\big|\\
		& \ \ +\sum_{j< 0} 2^{2jN}  \big|\boldsymbol{\mathrm{m}}_{j,3}(L_1,L_2)[(L_1+L_2)^{-2N}(f\otimes g)](\boldsymbol{\mathrm{x}})-\boldsymbol{\mathrm{m}}_{j,3}(L_1,L_2)[(L_1+L_2)^{-2N}(f\otimes g)](\overline{\boldsymbol{\mathrm{x}}})\big| \\ 
		&=:E_1+E_2+E_3+E_4.
	\end{aligned}
	\]
	
	Applying Lemma \ref{lem-kernel m(L,L)}, we obtain
	\[
	\begin{aligned}
		E_1&\lesi \sum_{j\ge 0\atop 2^{j/2} \boldsymbol{\mathrm{d}}(\boldsymbol{\mathrm{x}},\overline {\boldsymbol{\mathrm{x}}})<1} 2^{-2jN}(2^{j/2}\boldsymbol{\mathrm{d}}(\boldsymbol{\mathrm{x}},\overline {\boldsymbol{\mathrm{x}}})^{\delta_0}  \|\boldsymbol{\mathrm{m}}_{j,1}(2^j\cdot,2^j\cdot)\|_{\mathcal W^{2n+2}_\infty}\\
		&\quad \times \int_{X\times X}\frac{(1+2^{j/2}\textbf{d}(\boldsymbol{\mathrm{x}},\textbf{y}))^{-2n+1}}{\mu^2(B(\textbf{y},2^{-j/2}))}[(L_1+L_2)^{2N}(f\otimes g)](y_1,y_2)d\mu(y_1)d\mu(y_2),
	\end{aligned}
	\] 
	where $\textbf{y}=(y_1,y_2)\in X\times X$.
	
	From \eqref{eq-condition on m}, it follows that
	\[
	\|\boldsymbol{\mathrm{m}}_{j,1}(2^j\cdot,2^j\cdot)\|_{\mathcal W^{2n+2}_\infty}\lesi 2^{j\gamma}.
	\]
	
	Hence, we further obtain
	\[
	\begin{aligned}
		E_1&\lesi \sum_{j\ge 0} 2^{-j(N-\gamma-\delta_0)}\boldsymbol{\mathrm{d}}(\boldsymbol{\mathrm{x}},\overline {\boldsymbol{\mathrm{x}}})^{\delta_0}\|(L_1+L_2)^{2N}(f\otimes g)\|_{L^\vc}\\
		& \quad \times  \int_{X\times X}\frac{(1+2^{j/2}\textbf{d}(\boldsymbol{\mathrm{x}},\textbf{y}))^{-2n+1}}{\mu^2(B(\textbf{y},2^{-j/2}))} d\mu(y_1)d\mu(y_2)\\
		&\lesi \boldsymbol{\mathrm{d}}(\boldsymbol{\mathrm{x}},\overline {\boldsymbol{\mathrm{x}}})^{\delta_0}  \|(L_1+L_2)^{2N}(f\otimes g)\|_{L^\vc},
	\end{aligned}
	\]
	provided that $N>|\gamma|+\delta_0$.
	
	Similarly, we have
	\[
	\begin{aligned}
		E_4&\lesi \sum_{j<0} 2^{j(N+\gamma+\delta_0)}d(x,\overline x)^{\delta_0}   \|(L_1+L_2)^{-2N}(f\otimes g)\|_{L^\vc}\\
		&\lesi \boldsymbol{\mathrm{d}}(\boldsymbol{\mathrm{x}},\overline {\boldsymbol{\mathrm{x}}})^{\delta_0}   \|(L_1+L_2)^{-2N}(f\otimes g)\|_{L^\vc},
	\end{aligned}
	\]
	as long as $N>|\gamma|+\delta_0$.
	
	For $E_2$ and $E_3$, applying Lemma \ref{lem-kernel m(L,L)} again and using a similar argument, we get
	\[
	\begin{aligned}
		E_2&\lesi \sum_{j\ge 0\atop 2^{-j/2}\le \boldsymbol{\mathrm{d}}(\boldsymbol{\mathrm{x}},\overline {\boldsymbol{\mathrm{x}}})} 2^{-jN} \| \boldsymbol{\mathrm{m}}_{j,2}(2^j\cdot,2^j\cdot)\|_{\mathcal W^{2n+2}_\infty}\\
		&\quad \times \int_{X\times X}\frac{(1+2^{j/2}\textbf{d}(\textbf{x},\textbf{y}))^{-2n+1}}{\mu^2(B(\textbf{y},2^{-j/2}))}[(L_1+L_2)^{2N}(f\otimes g)](y_1,y_2)d\mu(y_1)d\mu(y_2)\\
		&\lesi \sum_{j\ge 0\atop 2^{-j/2}\le \boldsymbol{\mathrm{d}}(\boldsymbol{\mathrm{x}},\overline {\boldsymbol{\mathrm{x}}})} 2^{-j(N-\gamma)}\|(L_1+L_2)^{2N}(f\otimes g)\|_{L^\vc}\\
		&\lesi \sum_{j\ge 0\atop 2^{-j/2}\le \boldsymbol{\mathrm{d}}(\boldsymbol{\mathrm{x}},\overline {\boldsymbol{\mathrm{x}}})} 2^{-j(N-\gamma-\delta_0/2)}\boldsymbol{\mathrm{d}}(\boldsymbol{\mathrm{x}},\overline {\boldsymbol{\mathrm{x}}})^{\delta_0} \|(L_1+L_2)^{2N}(f\otimes g)\|_{L^\vc}\\
		&\lesi \boldsymbol{\mathrm{d}}(\boldsymbol{\mathrm{x}},\overline {\boldsymbol{\mathrm{x}}})^{\delta_0} \|(L_1+L_2)^{2N}(f\otimes g)\|_{L^\vc},
	\end{aligned}
	\]
	provided $N>|\gamma|+\delta_0$.
	
	Similarly,
	\[
	E_3\lesi \boldsymbol{\mathrm{d}}(\boldsymbol{\mathrm{x}},\overline {\boldsymbol{\mathrm{x}}})^{\delta_0} \|(L_1+L_2)^{2N}(f\otimes g)\|_{L^\vc}.
	\]
	
	Combining the estimates for $E_1, E_2, E_3, E_4$, we conclude
	\[
	\Big|\boldsymbol{\mathrm{m}}(L_1,L_2)(f\otimes g)(\textbf{x})-\boldsymbol{\mathrm{m}}(L_1,L_2)(f\otimes g)(\overline {\textbf{x}})\Big|\le C(f,g,N)  \boldsymbol{\mathrm{d}}(\boldsymbol{\mathrm{x}},\overline {\boldsymbol{\mathrm{x}}})^{\delta_0},
	\]
	which shows that the function $$\textbf{x}\mapsto \boldsymbol{\mathrm{m}}(L_1,L_2)(f\otimes g)(\textbf{x})$$ is continuous.
	
	This completes the proof.
\end{proof}

\subsection{Appendix B}
\begin{proof}[Proof of Lemma \ref{lem-Ds Dt}:]
	Without loss of generality, we might assume $s\ge t>0$. We write
	\[
	\int_{X}D_{s, N}(x,z)D_{t, N}(z,y)d\mu(z)=\int_{d(x,z)\ge d(x,y)/4}\ldots +\int_{d(x,z)< d(x,y)/4}\ldots =:E+F.
	\]
	By \eqref{doubling3} and Lemma \ref{lem-elementary},
	\[
	E\lesi D_{s, N-n}(x,y)\int_{X}D_{t, N}(z,y)d\mu(z)\lesi  D_{s, N-n}(x,y).
	\]
	In addition, $d(x,y)\simeq d(z,y)$ as long as $d(x,z)< d(x,y)/4$. Hence,
	\[
	F\lesi D_{t, N}(x,y)\int_{d(x,z)< d(x,y)/4}D_{s, N}(x,z)d\mu(z).
	\]
	If $d(x,y)\ge s \ge t$,  then
	\[
	\begin{aligned}
		D_{t, N}(x,y)&\simeq \f{1}{V(x\vee y,t)}\Big(\f{d(x,y)}{t}\Big)^{-N}\\
		&\lesi \f{1}{V(x\vee y,s)}\Big(\f{d(x,y)}{s}\Big)^{-N}\Big(\f{t}{s}\Big)^{N-n}\\
		&\lesi \f{1}{V(x\vee y,s)}\Big(\f{d(x,y)}{s}\Big)^{-N}\simeq D_{s,N}(x,y).	
	\end{aligned}
	\]
	This, together with Lemma \ref{lem-elementary}, yields
	\[
	F\lesi D_{s,N}(x,y).	
	\]
	If $t\le d(x,y)<s$, then  using the facts
	\[
	D_{s,N}(x,z)\simeq \f{1}{V(x\vee y,s)}\simeq \f{1}{V(x,s)}, \ \ d(x,z)<d(x,y)/4\le s/4
	\]
	and the doubling condition \eqref{doubling2}, we have
	\[
	\begin{aligned}
		F&\lesi D_{t, N}(x,y)\int_{d(x,z)< d(x,y)/4}\f{1}{V(x,s)} d\mu(z)\\
		&\lesi \f{1}{V(x,t)}\Big(\f{d(x,y)}{t}\Big)^{-N} \f{V(x,d(x,y))}{V(x,s)}=\f{1}{V(x,s)}\Big(\f{d(x,y)}{t}\Big)^{-N} \f{V(x,d(x,y))}{V(x,t)}\\
		&\lesi \f{1}{V(x,s)}\Big(\f{d(x,y)}{t}\Big)^{-N+n}\\
		&\lesi \f{1}{V(x,s)}\simeq D_{s,N}(x,y).
	\end{aligned}
	\]
	If $d(x,y)<t\le s$, then 
	\[
	D_{t,N}(x,y)\simeq \f{1}{V(x,t)} \ \text{and} \ \ \ D_{s,N}(x,z)\simeq \f{1}{V(x,s)}, \ \ d(x,z)<d(x,y)/4\le t/4 \le s/4,
	\]
	we have
	\[
	\begin{aligned}
		F &\lesi \f{1}{V(x,t)}  \f{V(x,d(x,y))}{V(x,s)}=\f{1}{V(x,s)}  \f{V(x,d(x,y))}{V(x,t)}\\
		&\lesi \f{1}{V(x,s)}\simeq D_{s,N}(x,y).
	\end{aligned}
	\]
	This completes our proof.
\end{proof}

\begin{proof}[Proof of Lemma \ref{lem-composite kernel}:]
	\noindent (a) By the H\"older inequality,
	\[
	\begin{aligned}
		\int_X &D_{2^{-\ell},2N}(x,y)D_{2^{-k},N}(y,z)D_{2^{-k},N}(y,u)d\mu(y)\\
		&\le \Big[\int_X D_{2^{-\ell},N}(x,y)D_{2^{-k},N}(y,z)^2d\mu(y)\Big]^{1/2}\Big[\int_X D_{2^{-\ell},N}(x,y) D_{2^{-k},N}(y,u)^2d\mu(y)\Big]^{1/2}.
	\end{aligned}
	\]
	On the other hand,
	\[
	D_{2^{-k},N}(y,z)^2 \le \f{1}{V(z,2^{-k})}D_{2^{-k},N}(y,z) \ \ \text{and} \ \ 	D_{2^{-k},N}(y,u)^2\lesi \f{1}{V(u,2^{-k})}D_{2^{-k},N}(y,u).
	\]
	Hence,
	\[
	\begin{aligned}
		\int_X &D_{2^{-\ell},N}(x,y)D_{2^{-k},N}(y,z)D_{2^{-k},N}(y,u)d\mu(y)\\
		&\lesi \f{1}{\sqrt{V(u,2^{-k}) V(z,2^{-k})}}\Big[\int_X D_{2^{-\ell},N}(x,y)D_{2^{-k},N}(y,z) d\mu(y)\Big]^{1/2}\\
		& \ \ \times\Big[\int_X D_{2^{-\ell},N}(x,y) D_{2^{-k},N}(y,u) d\mu(y)\Big]^{1/2}\\
		&\lesi \f{1}{\sqrt{V(u,2^{-k}) V(z,2^{-k})}}\Big[ D_{2^{-k},N-n}(x,z)D_{2^{-k},N-n}(x,u)  \Big]^{1/2}\\
		&\lesi   D_{2^{-k},\f{N-2n}{2}}(x,z)D_{2^{-k},\f{N-2n}{2}}(x,u), 
	\end{aligned}
	\]	
	where in the second inequality we used Lemma \ref{lem-Ds Dt} and in the last inequality we used \eqref{doubling3}.

	(b) By the H\"older inequality,
	\[
	\begin{aligned}
		\int_X &D_{2^{-\ell},2N}(x,y)D_{2^{-k},N}(y,z)D_{2^{-k},N}(y,u)d\mu(y)\\
		&\le \Big[\int_X D_{2^{-\ell},N}(x,y)^2D_{2^{-k},N}(y,z)d\mu(y)\Big]^{1/2}\Big[\int_X D_{2^{-k},N}(y,z) D_{2^{-k},N}(y,u)^2d\mu(y)\Big]^{1/2}\\
		&\lesi \f{1}{\sqrt{V(x,2^{-\ell})V(u,2^{-k})}}\Big[\int_X D_{2^{-\ell},N}(x,y)D_{2^{-k},N}(y,z)d\mu(y)\Big]^{1/2}\\
		& \ \ \times\Big[\int_X D_{2^{-k},N}(u,y) D_{2^{-k},N}(y,z) d\mu(y)\Big]^{1/2}.
	\end{aligned}
	\]
	Applying Lemma \ref{lem-Ds Dt} and \eqref{doubling3},
	\[
	\begin{aligned}
		\int_X &D_{2^{-\ell},2N}(x,y)D_{2^{-k},N}(y,z)D_{2^{-k},N}(y,u)d\mu(y)\\
		&\lesi \f{1}{\sqrt{V(x,2^{-\ell})V(u,2^{-k})}}\Big[ D_{2^{-\ell},N-n}(x,z) D_{2^{-k},N-n}(u,z)  \Big]^{1/2}\\
		&\lesi  \widetilde  D_{2^{-\ell},\f{N-2n}{2}}(x,z) D_{2^{-k},\f{N-2n}{2}}(z,u). 
	\end{aligned}
	\]
	This completes the proof.
\end{proof}

\begin{proof}[Proof of Proposition {prop1-maximal function}:]
	It is obvious that 
	\[
	\psi_j(\sqrt{L})=\sum_{k=j-2}^{j+3}\psi_j(\sqrt{L})\varphi_k(\sqrt{L}).
	\]
	Fix $\lambda>0$. We set $\ell = \lfloor 2n+\lambda \rfloor +2$ and $C_{\psi,\ell}= \|\psi\|_\vc + \|\psi^{(\ell)}\|_{L^\vc}$. By Proposition \ref{thm-kernel estimate for functional calculus} we have,  for $y\in X$ and $N=\ell -\lambda >n$,
	\begin{equation}\label{eq1-pro1}
		\begin{aligned}
			|\psi_j(\sqrt{L})f(y)|&\leq \sum_{k=j-2}^{j+3}|\psi_j(\sqrt{L})\varphi_k(\sqrt{L})f(y)|\\
			&\lesi  C_{\psi,\ell}\sum_{k=j-2}^{j+3}\int_X \f{1}{V(y,2^{-j})}(1+2^jd(y,z))^{-N-\lambda}|\varphi_k(\sqrt{L})f(z)|d\mu(z).
		\end{aligned}
	\end{equation}
	It follows that
	\begin{equation}\label{eq2-pro1}
		\begin{aligned}
			\f{|\psi_j(\sqrt{L})f(y)|}{(1+d(x,y)/s)^\lambda}&\lesi  C_{\psi,\ell}\sum_{k=j-2}^{j+3}\int_X \f{1}{V(y,2^{-j})}(1+2^jd(y,z))^{-N-\lambda}\f{|\varphi_k(\sqrt{L})f(z)|}{(1+2^jd(x,y))^\lambda}d\mu(z)\\
			&\lesi C_{\psi,\ell} \sum_{k=j-2}^{j+3}\int_X \f{1}{V(y,2^{-j})}(1+2^jd(y,z))^{-N}\f{|\varphi_k(\sqrt{L})f(z)|}{(1+2^jd(x,y))^\lambda(1+2^jd(y,z))^\lambda}d\mu(z)\\
			&\lesi C_{\psi,\ell} \sum_{k=j-2}^{j+3}\int_X \f{1}{V(y,2^{-j})}(1+2^jd(y,z))^{-N}\f{|\varphi_k(\sqrt{L})f(z)|}{(1+2^jd(x,z))^\lambda}d\mu(z).\\
		\end{aligned}
	\end{equation}
	Using the fact that $2^k\simeq 2^j$, we obtain 
	\begin{equation}\label{eq2s-pro1}
		\begin{aligned}
			\f{|\psi_j(\sqrt{L})f(y)|}{(1+d(x,y)/s)^\lambda}&\lesi  C_{\psi,\ell}\sum_{k=j-2}^{j+3}\int_X \f{1}{V(y,2^{-j})}(1+2^jd(y,z))^{-N}\f{|\varphi_k(\sqrt{L})f(z)|}{(1+2^kd(x,z))^\lambda}d\mu(z)\\
			&\lesi C_{\psi,\ell}\sum_{k=j-2}^{j+3} \varphi^*_{k,\lambda}(\sqrt{L})f(x)\int_X \f{1}{V(y,2^{-j})}(1+2^jd(y,z))^{-N}d\mu(z)\\
			&\lesi C_{\psi,\ell}\sum_{k=j-2}^{j+3} \varphi^*_{k,\lambda}(\sqrt{L})f(x),
		\end{aligned}
	\end{equation}
	where in the last inequality we use Lemma \ref{lem-elementary}.
	Taking the supremum over $y\in X$, we derive \eqref{eq-psistar vaphistar}.
\end{proof}

\subsection{Appendix C} This section is devoted to proving Lemma \ref{lem-kernel-representation} and the inequality \eqref{eq- equivalence of maximal functions}. Although some ideas are inspired from \cite{PK}, igorous computations are required to keep track of explicit constants.

\begin{prop}\label{prop1-frame}
	Let $\Gamma$ be as in Lemma \ref{lem-exp kernel}. For $j \in \mathbb Z$ and $1 \leq p < \infty$,  we have
	\begin{equation} \label{eq:4.3}
		\sum_{Q \in \mathcal D_{j}} \int_{Q}  |\Gamma_j(\sqrt L)f(x) - \Gamma_j(\sqrt L)f(x_Q)|^p d\mu(x) \leq (a_nb_\sharp c_{n,p}^\star 2^{-\delta_0 k_0})^p \|\Gamma_j(\sqrt L)f\|_p^p.
	\end{equation}
	%and for any $f \in \Sigma^\infty_\lambda$,
	%\begin{equation} \label{eq:4.4}
	%	\sup_{\xi \in X_\delta} \sup_{x \in A_\xi} |f(x) - f(\xi)| \leq K(\sigma^*) \gamma^\alpha  c_\diamond \|f\|_\infty,
	%\end{equation}
	where $a_n, b_\sharp$ and $c^\star_{n,p}$ are constants in Lemma \ref{lem-elementary}, Remark \ref{rem1} and \eqref{eq-sharp weighted estimates for M}, respectively.
\end{prop}

\begin{proof}
	Let $\Sigma$ be the function in Remark \ref{rem1}. Then we have
	\[
	\Gamma_j(\sqrt L)f(x) = \int_X \Sigma_j(\sqrt L)(x, y)[\Gamma_j(\sqrt L)f](y)\, d\mu(y).
	\]
	Using Theorem \ref{thm-kernel estimate for functional calculus}, Lemma \ref{lem-elementary} and \eqref{eq-sharp weighted estimates for M},
	\[
	\begin{aligned}
		\sum_{Q \in \mathcal D_j} \int_{Q}&  |\Gamma_j(\sqrt L)f(x) - \Gamma_j(\sqrt L)f(x_Q)|^p d\mu(x) \\
		&= \sum_{Q \in \mathcal D_{j}} \int_{Q}  \left| \int_X \left[\Sigma_j(\sqrt L)(x, y) - \Sigma_j(\sqrt L)(x_Q, y)\right] [\Gamma_j(\sqrt L)f](y) d\mu(y) \right|^p d\mu(x) \\
		&\le b^p_\sharp\sum_{Q \in \mathcal D_{j}} \int_{Q} \left| \int_X [2^{j}d(x,x_Q)]^\alpha  D_{2^{-j},n+1}(x,y)[\Gamma_j(\sqrt L)f](y) d\mu(y) \right|^p d\mu(x) \\
		&\le (a_nb_\sharp  2^{-\delta_0 k_0})^p\sum_{Q \in \mathcal D_{j}} \int_{Q}|\mathcal M[\Gamma_j(\sqrt L)f](x)|^pd\mu(x)=(b_\sharp a_n 2^{-\delta_0 k_0})^p\|\mathcal M[\Gamma_j(\sqrt L)f]\|_p\\
		&\le (a_nb_\sharp c_{n,p}^\star 2^{-\delta_0 k_0})^p\|\Gamma_j(\sqrt L)f\|_p^p.
	\end{aligned}
	\]
	This completes our proof.
\end{proof}

\begin{prop}\label{prop2-frame}
	Let $\Gamma$ be as in Lemma \ref{lem-exp kernel}. Let $\varepsilon>0$ such that
	\begin{equation} \label{eq:4.8}
		a_nb_\sharp c_{n,2}^\star 2^{-\delta_0 k_0} \le \varepsilon/4,
	\end{equation}
	where $a_n, b_\sharp$ and $c^\star_{n,2}$ are constants in Lemma \ref{lem-elementary}, Remark \ref{rem1} and \eqref{eq-sharp weighted estimates for M} for $p=2$, respectively. Then we have
	\begin{equation} \label{eq:4.9}
		(1 - \varepsilon) \|\Gamma_j(\sqrt L)f\|_2^2 \leq   \sum_{Q \in \mathcal D_{j}} V(Q)|\Gamma_j(\sqrt L)f(x_Q)|^2   \leq (1 + \varepsilon) \|\Gamma_j(\sqrt L)f\|_2^2.
	\end{equation}
\end{prop}

\begin{proof}
	Note that 
	\begin{equation}\label{eq-Minkowski inequality}
		|a|^2  \le (1+\delta)\Big[\f{1}{\delta}|a-b|^2+|b|^2 \Big]
	\end{equation}
	Applying this with $\delta := \varepsilon/2$, 
	\[
	\begin{aligned}
	\int_{Q} |\Gamma_j(\sqrt L)f(x)|^2 d\mu(x) &\leq  \f{2+\varepsilon}{\varepsilon} \int_{Q} |\Gamma_j(\sqrt L)f(x) - \Gamma_j(\sqrt L)f(x_Q)|^p d\mu(x)\\
	& \ \ \ \ \  + (1+\varepsilon/2)V(Q)|\Gamma_j(\sqrt L)f(x_Q)|^2.
	\end{aligned}
	\]
	This, together with Proposition \ref{prop1-frame}, implies
	\[
	\begin{aligned}
		\|\Gamma_j(\sqrt L)f\|_2^2&=\sum_{Q\in \mathcal D_j}\int_{Q} |\Gamma_j(\sqrt L)f(x)|^2 d\mu(x)\\
		&\le \f{(a_nb_\sharp c_{n,2}^\star 2^{-\delta_0 k_0})^2(2+\varepsilon)}{\varepsilon}\|\Gamma_j(\sqrt L)f\|_2^2 +(1+\varepsilon/2)\sum_{Q\in \mathcal D_j}V(Q)|\Gamma_j(\sqrt L)f(x_Q)|^2, 
	\end{aligned}
	\]
	which implies
	\[
	\begin{aligned}
		(1+\varepsilon/2)\sum_{Q\in \mathcal D_j}V(Q)|\Gamma_j(\sqrt L)f(x_Q)|^2&\ge \Big[1-\f{(a_nb_\sharp c_{n,2}^\star 2^{-\delta_0 k_0})^2 (2+\varepsilon)}{\varepsilon}\Big]\|\Gamma_j(\sqrt L)f\|_2^2\\
		&\ge  [1-\varepsilon(2+\varepsilon)/8]\|\Gamma_j(\sqrt L)f\|_2^2,
	\end{aligned}
	\]
	as long as $a_nb_\sharp c_{n,2}^\star 2^{-\delta_0 k_0} \le \varepsilon/4$.
	
	It is equivalent to that 
	\[
	\begin{aligned}
		\sum_{Q\in \mathcal D_j}V(Q)|\Gamma_j(\sqrt L)f(x_Q)|^2&\ge \f{1-\varepsilon/4 -\varepsilon^2/8}{1+\varepsilon/2}\|\Gamma_j(\sqrt L)f\|_2^2\\
		&\ge (1-\varepsilon)\|\Gamma_j(\sqrt L)f\|_2^2.
	\end{aligned}
	\]
	
	Similarly, using \eqref{eq-Minkowski inequality} for $\delta=\varepsilon/2 $ and Proposition \ref{prop1-frame} again,
	\[
	\begin{aligned}
		\sum_{Q\in \mathcal D_j}V(Q)|\Gamma_j(\sqrt L)f(x_Q)|^2&=\int_{Q} |\Gamma_j(\sqrt L)f(x_Q)|^2 d\mu(x)\\
		&\leq  \f{2+\varepsilon}{\varepsilon} \sum_{Q\in \mathcal D_j}\int_{Q} |\Gamma_j(\sqrt L)f(x) - \Gamma_j(\sqrt L)f(x_Q)|^p d\mu(x)\\
		& \ \  + (1+\varepsilon/2)\sum_{Q\in \mathcal D_j}\int_Q|\Gamma_j(\sqrt L)f(x)|^2d\mu(x)\\
		&\le \f{(a_nb_\sharp c_{n,2}^\star 2^{-\delta_0 k_0})^2(2+\varepsilon)}{\varepsilon}\|\Gamma_j(\sqrt L)f\|_2^2 +(1+\varepsilon/2)\|\Gamma_j(\sqrt L)f\|_2^2\\
		&\le (1+\varepsilon)\|\Gamma_j(\sqrt L)f\|_2^2, 
	\end{aligned}
	\]
	as long as $a_nb_\sharp c_{n,2}^\star 2^{-\delta_0 k_0} \le \varepsilon/4$.
	
	This completes our proof.
\end{proof}
 
 We are ready to give the proof of Lemma \ref{lem-kernel-representation}.
\begin{proof}[Proof of Lemma \ref{lem-kernel-representation}:]
	Let $\Gamma$ be as in Lemma \ref{lem-exp kernel}. Set $\varepsilon = \f{1}{4c^2_\sharp c_\diamond}$, where $c_\diamond$ and $c_\sharp$ is the constant in Lemma \ref{lem-exp kernel} and Lemma \ref{lem-product exp}, respectively. From \eqref{eq-constant k0}, the condition \eqref{eq:4.8} is satisfied. Hence, by Proposition \ref{prop2-frame}, 
	\[
	(1 - \varepsilon)\|\Gamma_j(\sqrt L)f\|_2^2 
	\leq \sum_{Q \in \mathcal D_j} V(Q) |\Gamma_j(\sqrt L)f(x_Q)|^2 
	\leq (1 + \varepsilon)\|\Gamma_j(\sqrt L)f\|_2^2.
	\]
	With $\omega_Q := \frac{1}{1+\varepsilon}V(Q)$, we obtain
	\begin{equation}\label{eq1-lem2.5}
		(1 - 2\varepsilon)\|\Gamma_j(\sqrt L)f\|_2^2 
		\leq \sum_{Q \in \mathcal D_j} \omega_Q |\Gamma_j(\sqrt L)f(x_Q)|^2 
		\leq \|\Gamma_j(\sqrt L)f\|_2^2.
	\end{equation}
	
	Define
	\[
	V_jf(x) = \int_X V_j(x,y)f(y)\,d\mu(y),
	\]
	where
	\[
	V_j(x,y) = \sum_{Q\in \mathcal D_j} \omega_Q \Gamma_j(\sqrt L)(x,x_Q)\Gamma_j(\sqrt L)(x_Q,y).
	\]
	Set
	\[
	R_j := \Gamma_j(\sqrt L)^2 - V_j.
	\]
	
	Note that
	\[
	\langle R_jf,f \rangle = \|\Gamma_j(\sqrt L)f\|_2^2 - \langle V_jf,f \rangle. 
	\]
	
	Moreover,
	\[
	\langle V_jf,f \rangle = \sum_{Q\in \mathcal D_j} \omega_Q |\Gamma_j(\sqrt L)f(x_Q)|^2,
	\]
	which, together with \eqref{eq1-lem2.5}, implies
	\[
	(1-2\varepsilon)\|\Gamma_j(\sqrt L)f\|_2^2 \le \langle V_jf,f \rangle \le \|\Gamma_j(\sqrt L)f\|_2^2.
	\] 
	Consequently,
	\[
	\langle R_jf,f \rangle \le 2\varepsilon \|\Gamma_j(\sqrt L)f\|_2^2 \le 2\varepsilon \|f\|_2^2. 
	\]
	Thus $\|R_j\|_{2\to 2}\le 2\varepsilon$.
	
	Define
	\[
	T_j := ({\rm Id} - R_j)^{-1} = {\rm Id} + \sum_{k \geq 1} R_j^k =: {\rm Id} + S_j.
	\]
	Since $\Gamma_j^2(\sqrt L)\psi_j(\sqrt L)f = \psi_j(\sqrt L)f$, we obtain
	\[
	\begin{aligned}
		\psi_j(\sqrt L)f 
		&= T_j\circ \psi_j(\sqrt L)f - T_j\circ R_j\circ \psi_j(\sqrt L)f\\
		&= T_j\circ \psi_j(\sqrt L)f - T_j\circ \Gamma_j^2(\sqrt L)\psi_j(\sqrt L)f + T_j\circ V_j\circ \psi_j(\sqrt L)f\\
		&= T_j\circ V_j \circ\psi_j(\sqrt L)f.
	\end{aligned}
	\]
	
	Furthermore,
	\[
	\begin{aligned}
		V_j\circ\psi_j(\sqrt L)f(x)
		&= \sum_{Q\in \mathcal D_j} \omega_Q \Gamma_j(\sqrt L)(x,x_Q)\Gamma_j(\sqrt L)\psi_j(\sqrt L)f(x_Q)\\
		&= \sum_{Q\in \mathcal D_j} \omega_Q \Gamma_j(\sqrt L)(x,x_Q)\psi_j(\sqrt L)f(x_Q),
	\end{aligned}
	\]
	where we used $\Gamma_j(\sqrt L)\psi_j(\sqrt L)f= \psi_j(\sqrt L)f$.
	
	Therefore,
	\[
	\psi_j(\sqrt L)f  
	= \sum_{Q\in \mathcal D_j} \omega_Q \psi_j(\sqrt L)f(x_Q)\, T_j\big[\Gamma_j(\sqrt L)(\cdot,x_Q)\big](x).
	\]
	Set 
	\[
	\widetilde \Gamma_j(x,y) := T_j\big[\Gamma_j(\sqrt L)(\cdot,y)\big](x).
	\]
	Then
	\[
	\psi_j(\sqrt L)f(x)  
	= \sum_{Q\in \mathcal D_j} \omega_Q \psi_j(\sqrt L)f(x_Q)\,\widetilde \Gamma_j(x,x_Q),
	\]
	which proves (iii).
	
	\smallskip
	
	We next prove that $\widetilde \Gamma_j$ satisfies (i)-(ii). Firstly, from Lemma \ref{lem-exp kernel} and \ref{lem-product exp}, it follows that for any $k\in \mathbb N$,
	\begin{equation}\label{eq0-kernel-of-Rj}
		|R_j(x,y)| \lesssim 2c_\sharp c_{\diamond}E_{2^{-j}, \kappa_0}(x,y), \qquad x,y\in X,
	\end{equation}
	\begin{equation}\label{eq1-kernel-of-Rj}
		|L^kR_j(x,y)| \lesssim 2^{2jk}E_{2^{-j}, \kappa_0}(x,y), \qquad x,y\in X,
	\end{equation}
	and
	\begin{equation}\label{eq2-kernel-of-Rj}
		|R_j(x,y)-R_j(\overline x,y)| 
		\lesssim (2^j d(x,\overline x))^{\delta_0}  E_{2^{-j}, \kappa_0}(x,y),
	\end{equation}
	for all $x,\overline x,y\in X$ with $d(x,\overline x)<2^{-j}$.
	
	From \eqref{eq0-kernel-of-Rj} and Lemma \ref{lem-product exp}, for each $k\in \mathbb N$, $j\in \mathbb Z$ and for all $x,y\in X$,
	\[
	|R_j^k(x,y)|\le c^{k-1}_\sharp (2c_\sharp c_{\diamond})^kE_{2^{-j}, \kappa_0}(x,y).
	\]
	On the other hand, for $k\ge 2$
	\[
	\begin{aligned}
		|R_j^k(x,y)|&\le \|R_j(x,\cdot)\|_2\|R_j\|_{2\to 2}^{k-2}\|R_j(\cdot,y)\|_2\\
		&\le C\f{(2c_\sharp c_{\diamond})^2\varepsilon^{k-2}}{\sqrt{V(x,2^{-j})V(y,2^{-j})}},
	\end{aligned}
	\]
	where the constant $C$ only depends on $n$ and $\kappa_0$.
	
	Interpolating these two inequalities,
	\begin{equation}\label{eq-kernel of R_jk}
		|R_j^k(x,y)|\lesi (2c^2_\sharp c_\diamond \varepsilon)^{(k-2)/2}E_{2^{-j}, \kappa_0}(x,y) = 2^{-(k-2)/2}E_{2^{-j}, \kappa_0}(x,y),
	\end{equation}
	since $\varepsilon = \f{1}{2c_\sharp}$.
	
	We now prove the estimate (i). Indeed, by Lemma \ref{lem-exp kernel},
	\[
	|\Gamma_j(\sqrt L)(x,y)| \lesssim c_{\diamond}E_{2^{-j}, \kappa_0}(x,y).
	\]	
	Moreover, from $\widetilde \Gamma_j(x,y) := T_j\big[\Gamma_j(\sqrt L)(\cdot,y)\big](x)$ and $T_j={\rm Id} + S_j$, we have
	\[
	\widetilde \Gamma_j(x,y)
	= \Gamma_j(\sqrt L)(x,y) + \int_{X} S_j(x,z)\Gamma_j(\sqrt L)(z,y)\,d\mu(z).
	\]
	Hence, from this and Lemma \ref{lem-product exp}, it suffices to show that 
	\[
	|S_j(x,y)| \lesssim E_{2^{-j}, \kappa_0/2}(x,y).
	\]	
	To do this, using \eqref{eq-kernel of R_jk},
	\[
	\begin{aligned}
		|S_j(x,y)|&\le \sum_{k\ge 1}|R_j^k(x,y)|\\
		&\lesi E_{2^{-j}, \kappa_0/2}(x,y) +\sum_{k\ge 2}2^{-(k-2)/2} E_{2^{-j}, \kappa_0/2}(x,y)\\
		&\lesi E_{2^{-j}, \kappa_0/2}(x,y),
	\end{aligned}
	\]
	as desired, and therefore (i) holds for $m=0$.
	
	In addition, note that 
	\[
	\begin{aligned}
		\widetilde \Gamma_j(x,y)
		&= \Gamma_j(\sqrt L)(x,y) + \int_{X} S_j(x,z)\Gamma_j(\sqrt L)(z,y)\,d\mu(z)\\
		&= \Gamma_j(\sqrt L)(x,y) + \int_{X} R_j(x,z)\,\widetilde{\Gamma}_j(z,y)\,d\mu(z),
	\end{aligned}
	\]
	where we used $S_j = \sum_{k\ge 1} R_j^k= R_j \sum_{k\ge 0} R_j^k =R_jT_j$.  
	
	This, together with (i) for $m=0$,  \eqref{eq1-kernel-of-Rj}, \eqref{eq2-kernel-of-Rj}, and Lemmas \ref{lem-exp kernel} and \ref{lem-product exp}, proves (i) and (ii).
	
	\medskip
	It remains to prove (iii). To do this, we compute
	\[
	\begin{aligned}
		L^m\big[\widetilde \Gamma_j(\cdot,y)\widetilde \Gamma_j(\cdot,w)\big](x)
		&= L^m\big[\Gamma_j(\sqrt L)(\cdot,y)\Gamma_j(\sqrt L)(\cdot,w)\big](x)\\
		&\quad + \int_{X}L^m\big[\Gamma_j(\sqrt L)(\cdot,y)R_j(\cdot,z)\big](x)\,\widetilde{\Gamma}_j(z,w)\,d\mu(z)\\
		&\quad + \int_{X}L^m\big[\Gamma_j(\sqrt L)(\cdot,w)R_j(\cdot,z)\big](x)\,\widetilde{\Gamma}_j(z,y)\,d\mu(z)\\
		&\quad + \int_{X}\int_{X} L^m\big[R_j(\cdot,z_1)R_j(\cdot,z_2)\big](x)\,\widetilde{\Gamma}_j(z_1,y)\widetilde{\Gamma}_j(z_2,w)\,d\mu(z_1)d\mu(z_2)\\
		&=: E_1+E_2+E_3+E_4.
	\end{aligned}
	\]
	
	By condition (A3),
	\[
	E_1 \lesssim 2^{2jm}D_{2^{-j},N}(x,y)D_{2^{-j},N}(x,w).
	\]
	
	The estimates of $E_2,E_3$ are similar and simpler, so we only detail $E_4$.  Recall that 
	\[
	R_j(x,y) =  \Gamma^2_j(\sqrt L)(x,y) -\sum_{Q\in \mathcal D_j} \omega_Q \Gamma_j(\sqrt L)(x,x_Q)\Gamma_j(\sqrt L)(x_Q,y),
	\]
	which implies that 
	\[
	\begin{aligned}
		L^m&\big[R_j(\cdot,y)R_j(\cdot,z)\big](x)\\
		= \ &  L^m\big[\Gamma^2_j(\sqrt L)(\cdot,y)\Gamma^2_j(\sqrt L)(\cdot,z)\big](x) \\
		&-\sum_{Q\in \mathcal D_j} \omega_Q L^m\big[\Gamma_j(\sqrt L)(\cdot,x_Q)\Gamma^2_j(\sqrt L)(\cdot,z)\big](x)\Gamma_j(\sqrt L)(x_Q,y)\\
		&-\sum_{Q\in \mathcal D_j} \omega_Q L^m\big[\Gamma_j(\sqrt L)(\cdot,x_Q)\Gamma^2_j(\sqrt L)(\cdot,y)\big](x)\Gamma_j(\sqrt L)(x_Q,z)\\
		&+\sum_{Q,Q'\in \mathcal D_j} \omega_Q \omega_{Q'} L^m\big[\Gamma_j(\sqrt L)(\cdot,x_Q)\Gamma_j(\sqrt L)(\cdot,x_{Q'})\big](x)\Gamma_j(\sqrt L)(x_Q,y)\Gamma_j(\sqrt L)(x_{Q'},z)\\
		=:&F_1+F_2+F_3+F_4.
	\end{aligned}
	\]
	We now estimate $F_4$. For $M=N+2n$, by (A3) and Theorem \ref{thm-kernel estimate for functional calculus},
	\[
	\begin{aligned}
		F_4&\lesi 2^{2mj}\sum_{Q,Q'\in \mathcal D_j} \omega_Q \omega_{Q'} D_{2^{-j},M}(x,x_Q)D_{2^{-j},M}(x,x_{Q'})D_{2^{-j},M}(x_Q,y)D_{2^{-j},M}(x_{Q'},z)\\
		&\simeq 2^{2mj} \Big[\sum_{Q\in \mathcal D_j}\omega_Q D_{2^{-j},\sigma}(x,x_Q)D_{2^{-j},M}(x_Q,y)\Big]\Big[\sum_{Q'\in \mathcal D_j}\omega_{Q'}D_{2^{-j},M}(x,x_{Q'})D_{2^{-j},M}(x_{Q'},z)\Big]\\
		&\simeq 2^{2mj} \Big[\sum_{Q\in \mathcal D_j}\int_Q D_{2^{-j},\sigma}(x,w)D_{2^{-j},M}(w,y)d\mu(w)\Big]\Big[\sum_{Q'\in \mathcal D_j}\int_{Q'}D_{2^{-j},M}(x,w)D_{2^{-j},M}(w,z)d\mu(w)\Big]\\
		&\simeq 2^{2mj} \Big[ \int_X D_{2^{-j},M}(x,w)D_{2^{-j},M}(w,y)d\mu(w)\Big]\Big[ \int_{X}D_{2^{-j},M}(x,w)D_{2^{-j},M}(w,z)d\mu(w)\Big]\\
		&\lesi 2^{2mj}D_{2^{-j},M}(x,y)D_{2^{-j},M}(x,z),
	\end{aligned}
	\]
	where in the last inequality we used Lemma \ref{lem-Ds Dt}.
	
	Similarly, 
	\[
	F_1+F_2+F_3\lesi 2^{2mj}D_{2^{-j},N+n}(x,y)D_{2^{-j},M}(x,z).
	\]
	Therefore,
	\[
	L^m\big[R_j(\cdot,y)R_j(\cdot,z)\big](x)\lesi 2^{2mj}D_{2^{-j},N+n}(x,y)D_{2^{-j},M}(x,z).
	\]
	This, together with Lemma \ref{lem-Ds Dt}, implies that 
	\[
	\begin{aligned}
		E_4&\lesi 2^{2mj}\int_{X}\int_{X}D_{2^{-j},N}(x,z_1)D_{2^{-j},N}(x,z_2)  D_{2^{-j},N}(z_1,y)D_{2^{-j},N}(z_2,w)d\mu(z_1)d\mu(z_2)\\
		&\lesi 2^{2mj}D_{2^{-j},N}(x,y)D_{2^{-j},N}(x,w).
	\end{aligned}
	\]
	
	This completes the proof of (iii).
\end{proof}

We now give the proof of the inequality \eqref{eq- equivalence of maximal functions}. In fact, the inequality \eqref{eq- equivalence of maximal functions} is a direct consequence of the following proposition.
\begin{prop}\label{prop-maximal function equivalence}
	Let $L$ satisfy the condition (A1)  and let $w\in A_\vc$, $r\in (0,1]$, $k, n_2\in \mathbb Z$ and $\lambda>0$. Then we have
	\[
	\mathcal M_{r,w}\Big(  \Phi_{k,n_2,\lambda}^*(\sqrt{L})g \Big)(x)\simeq \mathcal M_{r,w}\Big(  \Phi_{k,n_2,\lambda}^*(\sqrt{L})g \Big)(y)
	\]
	for $x,y\in Q\in \mathcal D_k$.
\end{prop}

\begin{proof}
	It suffices to prove the claim for $r=1$ and $w\equiv 1$, since the case $r\in(0,1)$ and $w\in A_\infty$ follows by the same argument with minor modifications. By definition,
	\[
	\mathcal M\big(\Phi_{k,n_2,\lambda}^*(\sqrt{L})g\big)(x)
	= \sup_{B\ni x} \fint_B \Phi_{k,n_2,\lambda}^*(\sqrt{L})g(z)\,d\mu(z).
	\]
	Split the supremum according to the radius of the ball:
	\[
	\begin{aligned}
		\mathcal M\big(\Phi_{k,n_2,\lambda}^*(\sqrt{L})g\big)(x)
		&\le \sup_{\substack{B\ni x\\ r_B\ge 4\ell(Q)}} \fint_B \Phi_{k,n_2,\lambda}^*(\sqrt{L})g(z)\,d\mu(z)
		\\ &\quad
		+ \sup_{\substack{B\ni x\\ r_B< 4\ell(Q)}} \fint_B \Phi_{k,n_2,\lambda}^*(\sqrt{L})g(z)\,d\mu(z).
	\end{aligned}
	\]
	
	If $B\ni x$ and $r_B\ge 4\ell(Q)$, then for every $y\in Q$ we have $y\in B$. Hence for any fixed $y\in Q$,
	\[
	\fint_B \Phi_{k,n_2,\lambda}^*(\sqrt{L})g(z)\,d\mu(z)
	\le \mathcal M\big(\Phi_{k,n_2,\lambda}^*(\sqrt{L})g\big)(y),
	\]
	and therefore
	\[
	\sup_{\substack{B\ni x\\ r_B\ge 4\ell(Q)}} \fint_B \Phi_{k,n_2,\lambda}^*(\sqrt{L})g(z)\,d\mu(z)
	\le \mathcal M\big(\Phi_{k,n_2,\lambda}^*(\sqrt{L})g\big)(y).
	\]
	
	If $B\ni x$ and $r_B<4\ell(Q)$, then for any $z\in B$ and any $y\in Q$ (with $Q$ containing $x$) we have 
	\[
	\Phi_{k,n_2,\lambda}^*(\sqrt{L})g(z)
	\lesssim \Phi_{k,n_2,\lambda}^*(\sqrt{L})g(y).
	\]
	Therefore
	\[
	\begin{aligned}
		\sup_{\substack{B\ni x\\ r_B<4\ell(Q)}} \fint_B \Phi_{k,n_2,\lambda}^*(\sqrt{L})g(z)\,d\mu(z)
		&\lesssim \Phi_{k,n_2,\lambda}^*(\sqrt{L})g(y)
		\\&\le \mathcal M\big(\Phi_{k,n_2,\lambda}^*(\sqrt{L})g\big)(y).
	\end{aligned}
	\]
	
	Combining the two cases yields
	\[
	\mathcal M\big(\Phi_{k,n_2,\lambda}^*(\sqrt{L})g\big)(x)
	\lesssim \mathcal M\big(\Phi_{k,n_2,\lambda}^*(\sqrt{L})g\big)(y),
	\]
	for any fixed $y\in Q$, as required. This completes the proof.
	
\end{proof}

{\bf Acknowledgement.} The author was supported by the research grant ARC DP140100649 from the Australian Research Council.


\begin{thebibliography}{99}
	\bibitem{A} D. G. Aronson, \textit{Bounds for the fundamental solution of a parabolic equation}, Bull. Amer. Math. Soc. \textbf{73} (1967), 890-896.
	
	
	\bibitem{AR} P. Auscher and E. Russ,  \textit{Hardy spaces and divergence operators on strongly Lipschitz domains of $\mathbb R^n$}, J. Funct. Anal. \textbf{201} (2003), no. 1, 148–184.
	
	
	\bibitem{BBR} N. Badr, F. Bernicot, E. Russ, \textit{Algebra properties for Sobolev spaces – applications to semilinear PDEs on manifolds}, J. Anal. Math. \textbf{118} (2) (2012) 509–544.	

	\bibitem{BPT} H.-Q. Bui, M. Palusz\'ynski and M. H. Taibleson, \textit{A maximal function characterization of weighted Besov-	Lipschitz and Triebel-Lizorkin spaces}, Studia Math. \textbf{119} (1996), 219-246.
	
	\bibitem{BBD} H.-Q. Bui, T. A. Bui, and X. T. Duong, \emph{Weighted Besov and Triebel-Lizorkin spaces associated with operators and applications}, Forum Math. Sigma \textbf{8} (2020), Paper No. e11, 95 pp.
	
	\bibitem{BD} T. A. Bui and X. T. Duong, \emph{Higher order Riesz transforms of Hermite operators on new Besov and Triebel-Lizorkin spaces}, Constructive Approximation \textbf{53} (2021), 85--120.
	
	\bibitem{BD2} T. A. Bui and X. T. Duong, \emph{Spectral multipliers of self-adjoint operators on Besov and Triebel-Lizorkin spaces associated to operators}, Int. Math. Res. Not. IMRN \textbf{23} (2021), 18181--18224.
	
	\bibitem{BDK} T. A. Bui, X. T. Duong, and F. K. Ly, \emph{Maximal function characterizations for new local Hardy type spaces on spaces of homogeneous type}, Trans. Amer. Math. Soc. \textbf{370} (2018), 7229--7292.
	
	\bibitem{BHH} T. A. Bui, Q. Hong, and G. Hu, \emph{Heat kernels and theory of Hardy spaces associated to Schr\"odinger operators on stratified groups}, J. Differential Equations \textbf{353} (2023), 147--224.
	
	%\bibitem{BM} T. A. Bui and S. Mukherjee, \textit{Dunkl paraproducts and fractional Leibniz rules for the Dunkl Laplacian}. Available at: https://arxiv.org/abs/2507.10042.
	
	
	%\bibitem{Benyi} Á. Bényi, \emph{Bilinear pseudodifferential operators with forbidden symbols on Lipschitz and Besov spaces}, J. Math. Anal. Appl. \textbf{284} (2003), 97--103.
	
	\bibitem{BenyiTorres} Á. Bényi and R. H. Torres, \emph{Symbolic calculus and the transposes of bilinear pseudodifferential operators}, Comm. Partial Differential Equations \textbf{28} (2003), 1161--1181.
	
	%\bibitem{BenyiNahmodTorres} Á. Bényi, A. Nahmod, and R. H. Torres, \emph{Sobolev space estimates and symbolic calculus for bilinear pseudodifferential operators}, J. Geom. Anal. \textbf{16} (2006), 431--453.
	
	\bibitem{BourgainLi} J. Bourgain and D. Li, \emph{On an endpoint Kato–Ponce inequality}, Differential Integral Equations \textbf{27} (2014), 1037--1072.
	
	\bibitem{BrummerNaibo} J. Brummer and V. Naibo, \emph{Bilinear operators with homogeneous symbols, smooth molecules, and Kato–Ponce inequalities}, Proc. Amer. Math. Soc. \textbf{146} (2018), 1217--1230.
	
	%\bibitem{BrummerNaiboWeighted} J. Brummer and V. Naibo, \emph{Weighted fractional Leibniz-type rules for bilinear multiplier operators}, Potential Anal.  51 (2019), 71–99.
	

	\bibitem{Bruno} T. Bruno, \textit{Homogeneous algebras via heat kernel estimates}, Trans. Amer. Math. Soc.  \textbf{375} (2022), pp. 6903-6946
		
	\bibitem{ChristWeinstein} M. Christ and M. Weinstein, \emph{Dispersion of small amplitude solutions of the generalized Korteweg–de Vries equation}, J. Funct. Anal. \textbf{100} (1991), 87--109.
	
	\bibitem{C} M. Christ, \emph{A $Tb$ theorem with remarks on analytic capacity and the Cauchy integral}, Colloq. Math. \textbf{61} (1990), 601--628.

	\bibitem{CM} R. Coifman and Y. Meyer, \textit{Au del\`a des op\'erateurs pseudo-diff\'erentiels}, volume 57 of Ast\'erisque. Soci\'et\'e Math\'ematique de France, Paris, 1978.
	
	\bibitem{CW} R. R. Coifman and G. Weiss, \textit{Extensions of Hardy spaces and their use in analysis}, Bull. Amer. Math. Soc. \textbf{83} (1977), 569--645.
	
	\bibitem{CRT} T. Coulhon, E. Russ and V. Tardivel-Nachef, \textit{Sobolev algebras on Lie groups and Riemannian manifolds}, Amer. J. of Math. \textbf{123} (2001), 283–342.
	
	%\bibitem{CruzUribeNaibo} D. Cruz-Uribe and V. Naibo, \emph{Kato–Ponce inequalities on weighted and variable Lebesgue spaces}, Differential Integral Equations \textbf{29} (2016), 801--836.
	
	\bibitem{CR} D. Cruz-Uribe and C. Rios, \textit{Gaussian bounds for degenerate parabolic equations}, J. Funct. Anal. \textbf{255} (2008),
	283–312.
	
	%\bibitem{CMS} R. R. Coifman, Y. Meyer, and E. M. Stein, \emph{Some new functions and their applications to harmonic analysis}, J. Funct. Anal. \textbf{62} (1985), 304--315.
	
	%\bibitem{CD} T. Coulhon and X. T. Duong, \emph{Maximal regularity and kernel bounds: observations on a theorem by Hieber and Pr\"uss}, Adv. Differential Equations \textbf{5} (2000), 343--368.
		\bibitem{DY1} X.T. Duong and L.X. Yan, \textit{New function spaces of BMO type, the
	John--Nirenberg inequality, interpolation and applications}, Comm.
	Pure Appl. Math. \textbf{58} (2005), 1375--420.
	
	
	\bibitem{DY} X. T. Duong and L. Yan, \emph{Duality of Hardy and BMO spaces associated with operators with heat kernel bounds}, J. Amer. Math. Soc. \textbf{18} (2005), 943--973.
	
	\bibitem{DJ} J. Dziuba\'nski, K. Jotsaroop, \textit{On Hardy and BMO spaces for Grushin operator}, J. Fourier Anal. Appl. \textbf{22} (2016), no. 4, 954–995.
	
	\bibitem{DPW} J. Dziubański, M. Preisner and B. Wr\'obel, \textit{Multivariate H\"ormander-type multiplier theorem for the Hankel transform}, J. Fourier Anal. Appl. \textbf{19} (2013), no. 2, 417–437.
	
	\bibitem{Fang et al} J. Fang, H. Li, and J. Zhao, \textit{Fractional Leibniz rules for bilinear spectral multipliers on Carnot groups}, ollect. Math. (2025). https://doi.org/10.1007/s13348-025-00488-6
	
	\bibitem{F} C. Fefferman, \emph{The uncertainty principle}, Bull. Amer. Math. Soc. \textbf{9} (1983), 129--206.
	
	\bibitem{FS} G. B. Folland and E. M. Stein, \emph{Hardy Spaces on Homogeneous Groups}, Princeton Univ. Press, Princeton, 1982.
	
	\bibitem{G} Y. Guivarc’h, \emph{Croissance polyn\^omiale et p\'eriodes des fonctions harmoniques}, Bull. Soc. Math. France \textbf{101} (1973), 333--379.
	
	\bibitem{G.etal} A. G. Georgiadis, G. Kerkyacharian, G. Kyriazis, and P. Petrushev, \emph{Homogeneous Besov and Triebel–Lizorkin spaces associated to non-negative self-adjoint operators}, J. Math. Anal. Appl. \textbf{449} (2017), 1382--1412.
	
	\bibitem{GLY} L. Grafakos, L. Liu, and D. Yang, \emph{Vector-valued singular integrals and maximal functions on spaces of homogeneous type}, Math. Scand. \textbf{104} (2009), 296--310.
	
	%\bibitem{GrafakosMaldonadoNaibo} L. Grafakos, D. Maldonado, and V. Naibo, \emph{A remark on an endpoint Kato–Ponce inequality}, Differential Integral Equations \textbf{27} (2014), 415--424.
	
	\bibitem{GrafakosOh} L. Grafakos and S. Oh, \emph{The Kato–Ponce inequality}, Comm. Partial Differential Equations \textbf{39} (2014), 1128--1157.
	
	\bibitem{GulisashviliKon} A. Gulisashvili and M. Kon, \emph{Exact smoothing properties of Schrödinger semigroups}, Amer. J. Math. \textbf{118} (1996), 1215--1248.
	
	\bibitem{HMY} Y. S. Han, D. M\"uller and D. Yang, \textit{A theory of Besov and Triebel--Lizorkin spaces on metric measure spaces modeled on Carnot--Carath\'eodory spaces}, Abstr. Appl. Anal. 2008, Art. ID 893409, 250 pp. 
	
	
	\bibitem{HS} Y. S. Han and E. T. Sawyer, \textit{Littlewood-Paley theory on spaces of homogeneous type and the classical function spaces}, Mem. Amer. Math. Soc. \textbf{110} (1994), no. 530, vi+126 pp.
	
	
	\bibitem{HartTorresWu} J. Hart, R. H. Torres, and X. Wu, \emph{Smoothing properties of bilinear operators and Leibniz-type rules in Lebesgue and mixed Lebesgue spaces}, Trans. Amer. Math. Soc. \textbf{370} (2018), 8581--8612.
	
	\bibitem{HLMMY} S. Hofmann, G. Lu, D. Mitrea, M. Mitrea, and L. Yan, \emph{Hardy spaces associated to non-negative self-adjoint operators satisfying Davies-Gaffney estimates}, Mem. Amer. Math. Soc. \textbf{214} (2011).
	
	\bibitem{Hu} G. Hu, \emph{Homogeneous Triebel-Lizorkin spaces on stratified Lie groups}, J. Funct. Spaces Appl. (2013), Article ID 475103, 16 pp.
	
%	\bibitem{JTW} S. Janson, M. H. Taibleson, and G. Weiss, \emph{Elementary characterizations of the Morrey--Campanato spaces}, Lecture Notes in Math., 992 (1983), 101--114.
	
	\bibitem{JY} R. Jiang and D. Yang, \emph{Orlicz-Hardy spaces associated with operators satisfying Davies-Gaffney estimates}, Commun. Contemp. Math. \textbf{13} (2011), 331--373.
	
	\bibitem{KatoPonce} T. Kato and G. Ponce, \emph{Commutator estimates and the Euler and Navier–Stokes equations}, Comm. Pure Appl. Math. \textbf{41} (1988), 891--907.
	
	\bibitem{KenigPonceVega} C. Kenig, G. Ponce, and L. Vega, \emph{Well-posedness and scattering results for the generalized Korteweg–de Vries equation via the contraction principle}, Comm. Pure Appl. Math. \textbf{46} (1993), 527--620.
	
%	\bibitem{KoezukaTomita} K. Koezuka and N. Tomita, \emph{Bilinear pseudodifferential operators with symbols in $BS_{1,1}^m$ on Triebel–Lizorkin spaces}, J. Fourier Anal. Appl. \textbf{24} (2018), 309--319.
	\bibitem{LZ} L. Liu and Y. Zhang, \textit{Fractional Leibniz-type rules on spaces of homogeneous type}, Potential Analysis \textbf{60} (2024), 555-595.
	
	
	\bibitem{MuscaluSchlag} C. Muscalu and W. Schlag, \emph{Classical and multilinear harmonic analysis. Vol. II}, Cambridge Studies in Advanced Mathematics, 138, Cambridge Univ. Press, 2013.
	
	\bibitem{NaiboLy} V. Naibo and F. K. Ly, \emph{Fractional Leibniz rules associated to bilinear Hermite pseudo-multipliers}, Int. Math. Res. Not. IMRN \textbf{7} (2023), 5401-5437.
	
	\bibitem{NaiboThomson} V. Naibo and A. Thomson, \emph{Coifman-Meyer multipliers: Leibniz-type rules and applications to scattering of solutions to PDEs}, Trans. Amer. Math. Soc. \textbf{372} (2019), 5453-5481.
	
	\bibitem{NSW} A. Nagel, E. M. Stein, and S. Wainger, \emph{Balls and metrics defined by vector fields I: Basic properties}, Acta Math. \textbf{155} (1985), 103--147.
	
	\bibitem{Nash} J. Nash, \textit{Continuity of solutions of parabolic and elliptic equations}, Amer. J. Math. \textbf{80} (1958), 931–954.
	
	
	\bibitem{PK} G. Kerkyacharian and P. Petrushev, \emph{Heat kernel based decomposition of spaces of distributions in the framework of Dirichlet spaces}, Trans. Amer. Math. Soc. \textbf{367} (2015), 121--189.
	
	\bibitem{Shen} Z. Shen, \emph{$L^p$ estimates for Schr\"{o}dinger operators with certain potentials}, Ann. Inst. Fourier (Grenoble) \textbf{45} (1995), 513--546.
	
	\bibitem{RS} D. W. Robinson, A. Sikora, \textit{Analysis of degenerate elliptic operators of Gru$\check{s}$in type}, Math. Z. \textbf{260} (2008), no. 3, 475–508.
	
	\bibitem{RS2} D. W. Robinson, A. Sikora, \textit{The limitations of the Poincar\'e inequality for Gru$\check{s}$in type operators}, J. Evol. Equ. \textbf{14} (2014), no. 3, 535–563.
	
	\bibitem{Sa} L. Saloff-Coste, \textit{A note on Poincar\'e, Sobolev and Harnack inequalities}, Duke J. Math. \textbf{65} (1992) 27–38.
	
	\bibitem{Sikora} A. Sikora, \emph{Multivariable spectral multipliers and analysis of quasielliptic operators on fractals}, Indiana Univ. Math. J. \textbf{58} (2009), 317--334.
	
	\bibitem{ST} K. Stempak and J. L. Torrea, \emph{Poisson integrals and Riesz transforms for Hermite function expansions with weights}, J. Funct. Anal. \textbf{202} (2003), 443--472.
	
	%\bibitem{SY} L. Song and L. Yan, \emph{Maximal function characterizations for Hardy spaces associated to nonnegative self-adjoint operators on spaces of homogeneous type}, to appear in J. Evol. Equ.
	
	\bibitem{Th} S. Thangavelu, \emph{Lectures on Hermite and Laguerre Expansions}, Math. Notes, vol. 42, Princeton Univ. Press, 1993.
	
	\bibitem{VSC} N. Varopoulos, L. Saloff-Coste, and T. Coulhon, \emph{Analysis and geometry on groups}, Cambridge Univ. Press, 1992.
	
	%\bibitem{Yan} L. Yan, \emph{Classes of Hardy spaces associated with operators, duality theorem and applications}, Trans. Amer. Math. Soc. \textbf{360} (2008), 4383--4408.
	
	\bibitem{YY} S. Yang and D. Yang, \emph{Atomic and maximal function characterizations of Musielak–Orlicz–Hardy spaces associated to non-negative self-adjoint operators on spaces of homogeneous type}, Collect. Math. \textbf{70} (2019), 197--246.
	
	\bibitem{YZ} D. Yang and Y. Zhou, \emph{Localized Hardy spaces $H^1$ related to admissible functions on RD-spaces and applications to Schr\"odinger operators}, Trans. Amer. Math. Soc. \textbf{363} (2011), 1197--1239.
	
	%\bibitem{YZ2} D. Yang and Y. Zhou, \emph{Radial maximal function characterizations of Hardy spaces on RD-spaces and their applications}, Math. Ann. \textbf{346} (2010), 307--333.
	
	%\bibitem{YYZ} D. Yang, D. Yang, and Y. Zhou, \emph{Localized Morrey-Campanato spaces on metric measure spaces and applications to Schr\"odinger operators}, Nagoya Math. J. \textbf{198} (2010), 77--119.
	
	%\bibitem{YYZ2} D. Yang, D. Yang, and Y. Zhou, \emph{Localized BMO and BLO spaces on RD-spaces and applications to Schr\"odinger operators}, Commun. Pure Appl. Anal. \textbf{9} (2010), 779--812.
	
\end{thebibliography}
\end{document}